\documentclass[12pt,a4paper]{amsart}

\usepackage[english]{babel}
\usepackage[T1]{fontenc}
\usepackage[utf8x]{inputenc}
\usepackage{lmodern}
\usepackage{graphicx}
\usepackage{amssymb}
\usepackage{amsmath}
\usepackage{amsfonts}
\usepackage{amsrefs}
\usepackage{amscd}
\usepackage{amsthm}
\usepackage{hyperref}
\usepackage[dvipsnames]{xcolor}
\usepackage{mathrsfs}
\usepackage{enumerate}
\usepackage{mathtools}
\usepackage{xy}
\usepackage{ucs}
\usepackage{fancyhdr}
\usepackage{relsize}
\usepackage{tabularx}
\usepackage{textcomp}
\usepackage[multiple]{footmisc}
\usepackage[shortlabels]{enumitem}

\usepackage[margin=2.5cm]{geometry}

\input{xy}
\xyoption{all}

\theoremstyle{plain}
\newtheorem{theorem}{Theorem}[section]
\newtheorem{lemma}[theorem]{Lemma}
\newtheorem{prop}[theorem]{Proposition}
\newtheorem{cor}[theorem]{Corollary}

\theoremstyle{definition}

\newtheorem{fact}[theorem]{Fact}
\newtheorem{example}[theorem]{Example}
\newtheorem*{ack}{Acknowledgements}

\theoremstyle{remark}
\newtheorem{remark}[theorem]{Remark}
\newtheorem{notation}[theorem]{Notation}

\newcommand{\Z}{\mathbf{Z}}
\newcommand{\R}{\mathbf{R}}
\newcommand{\C}{\mathbf{C}}
\newcommand{\Q}{\mathbf{Q}}

\newcommand{\F}{\mathbf{F}}

\newcommand{\expp}{\mathscr{E}xp\mathscr{M}}
\newcommand{\Var}{\mathrm{Var}}

\newcommand{\evar}{\mathrm{ExpVar}}

\newcommand{\LL}{\mathbf{L}}
\newcommand{\M}{\mathscr{M}}

\newcommand{\OO}{\mathcal{O}}
\DeclareMathOperator{\ord}{ord}
\renewcommand{\t}{\mathbf{t}}

\newcommand{\dx}{\mathrm{d}}

\newcommand{\A}{\mathbf{A}}
\newcommand{\Proj}{\mathbf{P}}

\newcommand{\spec}{\mathrm{Spec}\,}

\newcommand{\res}{\mathrm{res}}

\renewcommand{\ker}{\mathrm{Ker}}
\newcommand{\pt}{\mathrm{pt}}

\newcommand{\Pic}{\mathrm{Pic}}

\newcommand{\rk}{\mathrm{rk}}

\newcommand{\id}{\mathrm{id}}
\newcommand{\pr}{\mathrm{pr}}

\renewcommand{\bar}{\overline}
\newcommand{\ul}{\underline}
\DeclareMathOperator{\Span}{Span}

\newcommand{\x}{\mathbf{x}}
\newcommand{\y}{\mathbf{y}}
\renewcommand{\1}{\mathbf{1}}
\newcommand{\T}{\mathbf{T}}

\newcommand{\h}{\mathbf{h}}
\newcommand{\g}{\mathbf{g}}
\newcommand{\0}{\mathbf{0}}
\renewcommand{\u}{\mathbf{u}}

\renewcommand{\phi}{\varphi}
\newcommand{\G}{\mathbf{G}}
\newcommand{\eps}{\varepsilon}

\newcommand{\sfl}{\mathsf{\Lambda}}

\renewcommand{\a}{\mathfrak{a}}
\renewcommand{\b}{\mathfrak{b}}

\newcommand{\Mor}{\mathrm{Mor}}
\newcommand{\Spec}{\mathrm{Spec}}
\newcommand{\maj}{\mathrm{major}}
\newcommand{\minor}{\mathrm{minor}}
\newcommand{\Conf}{\mathrm{Conf}}

\newcommand{\Sym}{\mathrm{Sym}}
\newcommand{\Poly}{\mathrm{Poly}}
\newcommand{\MPoly}{\mathrm{MPoly}}

\newcommand{\MHM}{\mathrm{MHM}}
\newcommand{\mhm}{\mathrm{MHM}}
\newcommand{\HM}{\mathrm{HM}}
\newcommand{\Hdg}{\mathrm{Hdg}}
\newcommand{\VHS}{\mathrm{VHS}}

\newcommand{\mon}{\mathrm{mon}}
\newcommand{\gr}{\mathrm{Gr}}
\def\twtimes{{\kern .1em\mathop{\boxtimes}\limits^{\scriptscriptstyle{T}}\kern .1em}}
\newcommand{\tot}{\mathrm{tot}}
\newcommand{\supp}{\mathrm{supp}}

\numberwithin{equation}{section}

\begin{document}

\title{A motivic circle method}

\author{Margaret Bilu}
\address{Centre de Mathématiques Laurent Schwartz\\
Ecole Polytechnique\\
91128 Palaiseau Cedex\\France}

\email{margaret.bilu@polytechnique.edu}
\author{Tim Browning}
\address{IST Austria\\
Am Campus 1\\
3400 Klosterneuburg\\
Austria}
\email{tdb@ist.ac.at}


\begin{abstract}
The circle method has been successfully used over the last century to study  rational points on hypersurfaces. More recently, a version of the method over function fields, combined with spreading out techniques, has led to a range of results about moduli spaces of rational curves on hypersurfaces. In this paper a version of the circle method is implemented in the setting of the  Grothendieck ring of varieties. 
This allows us to approximate the classes of these moduli spaces directly, without relying on point counting, and leads to a deeper understanding of their geometry.

%
%
\end{abstract}

\date{\today}
\maketitle

\thispagestyle{empty}
\setcounter{tocdepth}{1}
\tableofcontents

\section{Introduction}

Let $k$ be a field and let $f\in k[x_1,\dots,x_n]$ be a non-singular homogeneous polynomial  of degree $d\geq 2$, defining a hypersurface 
$X\subset \A^n$. For global fields, the density of $k$-points on $X$ has been the object of intense study over the years.  When $k=\Q$, the Hardy--Littlewood circle method can be used to cast light on the limit
$$
\lim_{B\to \infty} B^{-(n-d)}\#\{\x\in \Z^n: |\x|\leq B, ~f(\x)=0\},
$$
where $\x=(x_1,\dots,x_n)$ and 
$|\x|=\max_{1\leq i\leq n}|x_i|.$
Thus, it follows from work of 
Birch \cite{birch} that this limit exists as a product of local densities, provided that $n>2^d(d-1)$.
Lee \cite{lee} has worked out the analogous statement when $k=\F_q(t)$, for a finite field $\F_q$ of characteristic $>d$. Thus, under the same assumption  $n>2^d(d-1)$, a similar statement is proved about 
the existence and nature of the limit
$$
\lim_{e\to \infty} q^{-e(n-d)}\#\{\g\in \F_q[t]^n: \deg (g_1),\dots,\deg(g_n)\leq e, ~f(\g)=0\},
$$
where $\g=(g_1,\dots,g_n)$. This paper is concerned with the field $k=\C(t)$ and   the geometry of the affine variety 
\begin{equation}\label{eq:brighton}
M_e=
\{\g\in \C[t]^n: \deg (g_1),\dots,\deg(g_n)\leq  e, ~f(\g)=0\},
\end{equation}
over $\C$,
as $e\to \infty$.
We have $(e+1)n$ indeterminates and $de+1$ equations, and so the expected dimension of $M_e$ is 
\begin{equation}\label{eq:mu}
\mu(e):=(e+1)n -(de +1) = e(n-d)+n-1.
\end{equation}
Typically, tools from algebraic geometry only allow one to prove that the expected dimension is actually  the true dimension when $X$ is {\em generic}.  
This follows from seminal  work of Harris, Roth and Starr \cite{HRS} when 
$d<n/2$, a  restriction that  has since been weakened to 
$d<2n/3$ by  Beheshti and Kumar \cite{BK} (assuming that $n\geq 23$),
and then to $d\leq n-3$ by Riedl and Yang \cite{RY}.
In the setting   $d=3$  of cubic hypersurfaces it is possible to obtain 
results for {\em all} smooth hypersurfaces in the family. 
Thus Coskun and Starr \cite{CS} have shown 
that $M_e$  is  irreducible  and of  dimension $\mu(e)$, 
for any smooth cubic hypersurface  $X\subset \A^n$, provided that   $n\geq  4$.
In the setting of 
smooth projective Fano varieties, the behaviour of  the invariants associated to 
 spaces like $M_e$ is explored from the perspective of  the {\em geometric Manin conjecture} by Lehmann and Tanimoto \cite{sho}. 

It turns out that the circle method can shed light on 
the geometry of $M_e$ when $d$ is small enough compared to $n$, for any  hypersurface $X \subset \A^n$ defined by a non-singular form. 
Thus, under the assumption 
$n> (5d-4)2^{d-1}$,  
it follows from \cite{BV} that 
$M_e$ is irreducible and of dimension $\mu(e)$. 
This was later refined in \cite[Thm.~1.2]{BSfree}, with the outcome that for 
$n>3\cdot 2^{d-1}(d-1)$, 
the variety $M_e$ is  a reduced and irreducible  complete intersection of dimension $\mu(e)$,
which is smooth outside a set of codimension at least $(n/2^{d-2}-6(d-1))(1+\lfloor \frac{e+1}{d-1}\rfloor)$.
The arguments in \cites{BV,BSfree} rely on {\em spreading out}, which translates the problem into a counting problem  over a finite field $\F_q$. An application of the  Lang--Weil estimates yields the desired conclusion via an application of the circle method over $\F_q(t)$.

This paper aims to tackle deeper questions about the geometry of $M_e$ without recourse to spreading out and a reduction to counting arguments. A step in this direction was taken recently in \cite{BS}, where a sheaf-theoretic version of the circle method is developed to study a variant of $M_e$, using tools from \'etale cohomology.  Our goal is to rework the circle method in a motivic setting and to try and understand the class 
of $M_e$ in the 
Grothendieck ring of varieties.

The {\em Grothendieck ring of varieties} $K_0(\Var_\C)$ over $\C$ is defined by generators and relations. The generators are $\C$-varieties $X$. 
The relations are given by 
$X-Y$, whenever $X$ and $Y$ are $\C$-varieties that are isomorphic over $\C$, and the scissor relation
$ X-Y-U$, 
whenever $X$ is a $\C$-variety, 
$Y$ a closed  subscheme of~$X$ and $U=X\setminus Y$ its open complement.
We  write~$[X]$ for the class in $K_0(\Var_\C)$ 
of $X$. The product $[X][Y] = [X\times_\C Y]$ endows $K_0(\Var_\C)$ with a ring structure. 
We denote by $\LL$ the class of $\A^1$ in $K_0(\Var_\C)$, which is  a zero divisor in the ring, by work of Borisov \cite{lev}.  Inverting $\LL$ gives the ring $\M_\C=K_0(\Var_\C)[\LL^{-1}]$.
We are interested in the class of $M_e$ in $\M_\C$. In fact, since the circle method involves working with exponential sums, in order to be able to set up a motivic version of it, we  need to work in the 
 {\em Grothendieck ring of varieties with exponentials} $K_0(\evar_\C)$ and its localisation 
 $\expp_\C=K_0(\evar_\C)[\LL^{-1}]$. The definition of 
 $K_0(\evar_\C)$  is given in 
Section \ref{s:def_groth}, together with some of its properties. In order to talk about convergence, as $e\to \infty$, we need to introduce a topology on $\expp_\C$, which refines the obvious filtration by dimension. In Section \ref{s:weight} we  carry out the construction of an appropriate  weight 
function
$$
w:\expp_\C \to \Z,
$$
coming from Hodge theory. We  denote by $\widehat{\M_{\C}}$  the completion of $\M_{\C}$ with respect to this weight topology.

Working in $\expp_\C$, our main result 
shows that the class of $M_e$ converges to the classes of local terms that resemble the main term in the circle method, 
as $e\to \infty$. 
Given $x\in \A^1$ and integer $N\geq 0$, 
our work involves the jet spaces
 \begin{equation}\label{eq:lambda}
    \Lambda_N(f,x) = \left\{
    \g\in \C[t]^n: \deg (g_1),\dots,\deg(g_n)<N, ~ f(\g)\in (t-x)^N\C[t]\right\},
  \end{equation}
parameterising  solutions of $f(\g) = 0$ modulo $(t-x)^N$, together with
\begin{equation}\label{eq:jet}\Lambda_N(f,\infty) = \left\{\g \in \C[t]^n: 
\begin{array}{l}
\deg(g_1),\dots,\deg(g_n)< N\\
f(g_1(t^{-1}),\dots,g_n(t^{-1})) \in t^{-N}\C[t^{-1}]
\end{array}
\right\}.
\end{equation}
Note that $\Lambda_N(f,\infty) $ can naturally be identified with $\Lambda_N(f,0)$. 
Bearing this notation in mind, we  prove the following result.

\begin{theorem}\label{theorem_naivespace} 
Let $f\in \C[x_1,\dots,x_n]$ be a non-singular homogeneous polynomial  of degree $d\geq 3$, defining a  hypersurface  $X\subset \A^n$.
Assume $n > 2^d(d-1)$ and let $e\geq 1$. Then 
$$
[M_e] = \LL^{\mu(e)}\left(\mathfrak{S}(f)  \cdot 
\lim_{N\to \infty }\LL^{-N(n-1)} [\Lambda_N(f,\infty)]
+ R_{e}\right)
$$
in $\widehat{\M_\C}$, 
where 
$$
\mathfrak{S}(f)=\prod_{x\in \A^1} 
\lim_{N\to \infty }\LL^{-N(n-1)} [\Lambda_N(f,x)]
$$
 is a motivic Euler product and 
$R_e$ is an error term satisfying
$$
w(R_e) \leq  4
  - \frac{n-2^{d}(d-1)}{2^{d-2}} \left(1+\left\lfloor \frac{e+1}{2d-2}\right\rfloor\right).
$$
\end{theorem}

The term $\mathfrak{S}(f)$ 
is the  \textit{motivic singular series} and  emerges as 
an infinite convergent sum, whose precise definition is given in 
\eqref{eq:SS}. In Sections \ref{s:SS1} and \ref{s:SS2} we  express 
 $\mathfrak{S}(f)$ as a motivic Euler product, using the construction of motivic Euler products found in   \cite[Section 3]{Bilu}. The relevant background facts about motivic Euler products is recalled in Section \ref{sect:motivic_euler_prods}.

\begin{remark}
In this paper the focus is on hypersurfaces of degree $d\geq 3$. However, 
on combining Propositions~\ref{prop:major} and \ref{prop:minor} in  \eqref{eq:plan}, 
it is also possible to deduce a statement for $d=2$, with the only difference being a weaker error term if $e\leq n/2-7$. 
\end{remark}

The proof of 
Theorem \ref{theorem_naivespace}  relies on the development of a motivic version of the Hardy--Littlewood circle method. The arguments parallel some of the sheaf-theoretic arguments in \cite{BS}, but are much more attuned to the steps taken in the classical  circle method over $\Q$ in \cite{birch}, and its incarnation over $\F_q(t)$ in \cite{lee}. The motivic circle method is developed over the course of several sections. In  Section
\ref{sec:circle} we lay the foundations of the method and in Section \ref{s:weight} we construct an appropriate weight function on relative Grothendieck rings of varieties with exponentials: this was one of the main difficulties we encountered, since  the weight function from \cite{Bilu} was not suitable for our purposes. In Section  \ref{sec:weyl} we work out a Weyl differencing argument to bound the weight of a general class of motivic exponential sums, using the properties of this new weight function.  In Sections \ref{sec:major} and 
\ref{sect:minor_arcs} we analyse the contribution from the major and minor arcs, 
respectively.  Our treatment of the minor arcs relies on tools from the geometry of numbers over the function field $\C(t)$, in the spirit of work by Mahler \cite{mahler}. We have collected together the necessary facts  in  Appendix~\ref{sec:gon}.

One notable feature of our work is that 
we work over $\C$ throughout  and never rely on the kind of  counting arguments over finite fields that arise from  spreading out.  In doing so, we believe that 
Theorem~\ref{theorem_naivespace}  has the potential to unlock even more geometric information about the  parameter space $M_e$, and its cousins. 
We explain some applications in the remainder of this introduction. 
\subsection*{The Hodge--Deligne polynomial}

As explained in 
\cite[Prop.~3.2.14 of Chapter 2]{CNS},
the Hodge--Deligne polynomial gives a unique motivic measure $\mathrm{HD}: K_0(\Var_\C) \to \Z[u,v]$, via
$$
\mathrm{HD}(X)= \sum_{p,q\geq 0} (-1)^{p+q} h^{p,q}(X)u^pv^q,
$$
where $h^{p,q}(X)$ are the virtual Hodge numbers of $X$. 
One has 
$\mathrm{HD}(\LL)=uv$. Moreover,  $\mathrm{HD}$ induces a motivic measure
$$\widehat{\M_{\C}}\to \Z[u,v][[(uv)^{-1}]].$$
Applying this motivic measure to both sides of the equality in Theorem 
\ref{theorem_naivespace}, it is possible to express a positive proportion of the coefficients of 
$\mathrm{HD}(M_e)$ in terms of  
the Hodge--Deligne polynomial of the motivic Euler product $\mathfrak{S}(f)$.
Following the categorification of inclusion--exclusion,  recently worked out by Das and Howe \cite{DH}, 
it would be interesting to discern whether this  could give access to homological stability results
 for the parameter space $M_e$.
 
While the  
use of spreading-out and counting procedures 
 usually  only gives  access to dimension and irreducibility results, which corresponds to the dominant term of the Hodge--Deligne polynomial, our motivic approach computes a positive proportion of the coefficients of this polynomial. We  take this point of view further in Corollary \ref{cor:coefficients} in the special case $e=1$.

\subsection*{The space of morphisms}

Let $f\in \C[x_1,\dots,x_n]$ be a non-singular homogeneous polynomial of degree $d$, and let $Z\subset \Proj^{n-1}$ be the hypersurface it defines. We can also study the space of morphisms of degree $e$ from $\Proj^1$ to $Z$, given by
$$
\Mor_{e}(\Proj^1,Z) = \left\{\g\in (\C[t]^n-\{0\})/\C^{\times}: 
\begin{array}{l}
\max \deg g_i = e, ~\gcd(g_1,\dots,g_n) = 1, \\
f(\g) = 0
\end{array}
\right\}.
$$
The restriction of  $(\C[t]^n-\{0\})/\C^{\times}$ to polynomials of degree $\leq e$ is  isomorphic to  $\Proj^{n(e+1)-1}$, and 
so the expected dimension of $\Mor_{e}(\Proj^1,Z)$ is
$
n(e+1) - 1 - (de+1) = \mu(e)-1,
$
because it is cut out by $de+1$ equations. 

Our next result is deduced from  Theorem \ref{theorem_naivespace} using inclusion--exclusion and  supplies a similar result for  
the space of morphisms $\Mor_{e}(\Proj^1,Z)$.

\begin{theorem}\label{theorem_morspace} 
Let $f\in \C[x_1,\dots,x_n]$ be a non-singular homogeneous polynomial  of degree $d\geq 3$, defining a  hypersurface  $Z\subset \Proj^{n-1}$.
Assume $n > 2^d(d-1)$ and let $e\geq 1$.
Then 
$$[\Mor_e(\Proj^1,Z)] = \frac{
\LL^{\mu(e) - 1}}
{1-\LL^{-1}}
\left( 
\prod_{v\in \Proj^1} c_v
+ S_e\right)
$$
in $\widehat{\M_\C}$,
where
$$
c_v=(1-\LL^{-1}) \frac{[Z]  }{\LL^{n-2}}
$$
and 
 $S_e$ is an error term satisfying
$$w(S_e)\leq 4 - \left(\frac{n-2^d(d-1)}{2^{d-1}(d-1)}\right)(e+1).
$$
\end{theorem}

\subsection*{Irreducibility and dimension of the space of morphisms}
We  describe the properties of the weight function $w:\expp_\C \to \Z$ in Section \ref{s:weight}.
One of its  key  properties is the identity   $w([V])=2\dim V$, for any $\C$-variety $V$, 
which is proved in Lemma \ref{lemma:weight_dimension}. In fact, by Lemma \ref{lemma:weight_cancellation},
we have  $$
w([V]-[W])\leq 2r-1,
$$ 
whenever  
$V,W$ are irreducible $\C$-varieties of equal dimension $r$.
Rather than apply the statement of Theorem \ref{theorem_morspace}, we  instead invoke  its proof and  the expression 
\eqref{eq:staging-post} in particular. 
Lemma \ref{lem:relation'} and Corollary
\ref{cor:dim} imply that $\Lambda_{N}(f,\infty)\subset \A^{Nn}$ is irreducible and of dimension $N(n-1)$. 
Hence  $\sigma_\infty(f)=1+E_1$ in \eqref{eq:staging-post}, 
with $w(E_1)\leq -1$.
Moreover, it follows from
Remark \ref{rem:weight-SS} that  $\mathfrak{S}(f) = 1 + E_2$ in 
$\expp_\C$,  where $w(E_2) <0$, and 
$$
\frac{(1-\LL^{-(n-d)})(1-\LL^{-(n-d) + 1})}{1-\LL^{-1}}=1+E_3,
$$
where $w(E_3)\leq -2$. 
Hence the following  result  is an easy consequence of 
\eqref{eq:staging-post} in the proof of 
Theorem \ref{theorem_morspace}, combined with \cite[II, Theorem 1.2]{kollar_rat_curves}, which ensures that all irreducible components of the space of morphisms have at least the expected dimension.

\begin{cor}\label{cor:irred}
Let $f\in \C[x_1,\dots,x_n]$ be a non-singular homogeneous polynomial  of degree $d\geq 3$, defining a  hypersurface  $Z\subset \Proj^{n-1}$.
Assume $n > 2^d(d-1)$
and let $e\geq 1$.
Then 
$\Mor_e(\Proj^1,Z)$ is an irreducible variety of dimension $\mu(e)-1$.
\end{cor}

This refines \cite[Thm.~1.1]{BSfree}, which achieves the same conclusion under the more stringent assumption $n>2^{d}\left(d-\frac{1}{2}\right)$.

\subsection*{Motivic Manin--Peyre}
Our work allows us to address a question of Peyre \cite[Question 5.4]{peyre}
about the convergence of 
$$
[\Mor_e(\Proj^1,Z)]\LL^{-e(n-d)},
$$
for a  smooth hypersurface $Z\subset \mathbf{P}^{n-1}$ of degree $d$ that is  defined over $\C$. 
If  $n > 2^d(d-1)$, then  Theorem \ref{theorem_morspace} shows that this sequence
converges in the weight topology to 
$$
\frac{\LL^{n-2}}{1-\LL^{-1}}
\prod_{v\in \Proj^1} c_v,
$$
where
$$
c_v=(1-\LL^{-1}) \frac{[Z]  }{\LL^{n-2}}.
$$
This is consistent with a motivic analogue of the Manin--Peyre conjecture described by Faisant \cite[p.~3]{faisant}, who suggests that the limit can be interpreted as an adelic volume
$$
\frac{\LL^{\dim(Z)}}{(1-\LL^{-1})^{\rk \Pic(Z)}} \prod_{v\in \Proj^1}\mathfrak{c}_v,
$$
where $
\mathfrak{c}_v = (1-\LL^{-1})^{\rk \Pic(Z)} [Z] \LL^{-\dim Z},
$
at all but finitely many places. 
Indeed, we have
$\dim(Z)=n-2$ and 
 $\rk \Pic(Z) = 1$, 
by the Lefschetz hyperplane theorem.

\subsection*{Weak approximation and equidistribution}

Given a smooth hypersurface $Z\subset \mathbf{P}^{n-1}$
and points $p_1,\dots,p_m\in \Proj^1$ and $x_1,\dots,x_m\in Z$, 
a  natural generalisation involves 
studying the space 
$\Mor_{e}(\Proj^1,Z;p_1,\dots,p_m;x_1,\dots,x_m)$ of degree $e$ morphisms  $g:\Proj^1\to Z$, such that $g(p_i)=x_i$ for $1\leq i\leq m$.
Assuming  that 
 $n>2^d(d-1)$ and 
$e$ is large enough in terms of $m$, 
the methods of this paper are  robust enough to study the class of this variety in 
 $\widehat{\M_\C}$, which  would yield an analogue of 
Corollary~\ref{cor:irred} for $\Mor_{e}(\Proj^1,Z;p_1,\dots,p_m;x_1,\dots,x_m)$.
This would address a question raised by Will Sawin in his  lecture at the Banff workshop
``Geometry via Arithmetic'' in July 2021, who highlighted that methods from algebraic geometry 
break down when $m$ is allowed to be arbitrary, 
even for generic hypersurfaces. 

Note that the study of  $\Mor_{e}(\Proj^1,Z;p_1,\dots,p_m;x_1,\dots,x_m)$ can also  be interpreted as a special case of the equidistribution principle, as described in \cite{faisant}. A related question is that of the
 weak approximation problem  over $\C(t)$, a topic that was studied  by 
 Hassett and Tschinkel
\cite{HT,HT2}.

\subsection*{The variety of lines on hypersurfaces} Our method is robust enough to give non-trivial results already in the case $e=1$ of lines. Let $Z\subset \Proj^{n-1}$ as above be a smooth hypersurface of degree $d$, and let $F_1(Z)$ be the Fano variety of lines associated to $Z$. We  prove the following result in Section~\ref{sec:lines}.

\begin{theorem}\label{thm:lines}
Let $Z\subset \mathbf{P}^{n-1}$ be a smooth hypersurface of degree $d\geq 3$ with $n>2^d(d-1)$. Then 
\begin{align*}
[F_1(Z)] 
&= 
 \frac{\LL^{-d} [Z]^2-
 \LL^{n-d-2} [Z]}{1 + \LL^{-1}} 
+\widehat{R}_{n,d}
\end{align*}
in $\widehat{\M_\C}$, where
$$
w(\widehat{R}_{n,d})\leq 4n-2d-6- \frac{n-2^{d}(d-1)}{2^{d-2}}.
$$
\end{theorem}

As a consequence, in Section \ref{sec:lines} we may compute a positive proportion of the coefficients of the Hodge--Deligne polynomial of $F_1(Z)$.

\begin{cor}\label{cor:coefficients}
 Let $Z\subset \Proj^{n-1}$ be a smooth hypersurface of degree $d\geq 3$, with $n> 2^{d}(d-1)$. Let $p,q\in \Z$ such that $p,q\leq 2n-5-d$ and
$$p+q > 4n - 2d - 6 - \frac{n-2^d(d-1)}{2^{d-2}}.$$
Then
$$
h^{p,q}(F_1(Z)) = 
\begin{cases}
 \left\lfloor \frac{2n-d-5-p}{2}\right \rfloor + 1  &\text{ if}\  p=q,\\
0 &\text{ otherwise.}
\end{cases}
$$
\end{cor}

The fact that $F_1(Z)$ is irreducible and of the expected dimension in this case was already known by \cite{beheshti-riedl} in the much less restrictive range $n\geq 2d$, and is an instance of the Debarre--de Jong conjecture. On the other hand, we have not been able to find the computation of these Hodge coefficients in the literature in general. In related work of 
Debarre and Manivel 
\cite[Thm.~3.4]{debarre-manivel}, which is conditional on the assumption that  $F_1(Z)$ is  smooth and of the expected dimension, a Lefschetz-type result for the inclusion of $F_1(Z)$ into the space of all lines in $\Proj^{n-1}$ is proved. This allows one to compute the Hodge numbers of $F_1(Z)$ for $p+q< n-4$. However, the smoothness of $F_1(Z)$ is only known when $Z$ is general (as explained in \cite[Thm.~2.1(b)]{debarre-manivel}) or when $d = 3$.

Let us now discuss in a bit more detail the case $d=3$, that is, when $Z$ is a smooth cubic hypersurface.  In this case, $F_1(Z)$ is known to be a smooth variety of  dimension $2n-8$ as soon as $n\geq 4$, by work of Altman and Kleiman \cite{altman}. Galkin and Schinder \cite{galkin} have established a relationship between 
 the classes $[Z]$ and $[F_1(Z)]$ in $K_0(\Var_\C)$:
 it follows from 
\cite[Thm.~5.1]{galkin} that 
\begin{equation}\label{eq:GS}
\LL^2[F_1(Z)]=[\Sym^2Z] - (1+\LL^{n-2}) [Z].
\end{equation}
In fact, in the cubic case, Corollary \ref{cor:coefficients} 
can be deduced from  \cite[Thm.~6.1]{galkin}. (See also 
Eq.~(4.21) in 
the book by Huybrechts \cite{huy}).

On 
taking  $d=3$ in  Theorem~\ref{thm:lines}, we obtain
\begin{align*}
\LL^2[F_1(Z)] 
&= 
\frac{\LL^{-1} [Z]^2-
 \LL^{n-3} [Z]}{1 + \LL^{-1}} 
+\widehat{R}_n,
\end{align*}
if $n\geq 17$, 
where
$
w(\widehat{R}_n)\leq \frac{7}{2}n.
$
Note that $w(\LL^2[F_1(Z)])=4+2(2n-8)=4n-12$, and so we do indeed get non-trivial information. Comparing it with 
\eqref{eq:GS}, and  assuming  $n\geq 17$, this  results in the expression
\begin{equation}\label{eq:pen}
[\Sym^2Z] =\frac{\LL^{-1} [Z]^2-
 \LL^{n-3} [Z]}{1 + \LL^{-1}} + \LL^{n-2} [Z]
+\widehat{R'}_n,
\end{equation}
where  $w(\widehat{R'}_n)\leq \frac{7}{2}n$.
In general there seems no reason to expect a relation between the class of 
$\Sym^2Z$ and the classes   $[Z]^2$ and $[Z]$ in the Grothendieck group, but \eqref{eq:pen} gives  a good approximation of the class of $\Sym^2Z$ 
in the weight topology.
In Remark \ref{rem:burillo} we use work of Burillo \cite{Burillo} to calculate part of the Hodge--Deligne polynomial of 
$\Sym^2Z$ and check that
Corollary \ref{cor:coefficients} is consistent with 
\eqref{eq:GS} when  $n\geq 17$.

It would be interesting to see whether our work could shed light on recent 
questions of Popov \cite{popov} around the class in $K_0(\Var_\C)$ 
of the moduli space that  parameterises twisted cubics in smooth cubic hypersurfaces $Z\subset \Proj^{n-1}$.

\bigskip

\begin{ack}
The authors are grateful to Yohan Brunebarbe, Tom Burel, Antoine Chambert-Loir, Loïs Faisant, Mirko Mauri and Will Sawin for useful comments. 
Thanks are also due to the anonymous referees for numerous helpful remarks. 
M.B.  received funding from the European Union's Horizon 2020 research and innovation programme under the Marie Sk\l{}odowska-Curie Grant agreement No. 893012.
T.B. was
supported by a FWF grant (DOI 10.55776/P36278)
and by a grant from the Institute
for Advanced Study School of Mathematics.
\end{ack}

\section{The Grothendieck ring of varieties with exponentials}\label{s:def_groth}

In this section we define the 
Grothendieck ring of varieties with exponentials, relative to an arbitrary noetherian scheme $S$. The {\em Grothendieck group of varieties with exponentials} $K_0(\evar_S)$  is defined by generators and relations. Generators are pairs $(X,f)$, where $X$ is a variety over $S$ and $f:X\to\A^1$ is a morphism (and by $\A^1$ we mean $\A^1_{\Z}$).  We call such a pair a \textit{variety with exponential}.
Relations are the following:
$$(X,f)-(Y,f\circ u) $$
whenever $X$, $Y$ are $S$-varieties, $f\colon X\to\A^1$
a morphism, and $u\colon Y\to X$  is an $S$-isomorphism;
$$ (X,f)-(Y,f|_Y)-(U,f|_U) $$
whenever $X$ is an $S$-variety, $f\colon X\to\A^1$ a morphism,
$Y$ a closed  subscheme of~$X$ and $U=X\setminus Y$ its open complement; and
$$ (X\times_{\Z}\A^1,\pr_2)$$
where $X$ is a $S$-variety and $\pr_2$ is the second projection.
We  write~$[X,f]$ (or $[X,f]_S$ if we want to keep track of the base scheme $S$) for the class in $K_0(\evar_S)$ 
of a pair~$(X,f)$. The product $[X,f][Y,g] = [X\times_S Y, f\circ \pr_1+g\circ \pr_2]$ endows $K_0(\evar_S)$ with a ring structure. 

We denote by $\LL$, or $\LL_S$, the class of $[\A^1_S,0]$ in $K_0(\evar_S)$. As for the usual Grothendieck ring $K_0(\Var_S)$, we may invert $\LL$, which gives us a ring denoted by $\expp_S$. As proved in 
\cite[Lemma~1.1.3]{CLL}, the natural morphisms $K_0(\Var_S)\to K_0(\evar_S)$ and $\M_S\to \expp_S$ given by sending $[X]$ to $[X,0]$ are injective.

\begin{remark}[Interpretation as exponential sums]\label{sect:exp_sum_interpretation}
Let $\psi:\F_q\to \C^*$ be a non-trivial additive character. Then there is a ring homomorphism
$$
K_0(\evar_{\F_q})\to \C,
$$
given by 
 $$[X,f]\mapsto \sum_{x\in X(\F_q)} \psi(f(x)).$$
 In general, even when the ground field $k$ is not finite, we  think of the elements of $\evar_k$ and $\expp_k$ as exponential sums. 
\end{remark}

Any morphism $u:T\to S$ of noetherian schemes naturally induces a group homomorphism
$$u_{!}: K_0(\evar_{T}) \to K_0(\evar_{S})$$
and a ring homomorphism
$$u^*: K_0(\evar_{S}) \to K_0(\evar_T).$$
The latter endows $K_0(\evar_T)$ with a $K_0(\evar_S)$-module structure: in particular, this allows us to write expressions involving elements of both rings, which are implicitly  viewed as living in $K_0(\evar_T)$. 

\subsection{Functional interpretation}

An element $\phi\in K_0(\evar_S)$ may be interpreted as a motivic function with source $S$. More precisely, for every point $s\in S$ with residue field $k(s)$, denote by $\phi(s)$ the element $s^*\phi\in K_0(\evar_{k(s)})$. Then the following lemma says that a motivic function is determined by its values:
\begin{lemma}\label{lemma:zero_criterion} Let $\phi\in K_0(\evar_S)$. Assume that  $\phi(s) = 0$
for all $s\in S$.
Then $\phi = 0$. 
\end{lemma}
\begin{proof} See \cite[Lemma 1.1.8]{CLL}. 
\end{proof}
\begin{remark} 
Let $k$ be a field. Using the relations, we see that  $[\A^1, \lambda \id] = 0$ in $K_0(\evar_k)$ for any non-zero $\lambda \in k$. More generally,
for any variety $X$ over $k$ with a morphism $u:X\to \G_m$, we have
$$[X\times \A^1, (x,t) \mapsto u(x)t] =0$$
in $K_0(\evar_k)$. Indeed, using Lemma \ref{lemma:zero_criterion} we see that this holds already in $K_0(\evar_X)$. 
\end{remark}

\begin{lemma}\label{lemma:linear_form} Let $V$ be a finite-dimensional $k$-vector space and $f:V\to k$ a linear form. Then
$$[V,f] = 
\begin{cases}

 \LL^{\dim V} & \text{if}\ f = 0,\\
 0 & \text{otherwise.}\end{cases}
 $$
\end{lemma}
\begin{proof} \cite[Lemma 1.1.11]{CLL}.
\end{proof}

\subsection{A cohomological realisation}
Let $k$ be a finite field, let $\ell\neq \mathrm{char}\, k$ be a prime, and let $k^s$ be a separable closure of $k$. We denote $G_k = \mathrm{Gal}(k^s/k)$. Consider the category $\mathrm{Rep}_{G_k}\Q_{\ell}$ of continuous $\Q_{\ell}$-representations of $G_k$ and the corresponding Grothendieck ring $K_0(\mathrm{Rep}_{G_k}\Q_{\ell})$.  
We fix a non-trivial additive character $\psi:k\to \C^*$, and denote by $\mathcal{L}_{\psi}$ the corresponding Artin-Schreier sheaf on $\A^1_k$ (see \cite{deligne-sommes-trig}). There is a ring homomorphism

$$K_0(\evar_{k}) \to K_0(\mathrm{Rep}_{G_k}\Q_{\ell}),$$ 
given by 
$$[X,f] \mapsto \sum_{i\geq 0}(-1)^i [H^{i}_{\text{ét},c}(X\times_kk^s, f^{*}\mathcal{L}_{\psi})].$$
This motivic measure provides a dictionary between some computations in this paper and those carried out in \cite{BS}.

\subsection{Motivic functions and integrals}\label{sect:motivic_integrals} In this section, we follow the ideas of \cite[1.2.1]{CLL}, except that we adapt the notation and normalisations to make them more convenient for our setting. 

The motivic functions we  consider are  motivic analogues of compactly supported and locally constant functions on the locally compact field $k((t^{-1}))$, where $k$ is a finite field. For such a function $\phi$, there exist integers $M\geq N$ such that $\phi$ is zero outside $t^{M}k[[t^{-1}]]$, and such that $\phi$ is invariant modulo $t^{N}k[[t^{-1}]]$, so that $\phi$ may be seen as a function on the quotient $t^{M}k[[t^{-1}]] / t^{N} k[[t^{-1}]]$. 
We then say that $\phi$ is of level $(M,N)$.
The latter can be endowed with the structure of an affine space of dimension $M-N$ over the field $k$, through the identification
$$\begin{array}{ccc} t^{M}k[[t^{-1}]] / t^{N} k[[t^{-1}]]& \to& \A_k^{M-N}(k)\\
 a_M t^{M} + a_{M-1} t^{M-1} + \dots + a_{N+1}t^{N+1} \bmod{t^Nk[[t^{-1}]]}
 &\mapsto &(a_M,\dots,a_{N+1}) \end{array}.$$
\begin{notation} For any field $k$ and for $M\geq N$ integers, we denote by $\A^{(M,N)}_k$ the affine space $\A^{M-N}_k$, interpreted as the domain of definition for functions which are zero outside $t^{M}k[[t^{-1}]]$ and invariant modulo $t^{N}k[[t^{-1}]]$ as above. More generally, for $n\geq 1$ an integer, we denote by $\A^{n(M,N)}_k$ the affine space $(\A^{(M,N)}_k)^n$. If $S$ is a variety over $k$, we denote by $\A^{n(M,N)}_S$ the extension $\A^{n(M,N)}_k\times_k S$.
\end{notation}

Writing $K_\infty=k((t^{-1}))$, 
we denote by $\mathscr{F}_S(K_{\infty}^n,M,N)$ the ring $\expp_{\A^{n(M,N)}_S}$ and interpret it as the ring of $S$-families of motivic functions of level $(M,N)$ in $n$ variables. In the same way that is  explained in \cite[1.2.3]{CLL}, as $M$ and $N$ vary, the rings $\mathscr{F}_S(K_{\infty}^n,M,N)$ fit into a directed system with direct limit denoted $\mathscr{F}_S(K_{\infty}^n)$, the total ring of $S$-families of motivic functions in $n$ variables. 

If we have an element $\phi\in\mathscr{F}_S(K_{\infty}^n)$ we define its integral in the following way: pick a pair $(M,N)$ such that $\phi\in\mathscr{F}_S(K_{\infty}^n,M,N)$, and we put
$$
\int \phi  := \LL^{(N+1)n} [\phi]_S,
$$
where $[\phi]_S$ is the image of $\phi$ in $\expp_S$ via the forgetful morphism
$$\expp_{\A^{n(M,N)}_S}\to \expp_S.$$
We may also write $\int \phi= \int \phi(\x) \mathrm{d}\x $, if we want to stress with respect to which variables the integral is taken, or even
$\int \phi= \int_{\A^{n(M,N)}} \phi(\x) d\x$,
  if the choice of $M,N$ has been made explicit. 
Analogously to \cite[1.2.3]{CLL}, this definition does not depend on the choice of $(M,N)$ (since the integral remains the same if we pick a larger $M$ or a smaller~$N$), and defines an $\expp_S$-linear map
$$\int:\mathscr{F}_S(K_{\infty})\to \expp_S.$$
\begin{example} Write $\T = t^{-1}k[[t^{-1}]]$, the analogue of the fundamental compact interval $[0,1]$ in the function field circle method literature. Let $\phi = \1_{\T^n}$ be the characteristic function of $\T^n$. Picking $M = N = -1$, we get $\int \phi = 1$. In other words, with this normalisation, the motivic volume of $\T^n$ is 1. 
\end{example}

\begin{remark}[Change of variables]\label{rem:change_of_var} Let $\phi\in \expp_{\A^{n(M,N)}_S}$ and assume that we want to perform a change of variables $\x\mapsto t^p\x=:\y$ for some integer $p \in \Z$. 
This change of variable induces an isomorphism 
$$
\A^{n(M,N)}_S\to \A^{n(M+p,N+p)}_S,
$$ 
which gives us an isomorphism between the corresponding Grothendieck rings, sending $\phi$ to the function $\psi:\y\mapsto \phi(t^{-p}\y)$. Moreover, this induces an isomorphism of the underlying $S$-varieties with exponentials, which implies the equality $[\phi]_S = [\psi]_S$. We may conclude that we have the change of variable formula
\begin{align*}
\int_{\A^{n(M,N)}} \phi(\x) \mathrm{d} \x = \LL^{n(N+1)}[\phi]_S = \LL^{n(N+1)}[\psi]_S 
&= \LL^{-pn}\int_{\A^{n(M+p,N+p)}} \psi(\y) \mathrm{d}\y\\
&= \LL^{-pn} \int_{\A^{n(M+p,N+p)}}\phi(t^{-p}\y) \mathrm{d} \y.
\end{align*}

\end{remark}

\subsection{Motivic Euler products} \label{sect:motivic_euler_prods}
We  use the motivic Euler products of \cite[Section~3]{Bilu}; more specifically, the variant where the coefficients are varieties with exponentials. We recall its definition here.  Let $X$ be a variety over a field $k$, and let $$(A_i)_{i\geq 1} = (X_i, f_i:X_i \to \A^1)_{i\geq 1}$$ be a family of quasi-projective varieties with exponentials over $X$. Let $n\geq 1$ be an integer, and let $\omega = (n_i)_{i\geq 1}$ be a partition of $n$, with $n_i$ being the number of occurrences of $i$, so that $n = \sum_{i\geq 1} i n_i$. We define the variety with exponential $\Conf^{\omega}(X_i,f_i)_{i\geq 1}$ in the following way: consider the product
$$\prod_{i\geq 1} X_i^{n_i}$$ together with the morphism to $\prod_{i\geq 1}X^{n_i}$ induced by the structural morphisms $X_i\to X$. We denote by $\left( \prod_{i\geq 1} X_i^{n_i} \right)_{*,X}$ the open subset lying above the complement of the big diagonal in $\prod_{i\geq 1}X^{n_i}$ (that is, points with no two coordinates being equal). Then the variety $\Conf^{\omega}(X_i)_{i\geq 1}$underlying $\Conf^{\omega}(X_i,f_i)_{i\geq 1}$ is defined to be the quotient of $\left( \prod_{i\geq 1} X_i^{n_i} \right)_{*,X}$ by the natural permutation action of the product of symmetric groups 
 $\prod_{i \geq 1} \mathfrak{S}_{n_i}$. 
 The corresponding morphism $f^{\omega}:\Conf^{\omega}(X_i)_{i\geq 1}\to \A^1$ is induced in the obvious way by the natural $\prod_{i\geq 1} \mathfrak{S}_{n_i}$-invariant morphism $\prod_{i\geq 1} f_i^{n_i} : \prod_{i\geq 1} X_i^{n_i} \to \A^1$,  
 given by 
 $$
 (x_{i,1},\dots,x_{i,n_i})_{i\geq 1} \mapsto \sum_{i\geq 1} \left(f_i(x_{i,1}) 
 + \dots + f_i(x_{i,n_i}) 
  \right).
 $$

We define the motivic Euler product by
$$\prod_{x\in X}\left(1 + \sum_{i\geq 1} A_{i,x}T^i\right) : = 1 + \sum_{n\geq 1} \left(\sum_{\omega} [\Conf^{\omega}(X_i,f_i)_{i\geq 1}]\right)T^n\in K_0(\evar_k)[[T]],
$$
where the inner sum is over partitions $\omega$ of $n$.
Note that the left-hand side is only a \textit{notation} for the series on the right-hand side. In \cite{Bilu}, it is shown that this notion of motivic Euler product satisfies many good properties, in particular in terms of multiplicativity.

\section{The motivic circle method}\label{sec:circle}
We start by recalling the main ideas behind the function field version of the circle method. Let $k = \F_q$ be a finite field. The field $k$ is endowed with an additive character $\psi: k\to \C^{*}$ defined by
$$x\mapsto \exp\left(\frac{2i\pi}{p}\mathrm{Tr}_{\F_q/\F_p}(x)\right).$$
The role of the compact unit interval $[0,1]$ is played by the set 
$$\T = t^{-1}\F_q[[t^{-1}]] = \{\alpha \in \F_q((t^{-1})),\ \ord(\alpha) \leq -1\}.$$ The field $k((t^{-1}))$ is locally compact, and we may normalise its Haar measure so that $
\int_{\T}\mathrm{d}\alpha = 1$. 
 The igniting spark of the circle method over $\F_q(t)$ is the identity
$$
\int_{\T} \psi(\res(x\alpha))\mathrm{d} \alpha = 
\begin{cases}
1 & \text{ if}\ x= 0,\\
0& \text{ otherwise},
\end{cases}
$$
for any $x\in \F_q[t]$, where, for a Laurent series $f\in k((t^{-1}))$, we denote by $\res(f)$ the coefficient of $t^{-1}$ in $f$. 

If we want to count solutions of some polynomial equation over $k$, we may use this identity to rewrite the quantity we are interested in as the integral
$$\int_{\T} S(\alpha) \mathrm{d} \alpha$$
of an appropriate exponential sum $S(\alpha)$. The method then proceeds, as in the classical case, with cutting up $\T$ into major and minor arcs. The contribution of the major arcs is rewritten as a product of local factors, but one requires a suitably strong upper bound for the contribution from the minor arcs.

The main obstruction in passing from $k = \F_q$ to $k = \C$ is that we lose  local compactness
of the field of Laurent series, and there is no theory of integration in the classical sense that we could use to reproduce the above steps. This is why we need to introduce ideas from \textit{motivic integration}: exponential sums are replaced by elements in the Grothendieck rings  of varieties with exponentials, bearing in mind Section \ref{sect:exp_sum_interpretation}. 
Integrals are replaced by the operation defined in Section \ref{sect:motivic_integrals}. We use the cut-and-paste relations in the Grothendieck ring of varieties to decompose our expression into major and minor arcs. The decomposition into local factors of the major arc contribution is done via the notion of motivic Euler products, while the minor arcs are bounded in a topology defined using Hodge theory. 
Once this translation into the motivic setting has been properly made, the motivic circle method bears many similarities with the function field circle method, as implemented by Lee \cite{lee}.

\begin{notation}\label{notation:poly_and_series}
For any integer $m$, let  $\Poly_{\leq m}$ be the space of polynomials in $t$  
of degree $\leq m$ with coefficients in $\C$.
We let  $\MPoly_{m}$ denote the space of monic polynomials in $t$ 
of degree exactly $m$ with coefficients in $\C$. Both spaces are simply affine spaces over $\C$; the first one has dimension $m+1$ and the second one has dimension $m$.  We let $\T$ be the space of elements $\alpha\in \C((t^{-1}))$ such that 
$\ord(\alpha)\leq -1$. For an element $f\in \C((t^{-1}))$, we denote by $\{f\}\in \T$ its fractional part.
\end{notation}
\subsection{Approximation by rational functions}\label{sect:approximation}
The classical circle method relies greatly on approximation of real numbers by rationals. In the same way, its function field version depends on approximation of power series in $\T$ by rational functions.

When we want to approximate $\alpha \in \T$ by a rational function $\frac{h_1}{h_2}$, there are two parameters:
\begin{itemize} \item the bound on the degree of $h_2$;
\item the precision of the approximation, as  quantified by a bound on the order of $\alpha h_2 - h_1$. 
\end{itemize}
Thus, we define
$$\T_{m,s} = \{\alpha\in \T,\ \exists h_2\in \Poly_{\leq m}, h_1\in \Poly_{< \deg h_2},\ \ \ord (\alpha h_2-h_1) < -s\},
$$
for $m\geq 0$ and $s\geq 1$.

We can reformulate the definition of $\T_{m,s}$ purely in terms of linear algebra, as follows. Writing $\alpha = \sum_{i\geq 1}b_it^{-i}$ for $b_i\in \C$, 
we have $\alpha\in \T_{m,s}$ if and only if there exists a polynomial $c_0 + c_1 t + \dots + c_m t^m$ such that the coefficients of $t^{-1},\dots,t^{-s}$ in the series
$$\left( \sum_{i\geq 1} b_i t^{-i} \right) (c_0 + c_1 t + \dots + c_m t^m)$$
are zero. This gives us a linear system of $s$ equations in $c_0,\dots,c_m$:
\begin{align*}
b_1 c_0 + b_2 c_1 + \dots + b_{m+1}c_m &= 0,\\
                                          b_2c_0 + b_3c_1 + \dots + b_{m+2} c_m &= 0,\\
                                          &\vdots  \\
                                          b_s c_0 + b_{s+1} c_1 + \dots + b_{m+s} c_m&=0,
                                          \end{align*}
                                          which translates into 
                                                                                    $$\left(\begin{array}{c} c_0\\
                                          \vdots\\
                                          c_m \end{array}\right)\in \mathrm{Ker} \left( \begin{array}{cccc} b_1 & b_2 & \dots & b_{m+1} \\
                                          \vdots & & & \vdots \\
                                          b_s & b_{s+1} & \dots & b_{m+s} \\
                                          \end{array}\right).$$
                                          In other words, denoting by $M$ the above $s\times (m+1)$ matrix, we have that $\alpha \in \T_{m,s}$ if and only if $M$ has non-trivial kernel. Our first observation, is that this is always satisfied if $m+1> s$, that is, if $m\geq s$. In other words, whenever $m\geq s$ we have $\T_{m,s} = \T$.  This gives us the functional version of Dirichlet approximation.
                                          
                                          \begin{lemma}[Dirichlet approximation] 
                                          Let $\alpha \in \T$. Then for any $m\geq 1$ there exist polynomials $h_1,h_2$ such that $\deg h_1 < \deg h_2 \leq m$ and
                                                                                  $\ord(\alpha h_2 - h_1) < -m.$
                                        
                                        \end{lemma}
                                       
                              In general, the larger $s$ becomes, the more equations we impose on the coefficients $c_0,\dots,c_m$ of $h_2$, and we get a more and more restrictive set. 
                              
                              In the motivic setting, there is an extra parameter, namely the number of coefficients of $\alpha$ we actually take into account in our analysis. More precisely, we  look at expressions of the form $\res(\alpha f(g_1,\dots,g_n))$ where $f$ is a polynomial of degree $d$ and $g_1,\dots,g_n$ have degrees at most $e$, and thus only the first $de+1$ coefficients of $\alpha$  matter. If we want to make do with only considering these coefficients, i.e.\ if we do not want the matrix $M$ to contain any additional coefficients of $\alpha$, we would need to impose the condition
                              $s+m\leq de+1.$
Motivated by the extra constraint, we define
$$A^{de+1}_{m} = \{(b_1,\dots,b_{de+1})\in \A^{de+1}: 
\sum_{i=1}^{de+1} b_i t^{-i} \in \T_{m, de-m+1}\}.
$$

\begin{remark}
\label{rem:Am_0}
 Note in particular that when $m\geq \frac{de+1}{2}$, then $de+1-m\leq \frac{de+1}{2} \leq m$. Thus  $A^{de+1}_m = \A^{de+1}$ in this case. 
\end{remark}

\begin{remark}\label{rem:Am}
Note that  $A^{de+1}_m$ is the space of $(b_1,\dots,b_{de+1})\in \A^{de+1}$ such that there exist coprime polynomials $h_1,h_2$ with $h_2$ monic, $\deg h_1< \deg h_2 \leq m$ and $\ord\left(\alpha - \frac{h_1}{h_2}\right)\leq -de-2 + m-\deg h_2$. 
\end{remark}

\begin{remark} \label{rem:stratum-dimension}
The sets $A^{de+1}_m$ 
can be stratified into sets $A^{de+1}_{m,m'}\subset A^{de+1}_m$ of those $\alpha$ such that $h_2(t)$ is of degree exactly $m'\leq m$. 
The space of monic $h_2(t)$ has dimension $m'$, and the space of 
polynomials of  degree $<m'$ also has dimension $m'$. Hence the dimension of the space of pairs of coprime polynomials $(h_1(t),h_2(t))$ is at most $2m'$. For each such pair, the constraint 
$\ord\left(\alpha - \frac{h_1}{h_2}\right)\leq -de-2 + m-m'$ means that only the coefficients 
$b_{de+2-m+m'},\dots,b_{de+1}$ are allowed to run freely in $\alpha$.
We may therefore  deduce that
$$\dim A^{de+1}_{m,m'} \leq 2m' + (m-m')\leq m+m' \leq 2m,$$
from which we in particular have that 
$\dim A^{de+1}_m \leq 2m.$
\end{remark}

\subsection{Activation of the circle method}

Recall the definition \eqref{eq:brighton} of the affine variety $M_e$.
We are now ready to analyse its  class  $[M_e]$ in the Grothendieck ring, using the motivic circle method point of view.  Our first  result involves the 
identification of $\alpha\in \C((t^{-1}))$ with $\A^{de+1}$, via its coefficient vector in the expansion
$\alpha = b_{1} t^{-1} + \dots + b_{de+1} t^{-de-1}$, together with the 
set  $\Poly_{\leq e}$ from Notation~\ref{notation:poly_and_series}. Finally, we recall that the residue map $\res(\beta)$ is  defined to be the  coefficient of $t^{-1}$ in the Laurent series expansion of any element $\beta\in \C((t^{-1}))$.

\begin{lemma}\label{lem:active}
 We have
\begin{align*}
[M_e] 
=  \LL^{-de-1}[(\Poly_{\leq e})^n\times \A^{de+1}, \res(\alpha f(g_1,\dots,g_n))]
\end{align*}
in the Grothendieck ring $\expp_{\Poly_{\leq e}^n}.$
\end{lemma}

\begin{proof}  Write $$f(g_1,\dots,g_n) = c_0(\g) + c_1(\g) t + \dots + c_{de}(\g) t^{de},$$ where $c_0,\dots,c_{de}$ are regular functions in the coefficients of $g_1,\dots,g_n$,
and $$\alpha = b_{1} t^{-1} + \dots + b_{de+1} t^{-de-1}.$$
Then $$\res( \alpha f(g_1,\dots,g_n)) = b_1 c_0 + \dots + b_{de+1}c_{de}$$
is a linear form in the coefficients of $\alpha$. 

Consider now a point $\g\in (\Poly_{\leq e})^n$. Then on evaluating the right-hand side at $\g$, we get
$$\LL^{-de-1}[\A^{de+1}_{k(\g)}, b_1c_0(\g) + \dots + b_{de+1}c_{de}(\g)].$$
By Lemma \ref{lemma:linear_form},  this is non-zero if and only if all of the coefficients 
$c_0(\g),\dots,c_{de}(\g)$ are zero. This occurs if and only if $\g\in M_{e}$, and in this case the expression is equal to 1. We conclude using Lemma~\ref{lemma:zero_criterion} with 
$$
\phi=[M_e]-\LL^{-de-1}[(\Poly_{\leq e})^n\times \A^{de+1}, \res(\alpha f(\g))]
$$
and $S=(\Poly_{\leq e})^n$.
\end{proof}

Adopting the notation in Section \ref{sect:motivic_integrals}, we may also write 
$$
[M_e] 
 = \int_{\A^{(-1,-de-2)}} [(\Poly_{\leq e})^n\times \A^{(-1,-de-2)}, \res(\alpha f(g_1,\dots,g_n))] \dx \alpha
$$
in  $\expp_{\Poly_{\leq e}^n}$, in order to emphasise the similarity with the  classical circle method. 
In fact, 
our plan is to write
\begin{equation}\label{eq:plan}
[M_e] = N_{\maj} + N_{\minor},
\end{equation}
where $N_{\maj}$ (resp. $N_{\minor}$) is the contribution of the major (resp. minor) arcs. We want to compute $N_{\maj}$ as precisely as possible, and find a weight bound on $N_{\minor}$ of the form 
$ w(N_{\minor}) < 2\mu(e)$.
Let
\begin{equation}\label{eq:tilde-nu}
\tilde \nu := \frac{n-2^{d}(d-1)}{2^{d-2}}>0.
\end{equation}
Although we delay defining the precise major and minor arcs that we work with until the relevant sections, 
we proceed by recording the  two main steps in the motivic circle method, whose proofs  occupy the bulk of this paper.

\begin{prop}\label{prop:major}
Let $\tilde \nu>0$ be given by \eqref{eq:tilde-nu} and recall the varieties defined in 
\eqref{eq:lambda} and \eqref{eq:jet}.
Then 
$$
N_{\maj}
 =  \LL^{\mu(e)}\left(\mathfrak{S}(f) \cdot \lim_{N\to \infty} \LL^{-(n-1)N}[\Lambda_{N}(f,\infty)]
+ R_e\right),
$$
where 
$$
w(R_e)\leq 
\begin{cases}
4
  -\tilde\nu \left(1+\left\lfloor \frac{e+1}{2d-2}\right\rfloor\right) & \text{ if $d\geq 3$,}\\
\max\left( -\frac{(e+1)(n-2)}{2}, ~-\frac{n(e+2)}{2}+2e+8\right) & \text{ if $d=2$.}
\end{cases} 
$$
and 
$$
\mathfrak{S}(f)=\prod_{x\in \A^1} 
\lim_{N\to \infty }\LL^{-N(n-1)} [\Lambda_N(f,x)].
$$
\end{prop}

\begin{prop}\label{prop:minor}
Let $\tilde \nu>0$ be given by \eqref{eq:tilde-nu}. Then 
$$
w(N_{\minor})\leq 
 2\mu(e) + 4 - {\tilde \nu}\left(1+ \left\lfloor \frac{e+1}{2d-2}\right\rfloor \right).
$$
\end{prop}

Note that 
when $e=1$, corresponding to the case of  lines,  we  have the bound
$w(N_{\minor}) \leq 2 \mu(1) - \frac{n-2^d d}{2^{d-2}},$
which is interesting as soon as $n> 2^d d$, if $d\geq 3$.

\begin{proof}[Proof of Theorem \ref{theorem_naivespace}]
This follows on combining Propositions \ref{prop:major} and \ref{prop:minor} in  \eqref{eq:plan}.
\end{proof}

\section{The weight function}\label{s:weight}

Our weight function is constructed with the aim of being able to prove an analogue of the Weyl differencing argument  \cite[Proposition 5.5]{BS} in our setting, which we do in Proposition  \ref{prop:weyl_differencing}. Unfortunately,  the weight function from \cite{Bilu} 
is not in an appropriate form to carry this out. In order to  mimic the steps in the proof of \cite[Proposition~5.5]{BS}, which were done using the properties of cohomology, we need a 
function that in some sense measures the weights of the cohomology of the fibres, while the weight function from \cite{Bilu} used the global notion of weight of a mixed Hodge module.  Thus, while over $\C$ it coincides with the weight function from \cite{Bilu}, in the relative setting we define it rather using the weights of the underlying variations of Hodge structures (which introduces a shift by the dimension of the support). 

Note that while most of the intermediary estimates in the implementation of the motivic circle method happen in relative Grothendieck rings and thus use our new weight, the final results are stated in the Grothendieck ring over $\C$, where our new weight coincides with the weight from \cite{Bilu}. 

In Section \ref{sect:mhm}, we recall some definitions and properties concerning mixed Hodge modules and the corresponding Grothendieck rings. In Section \ref{sect:weight_function_def} we define the weight function at the level of Grothendieck rings of mixed Hodge modules, and prove its fundamental properties. Then we pass to Grothendieck rings of varieties in Section \ref{sect:weight_function_varieties} using the same procedure as in \cite{Bilu}. Finally, in Section \ref{sect:convergence_power_series} we explain what we mean by convergence of power series in the topology defined by the weight.

\subsection{Mixed Hodge modules}\label{sect:mhm}
For  references about mixed Hodge modules, the interested reader may consult the axiomatic introduction by Peters and Steenbrink in \cite[Chapter 14]{peters-steenbrink}, Schnell's notes \cite{schnell}, or Saito's original paper \cite{Saito90}.
\subsubsection{The category of mixed Hodge modules}
If $S$ is a variety over $\C$, we denote by $\MHM_S$ the abelian category of mixed Hodge modules on $S$ and by $D^b(\MHM_S)$ its  bounded derived category. A morphism $f:S\to T$ between complex varieties induces functors $f_!:D^b(\MHM_S)\to D^b(\MHM_T)$ and
$f^*: D^b(\MHM_T)\to D^b(\MHM_S)$.

In the case where $S$ is a point, the category $\MHM_{\pt}$ is exactly the category of polarisable mixed Hodge structures.  For any integer $d\in \Z$, we denote by~$\Q_{\pt}^{\Hdg}(d)\in \MHM_{\pt}$ the Hodge structure of type $(-d,-d)$ with underlying vector space~$\Q$. For $d=0$, it will be denoted simply by $\Q^{\Hdg}_{\pt}$. 

Every mixed Hodge module $M$ on a variety $S$ has a well-defined notion of support, which will will be denoted by $\supp(M)$. 

\subsubsection{The complex $\Q_X^{\Hdg}$}

For any complex variety $X$, we denote by 
$a_X:X\to \Spec\ \C$ its structural morphism, and by $\Q_X^{\Hdg}$ the complex of mixed Hodge modules given by $a_X^{*}\Q_{\pt}^{\Hdg}$.
In the case when $X$ is smooth and connected, the complex of mixed Hodge modules $\Q_X^{\Hdg}$ is concentrated in degree $\dim X$, and $\mathcal{H}^{\dim X}\Q_X^{\Hdg}$ is pure of weight $\dim X$, given by the pure Hodge module associated to the constant rank one variation of Hodge structures of weight 0 on $X$.

\subsubsection{Mixed Hodge modules with monodromy}
We denote by $\mhm_X^{\mon}$ the category of mixed Hodge modules $M$ on a complex variety~$X$ endowed with commuting actions of a finite order operator $T_s:M\to M$ and a locally nilpotent operator $N:M\to M(-1)$. The category $\mhm_X$ can be identified with a full subcategory of $\mhm_X^{\mon}$ via the functor
$$\mhm_X\to \mhm_{X}^{\mon},$$
sending a mixed Hodge module $M$ to itself with $T_s = \id$ and $N = 0$. We refer to \cite[Section~4.1.6]{Bilu} for a definition of the twisted external tensor product $\twtimes$.

\subsubsection{The weight filtration on Hodge modules}
Each $M\in\mhm_S$ has a finite increasing weight filtration $W_{\bullet}M$, the graded parts of which will be denoted $\gr^{W}_{\bullet}$. For a bounded complex of mixed Hodge modules $M^{\bullet}$, we say~$M^{\bullet}$ has weight $\leq n$ if $\gr^W_i\mathcal{H}^j(M^{\bullet}) = 0$ for all integers $i$ and $j$ such that $i>j+n$. 

For varieties $X$ and $Y$ over $\C$ we say that a functor $F:D^{b}(\mhm_X)\to D^{b}(\mhm_Y)$ does not increase weights if for every $n\in\Z$ and every $M^{\bullet}\in D^{b}(\mhm_X)$ with weight $\leq n$, the complex $F(M^{\bullet})$ is also of weight $\leq n$. In particular, for any morphism of complex varieties~$f$, the functors $f_!$ and $f^{*}$ do not increase weights (see \cite[(4.5.2)]{Saito90}).

\subsubsection{Grothendieck ring of mixed Hodge modules}
We refer to \cite[Section 4.1.7]{Bilu} for the definition of the Grothendieck rings $K_0(\mhm_X)$ and $K_0(\mhm_X^{\mon})$. For every morphism $f:X\to Y$ of complex varieties, the functors $f_!, f^{*}$ between the corresponding derived categories of mixed Hodge modules induce a group morphism
$$f_!: K_0(\MHM_X^{\mon})\to K_0(\MHM_Y^{\mon})$$
and a ring morphism
$$f^{*}: K_0(\MHM_Y^{\mon})\to K_0(\MHM_X^{\mon}).$$

\begin{fact}\label{rem:purehodgemodules}
Let $\HM_X$ be the subcategory of $\MHM_X$ of Hodge modules of pure weight. By decomposing into graded pieces, one can see that the natural inclusion $K_0(\HM_X)\to K_0(\MHM_X)$ induces an isomorphism. 
	\end{fact}
\subsubsection{Hodge modules and variations of Hodge structures}\label{sec:structure_thm} 

Saito's structure theorem (see e.g. \cite[Theorem 3.21]{Saito90}) gives a correspondence between pure polarisable Hodge modules on a variety $Z$ with strict support in $Z$ and direct systems of polarisable variations of Hodge structures (with quasi-unipotent local monodromies) over smooth open subsets of $Z$, with a shift in weight (by the dimension of the support). Moreover, this category of Hodge modules is semi-simple, with simple objects coming, through this correspondence, from variations of Hodge structures for which the monodromy representation is irreducible \cite[Theorem 14.37]{peters-steenbrink}. We will use these facts to give a particularly simple presentation of Grothendieck rings of mixed Hodge modules.

We denote by $\mathrm{VHS}_X$ the disjoint union, over all irreducible closed subvarieties $Z\subset X$, of the direct limit over smooth open dense $U\subset Z$ of the sets of isomorphism classes of  polarisable variations of pure Hodge structures (satisfying the above condition of having quasi-unipotent local monodromies with irreducible action) on $U$.

There is a well-defined morphism
$$\Z^{(\VHS_X)}\to K_0(\MHM_X)$$
sending a variation of Hodge structures over an open dense subset  $U\subset Z$ of an irreducible closed subvariety $Z\subset X$ to the corresponding Hodge module with strict support in $Z$. 

\begin{lemma} The above morphism is an isomorphism, i.e. $K_0(\MHM_X)$ is freely generated by the elements of $\mathrm{VHS}_X$.  
\end{lemma}
\begin{proof} By Fact \ref{rem:purehodgemodules}, we may replace $K_0(\MHM_X)$ by $K_0(\HM_X)$. We construct the morphism in the other direction by using decomposition into simple objects and the above correspondence. 
\end{proof}

A similar property holds for $K_0(\MHM_X^{\mon})$, considering the set $\VHS^{\mon}_X$ where the additional datum of the operators $T_s$ and $N$ is added. 

\subsection{Definition of the weight function} \label{sect:weight_function_def}
We may define a filtration $W_{\leq \bullet}$ on the group $K_0(\MHM^{\mon}_X)$ in the following way:  $W_{\leq n}K_0(\MHM^{\mon}_X)$ is the subgroup generated by \begin{itemize}
\item classes of pure Hodge modules $(M,\id,N)$ (i.e. with $T_s = \id$), with irreducible support, of weight $m$ such that $m- \dim\supp(M) \leq n$, and 
\item classes of pure Hodge modules $(M,T_s,N)$, with irreducible support, of weight $m$ such that $m-\dim\supp(M)\leq n-1$.
\end{itemize}
This dichotomy depending on the form of the semisimple monodromy is of the same type as in the definition of the weight filtration in \cite[Section~4.5.1]{Bilu}, which was motivated in particular by the definition of the twisted exterior product. 
 
 Note that by \ref{sec:structure_thm}, this corresponds to filtering $K_0(\MHM^{\mon}_X)$ by the weights of the underlying variations of Hodge structures, slightly twisted depending on whether the semi-simple part $T_s$ of the monodromy is trivial or not. More precisely, we may describe the associated graded pieces 
   of this filtration in the following form:
 \begin{lemma}\label{lem:hodge_modules_graded} We have $$\mathrm{Gr}^W_{n}K_0(\MHM^{\mon}_{X}) = W_{\leq n}K_0(\MHM^{\mon}_X)/W_{\leq n-1}K_0(\MHM^{\mon}_X) = \Z^{(VHS^{\mon}_{n,X})}$$
 where $\VHS^{\mon}_{n,X}$ is the subset of $\mathrm{VHS}^{\mon}_{X}$ given by variations of Hodge structures which have  weight $n-1$ and correspond to Hodge modules with non-trivial $T_s$, together with those of weight $n$ corresponding to Hodge modules with trivial $T_s$.
 \end{lemma}
 
\begin{proof} This follows from \ref{sec:structure_thm}, using the fact that the shift in weight between Hodge modules and the corresponding variations of Hodge structures is of exactly the dimension of the support. 
\end{proof}
 We then define
$$w_X(\a) = \min\{n: \a\in W_{\leq n} K_0(\MHM_X^{\mon})\}.$$
This weight function satisfies the following properties.

\begin{prop}\label{prop:weight-properties} Let $S$ be a complex variety. 
\begin{enumerate} \item 
\label{item:weight-property-sum} Let $(\a_i)_{ i\in I}$ be a finite family of elements of $K_0(\MHM^{\mon}_S)$. Then 
$$w_S\left(\sum_{i\in I}\a_i\right)\leq \max_{i\in I} w_S(\a_i).$$

\item  
\label{item:weight-property-pushforward} Let $p:T\to S$ be a dominant morphism with fibres of dimension $\leq d$. Let $\a\in K_0(\MHM^{\mon}_T)$. Then there exists a
non-empty
 open subset $S_0\subset S$ such that, denoting by $\b$ the pullback of $\a$ to $p^{-1}(S_0)$, we have
$$w_{S_0}(p_!\b) \leq w_{p^{-1}(S_0)}(\b) + 2d.$$
If $S$ is a point, we have
$$w(p_!\a) \leq w_{T}(\a) + 2d.$$
\item \label{item:weight-property-subset}Let $T\subset S$ be a locally closed subset and $i:T\to S$ the corresponding inclusion morphism. For every $\a\in K_0(\MHM^{\mon}_{T})$, we have
$$w_{S}(i_! \a) \leq w_{T}(\a).$$
\end{enumerate} 
\end{prop}

\begin{remark} Property (2) differs from the corresponding property of the weight in \cite[Lemma 4.5.1.3(d)]{Bilu},  where the upper bound holds with $d$ in place of $2d$.
\end{remark}
\begin{proof}[Proof of Proposition \ref{prop:weight-properties}] 
~
\begin{enumerate}\item This follows immediately from the definition.
\item 

For clarity, we first do the proof without the extra monodromy operators.

Let $V$ be a variation of Hodge structures of weight $n$ over an open subset $U$ of a subvariety $Z$ of $T$, and let $M_V$ be the corresponding Hodge module over $Z$. 

Since the functor $p_!$ is of cohomological amplitude $\leq d$ (see e.g. \cite[Lemma 4.1.4.2]{Bilu}, we know that $p_{!} M_V$ is a complex of mixed Hodge modules over $S$, of amplitude $\leq d$. Since the functor $p_{!}$ does not increase weights, this complex moreover has weight $\leq n + \dim Z$, which means that 
$$\gr_{i}^W\mathcal{H}^j(p_{!}M_V) = 0\ \ \ \ \text{if}\ \ \ \ i> j+ n + \dim Z.$$ 
Taking classes in $K_0(\MHM_S^{\mon})$,
$$[p_!M_V] = \sum_{i\leq d} (-1)^i [\mathcal{H}^i p_{!} M_V]$$
is a sum of Hodge modules of weight $\leq n + \dim Z + d$. When $S$ is a point, we can conclude directly here, since these Hodge modules turn out to be simply Hodge structures, of weights bounded by $n + \dim Z + d\leq n + 2d.$

We now turn to the case where $S$ is not a point. Assume first that $p(U)$ is contained in a closed subvariety of $S$. Then the result is trivially true by taking $S_0$ to be the complement of that subvariety. Thus, we may now assume that $p$ remains dominant when restricted to $U$, and we reduce to proving the result for the induced map $p':U\to S$. We pick a non-empty open subset $S_0$ of $S$ above which $p'$ is a topological fibration. Then above $S_0$, the Hodge modules $\mathcal{H}^i p'_!M_V$ are variations of Hodge structures, with supports of dimensions $\geq \dim Z-d$, whence the result.

Now, assume there is also monodromy involved. We have shown that up to the above restriction to open subsets, the underlying variations of Hodge structures of $p_!M_V$ had weights $\leq n+2d$ if we start with $V$ of weight $n$. If $M_V$ has non-trivial semisimple monodromy, then by definition $w(M_V) = n+1$ and in this situation the bound $w_{S_0}(p'_!M_V)\leq n+1 +2d$ holds.  If now $M_V$ has trivial monodromy, then so do all of the Hodge modules involved in $p_!M_V$. We then have $w(M_V) = n$ and the bound $w_{S_0}(p'_!M_V)\leq n+ 2d$ holds. 
\item We write $\a = \sum a_i M_i$ where $a_i \in \Z$  and the $M_i$ are pure Hodge modules (with monodromy) supported inside $S_0$, with corresponding variations of Hodge structures $V_i$. Then $w_S(i_!\a)$ is at most the weight of $\sum a_i i_!M_i$ which has the same variations of Hodge structures occurring, hence the result.  \qedhere
\end{enumerate}
\end{proof}

\subsection{The weight function on Grothendieck rings of varieties}\label{sect:weight_function_varieties}

We have defined a weight filtration on the Grothendieck ring of mixed Hodge modules with monodromy $K_0(\MHM_S^{\mon})$. We now use the same procedure as in \cite{Bilu} to deduce from it a weight function on the Grothendieck ring of varieties with exponentials $\expp_S$. We only give a brief overview of the construction here, referring to \cite{Bilu} and the references therein for the details.

\subsubsection{Motivic vanishing cycles}
Let $k$ be a field of characteristic 0 and let $S$ be a variety over $k$. We refer to \cite[Section 2.3.5]{Bilu} for the definition (building on some earlier work \cite{lunts-schnurer} of Lunts and Schnürer) of the \textit{total vanishing cycles measure}
$$\Phi_S:\expp_S\to (\M_S^{\hat{\mu}}, \ast),$$
which is functorial in $S$. Here $\M_S^{\hat{\mu}}$ is the Grothendieck ring of varieties over~$S$ with $\hat{\mu}$-action (see \cite[Section 2.1.5]{Bilu}), and $\ast$ denotes Looijenga's convolution product (see \cite[Section 2.2.1]{Bilu}). Recall moreover that the restriction of $\Phi_S$ to $\M_S$ coincides with the natural inclusion $\M_S\to \M_S^{\hat{\mu}}$.

\subsubsection{The Hodge realisation}
Let $S$ be a complex variety.  There is a group morphism
$$\begin{array}{rccc}\chi^{\Hdg}_S:&\M_S^{\hat{\mu}}&\to &K_0(\mhm_S^{\mon})\\
                                     & [X\xrightarrow{f} S,\sigma] & \mapsto& \sum_{i\in\Z}(-1)^i[\mathcal{H}^i(f_!\Q_X^{\Hdg}),T_s(\sigma), 0] \end{array}$$
called the \textit{Hodge realisation morphism} (see \cite[Section 4.4.1]{Bilu} for a discussion of its properties).

\subsubsection{Weight function on Grothendieck rings of varieties}
Let $S$ be a complex variety. The weight filtration on the ring $\expp_S$ is given by
$$
W_{\leq n}\expp_S:= (\chi_S^{\Hdg}\circ \Phi_S)^{-1}(W_{\leq n}K_0(\mhm_S^{\mon})),
$$
for every $n\in \Z$. The completion with respect to this filtration is denoted by $\widehat{\expp_S}$.
The weight function on $\expp_S$, again denoted $w_S$, is given by the composition
$$\expp_S\xrightarrow{\Phi_S} \M_S^{\hat{\mu}}\xrightarrow{\chi_S^{\Hdg}}K_0(\mhm_S^{\mon})\xrightarrow{w_S}  \Z.$$

\subsubsection{Weight and dimension}

\begin{lemma}\label{lemma:weight_dimension} Let $S$ be a complex variety and $X$ a variety over $S$. One has the equality
$$w_S(X) = 2\dim_S X.$$
\end{lemma}

\begin{proof} Let $p:X\to S$ be the structural morphism, which we may assume to be dominant by replacing $S$ with the closure of $p(X)$. Further, we pick a stratification $(S_i)_i$ of $S$ such that over every stratum $p$ is topologically a fibration. For every $i$, put $X_i = p^{-1}(S_i)$ and $p_i = p_{|X_i}$. Up to refining our stratification, we may assume that over each $S_i$, the Hodge modules involved in the complex $(p_i)_!\Q^{\Hdg}_{X_i}$ correspond to variations of Hodge structures defined over $S_i$. Now pick a stratum $X_i$. 
By proper base change, for every $s\in S_i$, the fibre $((p_i)_!\Q^{\Hdg}_{X_i})_s$ is given by a complex with cohomology given by $H^*(X_s,\Q)$, which involves Hodge structures of weights at most $\dim_{S_i}X_i$, with equality for the Hodge structure in top weight. From this we may deduce, using Lemma \ref{lem:hodge_modules_graded}, that the weight of $p_!\Q_X^{\Hdg}$ is $\max\dim_{S_i}X_i = \dim_SX.$
\end{proof}
\begin{lemma}[Triangular inequality for weights]\label{lemma:triangular_ineq} Let $S$ be a complex variety, $X$ a variety over $S$ and $f:X\to \A^1$ a morphism. Then
$$w_S([X,f])\leq w_S(X).$$
\end{lemma}
\begin{proof} By \cite[Proposition 2.3.5.2]{Bilu}, we have the triangular inequality for motivic vanishing cycles 
$$\dim_S(\Phi_S([X,f]))\leq \dim_SX.$$
On the other hand, following the proof of Lemma 4.6.3.2 in \cite{Bilu}, we see that
$$w_S(\Phi_S([X,f])) \leq 2 \dim_S \Phi_S([X,f]).$$
(Here we consider the weight function on $\M_S^{\hat{\mu}}$, as defined in \cite{Bilu}.) 
We conclude by using Lemma \ref{lemma:weight_dimension}.
\end{proof}

\begin{lemma}[Cancellation of maximal weights] \label{lemma:weight_cancellation}
Let
 $X$ and $Y$ be complex varieties, both irreducible and of dimension $d$. Then
$$w([X]-[Y]) \leq 2d-1.$$
\end{lemma}
\begin{proof} This follows from \cite[Lemma 4.6.3.4]{Bilu}, given that over a point, our weight and the weight in \cite{Bilu} coincide. 
\end{proof}

\subsection{Convergence of power series and evaluation}\label{sect:convergence_power_series}
Let $X$ be a variety over $\C$, and consider a power series 
$$
F(T) =\sum_{i\geq 0} X_i T^i\in \expp_X[[T]].
$$ 
The radius of convergence of $F$ is defined by 
$$
\sigma_F = \limsup_{i\to \infty} 
\frac{w_X(X_i)}{2i}.
$$
We say that $F(T)$ converges for $|T|< \LL^{-r}$ if $r\geq \sigma_F$. If $F(T)$ converges for $|T|< \LL^{-r}$, then $F(\LL^{-m})$ exists as an element of $\widehat{\expp_X}$ for every $m> r$. 

If $F(T) = \prod_{x\in X}\left(1 + \sum_{i\geq 1} X_{i,x} T^i\right)$ is a motivic Euler product which converges for $|T|< \LL^{-r}$, then for any integer $M> r$ we write $F(\LL^{-m})$ in the form
$$\prod_{x\in X}\left(1 + \sum_{i\geq 1} X_{i,x} T^i\right)_{|T=\LL^{-m}}.$$

\begin{remark} Special values of motivic Euler products should be handled with extreme care, as motivic Euler products do not behave well with respect to non-monomial substitutions. (See \cite[Section 6.5]{BiluHowe} for a discussion of this fact.) In particular, in principle, when writing a special value of a motivic Euler product, one should always write out the formal product, and then specify at what value of $T$ it has been evaluated. 
\end{remark}

With this caveat, to highlight the analogy with number theory, we will allow ourselves the following abuse of notation (which is already employed in \cite{LVW} and \cite{faisant23}).

\begin{notation}\label{notation:euler_product} We will denote by
$$\prod_{x\in X} a_x$$
the evaluation at $T=1$ of the motivic Euler product
$$\prod_{x\in X}\left(1 + (a_x-1)T\right).$$
\end{notation}
\section{Weyl differencing and a general exponential sum bound}\label{sec:weyl}

Both the major and minor arc treatments in the circle method rely on a general upper bound for certain exponential sums, which in the $\F_q(t)$-setting take the form
$$\sum_{(g_1,\dots,g_n)\in \Poly_{< E}^n} \psi(\res(\alpha f(g_1,\dots,g_n)).$$
The aim of this section is to provide such bounds in the motivic setting. We start by proving a property of the weight in Section \ref{section:weight_property}, which allows us to adapt to our setting a geometric version of the classical Weyl differencing argument provided in \cite{BS}. This is achieved in Proposition \ref{prop:weyl_differencing}. The rest of the section is dedicated to proving our general bound. After an application of Proposition \ref{prop:weyl_differencing}, we rewrite our bound in terms of the multilinear forms associated to the polynomial $f$, and apply some inequalities coming from the geometry of numbers over $\C(t)$. 
We point out that we never have to rely on any point counting estimates, 
unlike for the procedure in \cite{BS}. 
Indeed, we   circumvent the use of spreading out arguments and rely purely on estimates for the dimensions of various  spaces. 

\subsection{Key properties of the weight}\label{section:weight_property}

\begin{lemma}\label{lem:minus} For every variety $X$ over $S$ and every morphism $f:X\to \A^{1}$, we have
$$\Phi_S([X,f]) = \Phi_{S}([X,-f]).$$
\end{lemma} 

\begin{proof} By Lemma \ref{lemma:zero_criterion} and \cite[Theorem 2.3.5.1]{Bilu} we may reduce to the case where $S$ is the spectrum of a field $k$ of characteristic zero. Now, by Bittner's presentation of the Grothendieck ring of varieties, we know that $\expp_{k}$ (which is a quotient of $K_0(\Var_{\A^1_k})$) is generated by classes $[X,f]$ such that $X$ is smooth and $f$ is proper (as in \cite[Theorem~5.1.3]{CLL}). For such a class, using \cite[Theorem 2.3.5.1]{Bilu} we may write 
$$\Phi_k([X,f]) = \eps_{!}(\phi_f^{\tot}),$$
where $\eps:\A^{1}\to k$ is the structural morphism and $\phi_f^{\tot}\in \M_{\A^{1}_k}^{\hat{\mu}}$ are the total motivic vanishing cycles  defined in \cite[Section 2.3.1]{Bilu}. We now remark that for a morphism $g:X\to \A^{1}$ defined on a smooth variety $X$, the motivic vanishing cycles $\phi_g$ only depend on a log-resolution of the zero locus of $g$, and therefore we have $\phi_g = \phi_{-g}$. Using this, we see that
for every $a\in \A^{1}_k$, we have 
$$(\phi^{\tot}_{f})_a = \phi_{f-a} = \phi_{-f+a} = (\phi^{\tot}_{-f})_{-a}$$
in $\M_{k(a)}^{\hat{\mu}}$. We thus see that, denoting by $i:\A^1\to \A^1$ the map $a\mapsto -a$, we have the identity $\phi^{\tot}_f = i_!\phi^{\tot}_{-f}$. Pushing to $\spec k$ via $\eps$ we get the result. 
\end{proof}

Throughout this section, let $S$ be a variety over $\C$ and let 
$G\in \OO_S[x_1,\dots,x_N]$, for an integer $N\geq 1$. For every point $s\in S$, we denote by $G_s$ the pullback of $G$ via $s$. 

\begin{lemma} \label{lem:weight-product} We have $$w_S([\A^N_S, G][\A^N_S,-G]) = 2w_S([\A^N_S,G]).$$
\end{lemma}

\begin{proof} 

We write
$$\chi_S^{\Hdg} \circ \Phi_S([\A^N_S, G]) = \sum_{Z_i\subset S}a_iH_i,$$
where each $H_i$ is a simple pure polarisable variation of Hodge structures with monodromy over $Z_i$. Up to stratifying, we may assume that for every $i,j$, either $Z_i= Z_j$ or $Z_i\cap Z_j = \varnothing$. 

Since $H_i$ is polarisable, denoting by $n_i$ its weight we have the selfduality 
$$
H_i \simeq H_i^{\vee} (-n_i).
$$
From this, using
Lemma \ref{lem:minus} and the  fact that $\chi_S^{\Hdg}$ and $\Phi_S$ are ring morphisms,
we get
\begin{align*}\chi_S^{\Hdg} \circ \Phi_S([\A^N_S, G][\A^N_S, -G])& = \chi_S^{\Hdg} \circ \Phi_S([\A^N_S, G])\chi_S^{\Hdg} \circ \Phi_S([\A^N_S, -G])  \\
&=\left( \sum_{Z_i\subset S}a_iH_i \right)\left( \sum_{Z_i\subset S}a_iH_i^{\vee}(-n_i) \right)\\
& =  \sum_{i,j} a_ia_j \mathrm{Hom}(H_i, H_j)(-n_i). 
\end{align*}
Because of the assumption on $Z_i, Z_j$ for $i\neq j$, and because of the simplicity of the $H_i$, we get that 
$$\chi_S^{\Hdg} \circ \Phi_S([\A^N_S, G][\A^N_S, -G])= \sum_{i} a_i^2\Q^{\Hdg}_S(-n_i),
$$
where $\Q_S^{\Hdg}$ denotes the constant rank one variation of Hodge structures of weight 0 over~$S$. We get a sum of variations of Hodge structure with positive coefficients: by Lemma \ref{lem:hodge_modules_graded}, its weight is therefore $2\max_i{n_i} = 2w_S([\A^N_S,G]).$
\end{proof}

\subsection{The inductive argument}

We define $V(G)$ to be the variety over $S$ given by $$
(\y^{(1)},\dots,\y^{(d-1)},s)\in (\A^{N})^{d-1}_S$$ 
such that 
$$\sum_{\eps_1,\dots,\eps_{d-1}\in \{0,1\}}(-1)^{\eps_1 + \dots + \eps_{d-1}}G_s(\x + \eps_1 \y^{(1)} + \dots + \eps_{d-1}\y^{(d-1)})$$
is a constant function of $\x$. 

\begin{prop} \label{prop:weyl_differencing}Let $G\in \OO_S[x_1,\dots,x_N]$ be a polynomial of degree $\leq d$ in $N$ variables. The weight function $w_S: \expp_S\to \Z$ satisfies
$$w_S([\A^{N}_S, G]) \leq \frac{\dim_SV(G) + N(2^{d-1} - (d-1))}{2^{d-2}}.$$
\end{prop}

\begin{proof} 
We essentially follow  the proof of \cite[Proposition 5.5]{BS}, working  by induction on~$d$, except that it is crucial for us to work over a base, as the ground field is not finite and we need to apply the induction hypothesis in the relative setting. Assume first that $d=1$. Then $V(G)$ is a subvariety of $S$ given by the $s\in S$ such that $G_s(\x)$ is constant. Assume first that $V(G)$ is empty. Then for every $s\in S$, the polynomial $G_s(\x)$ is of degree exactly 1, and so $[\A^N_{\kappa(s)}, G_s]$ = 0 by
Lemma \ref{lemma:linear_form}. Lemma~\ref{lemma:zero_criterion} therefore yields $[\A_S^N,G] = 0$, and so both sides of the inequality in the statement are equal to $-\infty$. On the other hand, if $V(G)$ is nonempty, then $\dim_SV(G) = 0$. By the triangular inequality of Lemma \ref{lemma:triangular_ineq}, we have 
$$w_S([\A^N_S, G]) \leq 2N,$$
which is exactly the inequality that needed to be proved. 

We now assume that the result is already known for polynomials of degree $\leq d-1$. Let~$G$ be a polynomial of degree $d\geq 2$. 
We begin with an application of 
Lemma \ref{lem:weight-product}, which yields
\begin{align*}
2w_S([\A_S^N,G]) &\leq w_S([\A_S^N,G][\A_S^N,-G])\\
&=
w_S([\A_S^N\times_S \A_S^N,G(\x)-G(\y)])\\
&=
w_S([\A_S^N\times_S \A_S^N,G(\x)-G(\x+\y)]),
\end{align*}
since $G(\x)-G(\y)$ is related to 
$G(\x)-G(\x+\y)$ via an invertible change of variables.  

For any $\y_0=(y_{1},\dots,y_{N})$, the difference $G(\x) - G(\x+\y_0)$ is a polynomial of degree $d-1$ in $x_1,\dots,x_{N}$. The variety $V(G)$  admits a map $p:V(G)\to \A_S^N$ along the coordinates $\y$, whose fibre over a point $\y_0\in \A_S^N$ is the variety in the definition of 
$V( G(\x) - G(\x+\y_0))$.  
Choose a stratification of $\A_S^N$, with  strata $W_j$ being  varieties  such that the fibre dimension of this map is constant on each stratum $W_j$.
Thus 
\begin{equation}\label{eq:yellow}
\dim_{W_j} V( G(\x) - G(\x+\y_0))=\dim_S p^{-1}(W_j)-\dim_S W_j \leq \dim_S V(G) -\dim_S W_j,
\end{equation}
for any  $\y_0\in W_j$.
On appealing to part (\ref{item:weight-property-sum}) of Proposition  \ref{prop:weight-properties},
it follows that
$$
2w_S([\A_S^N,G]) \leq \max_j  
w_S([\A_S^N\times_S W_j,G(\x)-G(\x+\y)]).
$$
Fix a stratum $W_j$ and denote by $q_j: W_j\to S$ the restriction to $W_j$ of the projection map $\A^N_S \to S$. Up to refining our stratification, we may assume that $q_j:W_j\to S_j$ (where $S_j = q_j(W_j)$) has fibres of constant dimension $\dim_S W_j=r\in [0,N]$ and that
$$
w_{S_j}([\A_{S_j}^N\times_{S_j} W_j,G(\x)-G(\x+\y)])\leq 
w_{W_j}([\A_{W_j}^N,G(\x)-G(\x+\y_0)]) +2r,
$$
by part (\ref{item:weight-property-pushforward}) of 
Proposition  \ref{prop:weight-properties}. 
It follows from \eqref{eq:yellow} and the inductive hypothesis that 
$$
w_{W_j}([\A_{W_j}^N,G(\x)-G(\x+\y_0)])\leq 
\frac{\dim_SV(G) -r+ N(2^{d-2} - (d-2))}{2^{d-3}},
$$
for any $\y_0\in W_j$. 
Hence it follows that 
$$
w_{S_j}([\A_{S_j}^N\times_{S_j} W_j,G(\x)-G(\x+\y)])
\leq 
\frac{\dim_SV(G) +(2^{d-2}-1)r+ N(2^{d-2} - (d-2))}{2^{d-3}}.
$$
We now apply part (\ref{item:weight-property-subset}) of Proposition \ref{prop:weight-properties} with $T=S_j$ (which up to refining the stratification some more may be assumed locally closed) to deduce that 
$$
w_S([\A_S^N\times_S W_j,G(\x)-G(\x+\y)])
\leq 
\frac{\dim_SV(G) +(2^{d-2}-1)r+ N(2^{d-2} - (d-2))}{2^{d-3}}.
$$
Taking the maximum over the different strata, we finally deduce that 
\begin{align*}
2w_S([\A_S^N,G])
&\leq \max_{0\leq r\leq N} 
\frac{\dim_SV(G) +(2^{d-2}-1)r+ N(2^{d-2} - (d-2))}{2^{d-3}}\\
&=
\frac{\dim_SV(G) +N(2^{d-1} - (d-1))}{2^{d-3}},
\end{align*}
which suffices to complete the proof.
\end{proof}

Lemma \ref{lemma:triangular_ineq}  yields the trivial upper bound 
$w_S([\A^N_S, G]) \leq 2N$, whereas 
Proposition \ref{prop:weyl_differencing} yields
$$
w_S([\A^{N}_S, G]) \leq 2N -\frac{ N(d-1)-\dim_SV(G)}{2^{d-2}}.
$$
Thus we get a saving over the trivial bound as soon as we get a non-trivial bound for the dimension of $V(G)$ over $S$.

\subsection{A general bound for exponential sums}\label{sec:exp_sums_general_bound}

In this section, we want to prove a general upper bound for an exponential sum of the form
$$ [\Poly_{<E}^n\times S, \res(\alpha f(g_1,\dots,g_n))],$$
when $S$ is a variety parameterising  values of $\alpha$. 
We  write our polynomial $f$ in the form
$$
f(x_1,\dots,x_n)=\sum_{j_1,\dots,j_d=1}^n c_{\mathbf{j}} x_{j_1}\dots x_{j_d},
$$
for symmetric coefficients $c_{\mathbf{j}}$. (i.e.\  $c_{\mathbf{j}}=c_{\sigma(\mathbf{j})}$ for any  $\sigma\in S_d$.)
Associated to $f$ are the multilinear forms
$$
\Psi_j(\h^{(1)},\dots,\h^{({d-1})})
=d!\sum_{j_1,\dots,j_{d-1}=1}^n c_{j_1,\dots,j_{d-1},j} 
h_{j_1}^{(1)}\dots h_{j_{d-1}}^{(d-1)},
$$
for $1\leq j\leq n$. Note that 
$f(\x)=\frac{1}{d!}\sum_{j=1}^n \Psi_j(\x,\dots,\x) x_j$.

Recall that for a Laurent series $h = \sum_{i\leq M}h_i t^{i}\in k((t^{-1}))$ we denote by
$\{h\} = \sum_{i\leq -1}h_i t^i$ its fractional part. The
main goal of this section is to prove the following general bound for the weight of the exponential sum.

\begin{prop}\label{pro:stone}
We have 
$$
w_{S}(
[\Poly_{<E}^n\times S, \res(\alpha f(g_1,\dots,g_n))]
)\leq 
\frac{\max_{\alpha\in S} \dim N(\alpha) + En(2^{d-1} - (d-1))}{2^{d-2}},$$
where
\begin{equation}\label{eq:Nalpha}
N(\alpha) = \left\{ (\u^{(1)},\dots,\u^{(d-1)})\in \left(\Poly_{< E}^n\right)^{d-1}: 
\begin{array}{l}
\ord \{\alpha \Psi_j(\u^{(1)},\dots,\u^{(d-1)})\} < -E\\
\forall j\in \{1,\dots,n\}
\end{array}
\right\}.
\end{equation}
\end{prop}

\subsubsection{Weyl differencing}
Denote by $G_{\alpha}$ the polynomial 
$$G_{\alpha} = \res\left(\alpha f(g_1,\dots,g_n)\right),
$$
for $\alpha\in S$, 
considered as a function in the coefficients of $g_1,\dots,g_n$. Writing
$$g_j(t) = a_{0,j} + a_{1,j}t + \dots + a_{E-1,j}t^{E-1},$$
we see that $G_{\alpha} $ is a polynomial of degree $d$ in the $N = En$ variables $(a_{i,j})$ for $0\leq i<E$ and  $1\leq j \leq n$. In this notation we have 
$$[\Poly_{<E}^n\times S, \res(\alpha f(g_1,\dots,g_n)]=[ \A^N_S, G_\alpha].
$$
Hence, on applying Proposition \ref{prop:weyl_differencing}, we get
$$ w_{S}([ \A^N_S, G_\alpha]) \leq \frac{\dim_SV(G_\alpha) + N(2^{d-1} - (d-1))}{2^{d-2}}.$$
Thus, to bound our exponential sum, it suffices to obtain a bound on $\dim V(G_{\alpha})$ for every $\alpha\in S$.

\subsubsection{Reformulation using multilinear forms}
Recalling that $\alpha=b_1t^{-1}+\cdots+b_{de+1}t^{-de-1}$, we see that 
the polynomial $G_{\alpha}$ is the coefficient of $t^{-1}$ in
$$(b_1t^{-1} + b_2 t^{-2} + \cdots )\sum_{j_1,\dots,j_d=1}^n c_{\mathbf{j}} \prod_{i=1}^d(a_{0,j_i} + a_{1,j_i}t + \dots + a_{E-1, j_i}t^{E-1}),$$
$$= (b_1t^{-1} + b_2 t^{-2} + \cdots ) \sum_{j_1,\dots,j_d=1}^n c_{\mathbf{j}} \sum_{i_1,\dots,i_d = 0}^{E-1} a_{i_1,j_{1}}\dots a_{i_d,j_d}t^{i_1 + \dots + i_d},
$$
which is given by
$$
\sum_{j_1,\dots,j_d = 1}^n\sum_{i_1,\dots,i_d= 0}^{E-1} b_{i_1 + \dots+  i_d + 1} c_{\mathbf{j}}a_{i_1,j_{1}}\dots a_{i_d,j_d}. 
$$
Using this expression, $V(G_{\alpha})$ is the set of $(\y^{(1)},\dots,\y^{(d-1)})\in (\A^{N})^{d-1}$ such that 
\begin{equation}\label{eq:lamb}
\sum_{\eps_1,\dots,\eps_{d-1}\in \{0,1\}}(-1)^{\eps_1 + \dots + \eps_{d-1}}\sum_{\mathbf{i,j}}b_{i_1 + \dots+ i_d + 1} c_{\mathbf{j}}\prod_{k=1}^d(a_{i_k,j_k} + \eps_1 y^{(1)}_{i_k,j_k} + \dots + \eps_{d-1}y_{i_k,j_k}^{(d-1)})
\end{equation}
is a constant function in the $(a_{i,j})_{\substack{0\leq i\leq E-1\\1\leq j \leq n}}$. On expanding the product, the expression  $\prod_{k=1}^d(a_{i_k,j_k} + \eps_1 y^{(1)}_{i_k,j_k} + \dots + \eps_{d-1}y_{i_k,j_k}^{(d-1)})$ is clearly linear in the variables $a_{i_1,j_1},\dots, a_{i_d,j_d}$. 
Hence \eqref{eq:lamb} is at most linear in the $(a_{i,j})_{\substack{0\leq i\leq E-1\\1\leq j \leq n}}$, and so it suffices to ensure that the coefficient of each $a_{i,j}$ vanishes. This gives us equations
\begin{equation}\label{eq:V_of_G_rewrite}\sum_{i_1,\dots,i_{d-1} = 0}^{E-1}\sum_{j_1,\dots,j_{d-1} = 1}^n b_{i_1 + \dots + i_{d-1} + i + 1}c_{j_1,\dots,j_{d-1},j}y_{i_{1},j_1}^{(1)}\dots y_{i_{d-1}, j_{d-1}}^{(d-1)} = 0,\end{equation}
for every $i\in \{0,\dots,E-1\}$ and every $j\in\{1,\dots,n\}$. 
The statement of Proposition \ref{pro:stone} is now a direct consequence of the following result. 

\begin{lemma} 
There is a canonical identification
$V(G_{\alpha}) = N(\alpha),$
where  $N(\alpha)$ is given by \eqref{eq:Nalpha}.
\end{lemma}

\begin{proof} Writing $\u^{(i)}= \sum_{k=0}^{E-1} \h_{k}^{(i)} t^k,$ with $\h_{k} = (h_{k,1},\dots,h_{k,n})\in \A^{n}$, we see that 
\begin{align*}  
\alpha &\Psi_j(\u^{(1)},\dots,\u^{(d-1)})\\
& =  \left(\sum_{r\geq 1}b_rt^{-r}\right) d! \sum_{j_1,\dots,j_{d-1} = 1}^{n}c_{j_1,\dots,j_{d-1},j}u_{j_1}^{(1)}\dots u_{j_{d-1}}^{(d-1)}\\
& =   d!  \left(\sum_{r\geq 1}b_rt^{-r}\right) \sum_{j_1,\dots,j_{d-1} = 1}^{n}c_{j_1,\dots,j_{d-1},j}\sum_{i_1,\dots,i_{d-1} = 0}^{E-1}h_{i_1,j_1}^{(1)}\dots h_{i_{d-1},j_{d-1}}^{(d-1)} t^{i_1 + \dots + i_{d-1}}\\
& =  d! \sum_{r\geq 1}b_r  \sum_{j_1,\dots,j_{d-1} = 1}^{n}c_{j_1,\dots,j_{d-1},j}\sum_{i_1,\dots,i_{d-1} = 0}^{E-1} h_{i_1,j_1}^{(1)}\dots h_{i_{d-1},j_{d-1}}^{(d-1)}t^{i_1+ \dots + i_{d-1} - r}.
\end{align*}
The space $N(\alpha)$ is defined by the vanishing of the coefficients of $t^{-1},\dots, t^{-E}$. For $i \in \{0,\dots,E-1\}$, the coefficient of degree $t^{-i-1}$ is given by 
$$d! \sum_{i_1,\dots,i_{d-1}= 0}^{E-1}\sum_{j_1,\dots,j_{d-1} = 1}^n b_{i_1+ \dots + i_{d-1} + i + 1} c_{j_1,\dots,j_{d-1},j}  h_{i_1,j_1}^{(1)}\dots h_{i_{d-1},j_{d-1}}^{(d-1)}.$$
 Comparing with equation (\ref{eq:V_of_G_rewrite}), we get the required identification. 
\end{proof}

Thus, it now remains to find a bound on $\dim N(\alpha)$. To be able to apply it in both  the major and in the minor arc setting, it will rely on parameters associated to a rational approximation of $\alpha$.  

\subsubsection{A dimension bound on zero-sets defined by multilinear forms} 
For given $E\geq 1$, 
let $V_E\subset \A^{(d-1)En}$ denote the variety 
of points 
$$
(\h_k^{(1)},\dots, \h_k^{(d-1)})_{0\leq k<E}\in \A^{(d-1)En},
$$ 
for which 
 $$
\Psi_j\left(\h_0^{(1)}+t\h_1^{(1)}+\dots+t^{E-1}\h_{E-1}^{(1)},\dots,
\h_0^{(d-1)}+t\h_1^{(d-1)}+\dots+t^{E-1}\h_{E-1}^{(d-1)}\right)
$$
vanishes identically in $t$, for $1\leq j\leq n$. We prove the following result. 

\begin{lemma}\label{lem:dim-bound}
We have $\dim V_E\leq (d-2)En$.
\end{lemma}

\begin{proof}
The  condition in the definition of $V_E$ is equivalent to the system of polynomial
equations
$$
\sum_{\substack{
0\leq i_1,\dots,i_{d-1}<E
\\ i_1+\dots+i_{d-1}=\ell }} 
\Psi_j(\h_{i_1}^{(1)},\dots,\h_{i_{d-1}}^{({d-1})})=0,
$$
for $0\leq \ell\leq (d-1)(E-1)$ and $1\leq j\leq n$.
Define the diagonal $D\subset \A^{(d-1)En}$, via 
$$
D=\{\h_k^{(1)}=\h_k^{(2)}=\dots=\h_k^{(d-1)}: 0\leq k<E  \}.
$$
Then $D$ has affine dimension $En$. Moreover,  $D\cap V_E$ consists precisely of the points 
$(\h_0, \dots,\h_{E-1})\in \A^{En}$ for which 
$$
\sum_{\substack{
0\leq i_1,\dots,i_{d-1}<E
\\ i_1+\dots+i_{d-1}=\ell }} 
\Psi_j(\h_{i_1},\dots,\h_{i_{d-1}})=0,
$$
for $0\leq \ell\leq (d-1)(E-1)$ and $1\leq j\leq n$.
When  $\ell=0$ we deduce that 
$$
\Psi_j(\h_0,\dots , \h_0)=0,
$$
for $1\leq j\leq n$. 
But $\Psi_j(\x,\dots,\x)=d!\frac{\partial f}{\partial x_j}$ for $1\leq j\leq n$. Thus,
since $f$ is non-singular, the only solution to this system of equations is $\h_0=\0$. When $\ell=d-1$ we obtain the system of equations
$$
\Psi_j(\h_1,\dots , \h_1)=0,
$$
for $1\leq j\leq n$, since $\h_0=\0$ and the $\Psi_j$ are multilinear.  Thus also 
$\h_1=\0$. Proceeding in this way, for $\ell$ running through multiples of $d-1$, one ultimately concludes that $\h_0=\dots=\h_{E-1}=\0$, whence 
$\dim(D\cap V_E)=0$.
The affine dimension theorem  implies that any irreducible component of 
 $U\cap V$ has dimension at least $\dim U+\dim V-m$,  for any irreducible affine varieties $U,V\subset \mathbf A^m$. Clearly $D$ is irreducible and 
 $\dim(D\cap V_E)=\max_J \dim (D\cap J)$, where the maximum is over the irreducible components of $V_E$. 
 Hence, for any 
 irreducible component $J$ of $V_E$, we deduce that 
\begin{align*}
0&\geq \dim (D\cap J)\\
&\geq \dim J+\dim D- (d-1)En\\
&\geq \dim J- (d-2)En,
\end{align*}
from which the lemma follows.
\end{proof}

\subsubsection{Application of the shrinking lemma}

Recall the definition \eqref{eq:Nalpha} of $N(\alpha)$. 
We are now ready to prove the following general result, which gives us a means of bounding  the remaining term in Proposition \ref{pro:stone}.

\begin{lemma}\label{lemma:general_bound_exp_sums} 
Let $\alpha = \frac{h_1}{h_2} + \theta$ such that $\ord(\theta) <0$. Let  $\rho = \deg(h_2)$ and $\psi = \ord ( \theta)$. 
Assume $s\geq 0$ to be an integer chosen such that 
\begin{enumerate} \item We have $-E-1 - (d-1)s < -\rho$ and $(d-1)(E-1-s) + \psi < -\rho$. 
\item We either have $(d-1)(E-1-s) <\rho$ or $-E - (d-1)s - \psi \leq \rho$. 
\end{enumerate}
Then 
$$\dim N(\alpha) \leq (d-2)En + sn.$$ 
\end{lemma}

\begin{proof} 
We  adopt the notation in 
Section \ref{sec:gon} in the proof of this result.  
The variety $N(\alpha)$ admits a projection $N(\alpha)\to \A^{E(d-2)}$,  along the 
coordinates $(\u^{(1)},\dots,\u^{(d-2)})$. 
Consider the linear forms $L_j(\u)=\alpha\Psi_j(\u^{(1)},\dots,\u^{(d-2)},\u)$, for $1\leq j\leq n$. 
These form a symmetric $n\times n$ matrix $\mathbf U$ (depending on 
$\u^{(1)},\dots,\u^{(d-2)}$ and $\alpha$).
The  fibre above  a point 
$(\u^{(1)},\dots,\u^{(d-2)})\in  \A^{E(d-2)}$ is 
equal to the  space 
whose dimension is taken in the definition of 
$\nu(\sfl_{E,E}(\mathbf{U}),0)$ that is 
discussed in Remark \ref{rem:5ways}.
Taking $a=b=E$, 
an application of  Lemma \ref{new-geometry}  now yields
$$
\nu(\sfl_{E,E}(\mathbf{U}),0)\leq ns+
\nu(\sfl_{E-s,E+s}(\mathbf{U}),-s),
$$
for any integer $s\geq 0$.
But,  according to Remark \ref{rem:5ways}, 
$\nu(\sfl_{E-s,E+s}(\mathbf{U}),-s)$ is just the dimension of the space of $\u\in  \Poly_{< E-s}^n$ for which 
$\ord\left\{ \alpha\Psi_j(\underline{\u})\right\}<-E-s$, for all $j\in \{1,\dots,n\}$. Hence 
$$
\dim 
N(\alpha)\leq ns + \dim N_s'(\alpha),
$$
for any integer $s\geq 0$, where
$$
N_s'(\alpha)=\left\{ \underline{\u}\in \left(\Poly_{< E}^n\right)^{d-2}\times
 \Poly_{< E-s}^n:
 \ord\left\{ \alpha\Psi_j(\underline{\u})\right\}<-E-s, 
 \forall j\in \{1,\dots,n\}
  \right\}.
$$ 
Repeating this argument $d-2$ times, once for each of the remaining projections, we 
 conclude that  
\begin{equation}\label{eq:cup}
\dim N(\alpha)\leq (d-1)ns + \dim N_s(\alpha),
\end{equation}
for any integer $s\geq 0$, where
$$
N_s(\alpha)=\left\{ \underline{\u}\in \left(\Poly_{< E-s}^n\right)^{d-1}:
 \ord\left\{ \alpha\Psi_j(\underline{\u})\right\}<-E-(d-1)s, 
 \forall j\in \{1,\dots,n\}
  \right\}.
$$

Our goal is to choose $s$, depending on $\rho$ and $\psi$, such that any $\ul{\u}$ appearing in the definition of $N_s(\alpha)$ must in fact satisfy $\Psi_j(\ul{\u}) = 0$ for all $j\in \{1,\dots,n\}$. 
Let $\ul{\u}\in N_s(\alpha)$ and let $j\in \{1,\dots,n\}$. Put $m = \Psi_j(\ul{\u})\in k[t]$. 
Each coordinate of $\ul{\u}$ is a polynomial of degree at most $E-s-1$ in $t$.  
Since $\Psi_j$ is a multilinear form of total degree $d-1$ in the variables $\u^{(1)},\dots,\u^{(d-1)}$, 
it follows that 
$\deg(m) \leq (d-1)(E-s-1).$
Write $\{\alpha m\} = \left\{ \frac{h_1}{h_2}m + \theta m\right\}$, and note that we have the upper bound $$\ord\{ \theta m \} \leq \deg(m) + \ord(\theta) \leq (d-1)(E-s-1) + \psi.$$

We will proceed in two steps. First we will find conditions on $s$ such that $h_2$ should divide $m$. For this, it suffices to bound $\ord \left\{ \frac{h_1}{h_2}m\right\}$ by something smaller than $-\rho$.   We write
$$\ord \left\{ \frac{h_1}{h_2}m\right\} = \ord \{ \alpha m - \theta m\} \leq \max\left( \ord \{ \alpha m\},\ord \{\theta m\}\right).
$$
We have seen that 
$\ord\{ \theta m \} \leq  (d-1)(E-s-1) + \psi$. Moreover, since $m=\Psi_j(\underline{\u})$, it follows from the definition of $N_s(\alpha)$ that 
$\ord\left\{ \alpha m\right\}\leq -E-1-(d-1)s$.
Hence we
get
$$\ord \left\{ \frac{h_1}{h_2}m\right\}\leq  \max \left( -E-1 - (d-1)s,(d-1)(E-s-1) + \psi\right).$$
Thus, we see that condition $(1)$ in the statement of the lemma forces $\ord \{ \frac{h_1}{h_2}m\}< -\rho,$
and therefore forces $h_2$ to divide $m$. 

The second step consists in finding conditions which would ensure that $m = 0$. Given that $h_2$ divides $m$, one possibility is just to force $\deg(m) < \rho$. This is implied by the condition $$(d-1)(E-1-s) < \rho,$$ which is the first inequality in part $(2)$ of the lemma. 
Alternatively, note that since $h_2$ divides $m$, we have $$\ord\{\theta m\} = \ord\{\alpha m\} < -E-(d-1)s.$$ On the other hand, $\ord (\theta m) \leq (d-1)(E-s-1) + \psi$, and the latter is negative since condition $(1)$ is satisfied. This means that 
$$\ord(\theta m) = \ord \{\theta m\}  < -E - (d-1) s$$
and therefore $\deg (m) < -E - (d-1) s - \psi.$ Thus, if $-E-(d-1)s - \psi \leq \rho$, which is the second inequality in part $(2)$ of the lemma, then $m =0$. 

We may conclude that if $s$ is chosen so that conditions $(1)$ and $(2)$ in the lemma are satisfied, then 
$$N_s(\alpha) = \left\{ \underline{\u}\in \left(\Poly_{< E-s}^n\right)^{d-1}:
\Psi_j(\underline{\u})=0,  \forall j\in \{1,\dots,n\}\right\}.
$$ 
 Lemma \ref{lem:dim-bound} gives 
$\dim N_s(\alpha)\leq (d-2) (E-s)n$, and so the statement  follows from \eqref{eq:cup}.
\end{proof}

\begin{remark} Assume $\frac{h_1}{h_2} = 0$, so that  $\rho = 0$. Then we may take
$$ s = \max\left( 0, E-1 +  \left\lceil\frac{1 + \psi}{d-1}\right\rceil, \min\left(E, -\left\lfloor \frac{E + \psi}{d-1}\right\rfloor\right)\right).$$ 
\end{remark}
\begin{remark}\label{rem:theta=0}
Assume that $\alpha\in k(t)$, so that $\theta = 0$ and $\psi = -\infty$. Then we may take
$$s = \max\left(0, \left\lceil\frac{\rho - E}{d-1}\right\rceil, E-1 - \left\lfloor \frac{\rho-1}{d-1}\right\rfloor\right).$$
\end{remark}

\section{The motivic major arcs}\label{sec:major}
\subsection{Definition}
Recall that the motivic analogue of the space $\T$ occurring in the function field circle method is the space $\A^{(-1,-de-2)} $.  It is an affine space of dimension $de+1$, representing elements of $\T$ modulo $t^{-de-2}k[[t^{-1}]]$. 
Fix a positive parameter $\gamma >0$. For every element $\beta=(b_1,\dots,b_{de+1})\in \A^{(-1,-de-2)}$ and every integer $m\geq 0$, we define the matrices
$$A_{\beta,m}  = \left( \begin{array}{ccc} b_1 & \dots & b_{m+1} \\
 \vdots& \ddots &\vdots  \\
b_{de+1 - \gamma -m}& \dots & b_{de+1-\gamma}\end{array} \right) 
$$
and
$$A^{*}_{\beta,m} = \left( \begin{array}{ccc} b_1 & \dots & b_{m} \\
\vdots& \ddots &\vdots  \\
b_{de+1 - \gamma -m}& \dots & b_{de-\gamma}\end{array} \right),
$$
the latter being obtained from the former by removing the last column. Note that $\ker A_{\beta,m}$ is generated by a vector $(c_0,\dots,c_m)$ satisfying $c_m\neq 0$ if and only if $A^{*}_{\beta,m}$ has full rank. We define the subspaces $M_{m,\gamma}$ of $\A^{(-1,-de-2)}$ for $m\geq -1$ inductively in the following way. Let $M_{-1,\gamma} = \varnothing$, and  let 
$$
M_{m,\gamma} = M_{m-1,\gamma}\cup \left\{(b_1,\dots,b_{de+1}): \rk A_{\beta,m} \leq m,\ \rk A^{*}_{\beta,m} = m\right\}.
$$
Denote by $M^{*}_{m,\gamma}$ the second set in this union. By the discussion in Section \ref{sect:approximation}, we see that $(b_1,\dots,b_{de+1})$ is an element of $M^{*}_{m,\gamma}$ if and only if there exists $h_2$ of degree exactly $m$ and $h_1$ of degree $<m$ such that 
$\alpha = b_1t^{-1} + \dots + b_{de+1}t^{-de-1} $ satisfies \begin{equation} \label{eq:order_condition_with_m}\ord(h_2\alpha -h_1)\leq -de-2 + \gamma +m,
\end{equation} 
and such that there do not exist $h_1,h_2$ such that $\deg(h_1) < \deg (h_2) < m$ satisfying (\ref{eq:order_condition_with_m}). 
In particular, one has
$\alpha= \frac{h_1}{h_2} + \theta$
with $\ord \theta \leq -de-2 + \gamma$. Note moreover that if $(b_1,\ldots,b_{de+1})\in M_{m,\gamma}^{*}$, then $m$ is minimal such that $\alpha= \frac{h_1}{h_2} + \theta$ with $\deg(h_1)< \deg(h_2) = m$ and $\ord(\theta) < -de-2 + \gamma$. Indeed, if we had polynomials $h'_1,h'_2$ such that $\deg(h'_1) < \deg(h'_2) = m'$ for some $m'< m$, and we would have
$$\ord(h'_2 \alpha - h'_1) \leq -de-2 + \gamma + m'\leq -de-2 + \gamma + m,$$
which would contradict the condition above. In particular, this implies that these polynomials $h_1$ and $h_2$ are relatively prime.

If we choose $h_2$ to be monic, we see that the choices of $h_1,h_2,\theta$ for $(b_1,\ldots,b_{de+1})\in M^{*}_{m,\gamma}$ are unique ($h_2$ coming from the unique generator of $\ker A_{\beta,m}$ with last coordinate equal to 1), and that the coefficients of $h_1, h_2$ and $\theta$ are algebraic functions of the coefficients of $\alpha$. 

\begin{lemma}\label{lem:6.1}
 For any integers $\gamma>0$ and $m\geq 0$, we have
$$
M_{m,\gamma} = \left\{ (b_1,\dots,b_{de+1}): 
\begin{array}{l}
b_1t^{-1} + \dots + b_{de+1}t^{-de-1} = \frac{h_1}{h_2} + \theta\\
\gcd(h_1,h_2)=1 \text{ and $h_2$ monic}\\
\deg h_1 < \deg h_2 \leq m \\
\ord\, \theta\leq -de -2 + \gamma
\end{array}
\right\}.
$$
\end{lemma}
\begin{proof} We proceed by induction on $m$. The statement is clearly true for $m=0$. Assume it is true for $m-1$. Then by the above remark the inclusion of $M_{m,\gamma}$ in the right hand side is obvious. Conversely, assume that $b_1t^{-1} + \dots + b_{de+1}t^{-de-1} = \frac{h_1}{h_2} + \theta$ as in the statement of the lemma, and where we choose $h_2$ to be of minimal possible degree. If $\deg h_2 \leq m-1$, then we may conclude by induction that $(b_1,\dots,b_{de+1})\in M_{m-1,\gamma}\subset M_{m,\gamma}$. If $\deg h_2 = m$, then we have $(b_1,\dots,b_{de+1})\in M^{*}_{m,\gamma}\subset M_{m,\gamma}$, and so the inverse inclusion holds as well. 
\end{proof}

This characterisation of $M_{m,\gamma}$ is what we will use in practice.  
It is similar to the set $A_m^{de+1}$ that was discussed in Remark \ref{rem:Am}, the chief difference being that our bound on $\ord \theta$ does not depend on the degree of $h_2$.

By definition, we have

$$M_{m,\gamma} = \bigsqcup_{m'= 0}^{ m}M^*_{m',\gamma},$$ where using the description of Lemma \ref{lem:6.1},  $M^*_{m',\gamma}$ corresponds to the subset where $h_2$ has degree exactly $m'$. 
We  take the major arcs to be 
\begin{equation}\label{eq:Major} 
\mathfrak{M}=M_{e+1-\gamma,\gamma} = \bigsqcup_{0\leq m'\leq e+1-\gamma} M_{m',\gamma}^{*}.
\end{equation}
Bearing in mind Lemma \ref{lem:active}, 
we  then take 
\begin{equation}\label{eq:calculus}
N_{\maj} = \sum_{0\leq m'\leq e+1-\gamma} N_{\maj} (m'),
\end{equation}
where
\begin{equation}\label{eq:calculus-m}
N_{\maj}(m') = \LL^{-de-1}\left[ \Poly_{\leq e}^n \times M_{m',\gamma}^*,\res\left(\left(\frac{h_1}{h_2}+\theta\right)f(g_1,\dots,g_n)\right)\right],
\end{equation}
which is viewed as an element of $\expp_{\Poly_{\leq e}^n \times M_{m',\gamma}^*}.$ 
(We shall see that the condition $m'\leq e+1-\gamma$ arises naturally during the course of  the argument.)

For now we merely assume that  $\gamma\leq e$, although  
we shall  ultimately take 
\begin{equation}\label{eq:gamma}
\gamma=\left\lceil \frac{e+1}{2}\right\rceil.
\end{equation}
For any $m \geq 0$, let
\begin{equation}\label{eq:define_Bm}
B_{m} = ( \Poly_{<m}\times \MPoly_{m})_{*}
\end{equation}
denote the space of pairs $(h_1,h_2)$ of coprime polynomials such that $\deg h_1 < \deg h_2=m$, with $h_2$ monic.  Then we may define the ``exponential sum'' 
\begin{equation}\label{eq:def-Smf}
S_{m}(f) = \left[\Poly_{<m}^n\times B_{m}, \res \left(\left(\frac{h_1}{h_2}\right)f(\bar{g}_1,\dots,\bar{g}_n)\right)\right].
\end{equation}
 There is a piecewise isomorphism 
 $$
 M_{m,\gamma}^*\simeq B_{m} \times \A^{(-de-2 + \gamma, -de-2)},$$ 
where $B_{m}$ is given by \eqref{eq:define_Bm}, which is 
 obtained  
 by sending $\alpha$ to $(h_1,h_2,\theta)$. 
This induces natural morphisms
$$
\expp_{B_{m}} \to \expp_{M_{m,\gamma}^*}
$$
and 
\begin{equation}\label{eq:morphism_to_majarcsA}\expp_{\A^{(-de-2 + \gamma, -de-2)}} \to \expp_{M_{m,\gamma}^*}.\end{equation}
With these, the sum $S_{m}(f)$ can  be viewed as an element of $\expp_{M_{m,\gamma}^*}$, for each $m$.

\subsection{Contribution of $M_{m,\gamma}^*$}

We  start by computing the contribution 
\eqref{eq:calculus-m} for fixed $m\leq e+1$.   Note that $N_{\maj}(m)$ is viewed as an element of $\expp_{\Poly_{\leq e}^n \times M_{m,\gamma}^*}$, but we  allow ourselves the slight abuse of notation of also denoting by $N_{\maj}(m)$ the image of this element in some other Grothendieck rings.  
By Euclidean division, for each $1\leq i\leq n$,  
we may write $g_i = h_2 q_i + \bar{g}_i$,  which gives us an isomorphism
$$\Poly_{\leq e}^{n} \simeq \Poly_{\leq e-m}^n \times \Poly_{<m}^n,$$
 sending $(g_1,\dots,g_n)$ to $(q_1,\dots,q_n, \bar{g}_1,\dots \bar{g}_n)$. In terms of classes in the Grothendieck ring, this implies
\begin{align*}
N_{\maj}(m) 
&= \LL^{-de-1}\left[\Poly_{<m}^n\times M_{m,\gamma}^*, \res \left(\left(\frac{h_1}{h_2}\right)f(\bar{g}_1,\dots,\bar{g}_n)\right)\right] \\
&\quad \cdot\left[\Poly_{\leq e-m}^n\times \Poly_{<m}^n \times M_{m,\gamma}^*, \res\left(\theta f(\bar{g}_1 + h_2q_1,\dots,\bar{g}_n + h_2q_n)\right)\right],
\end{align*}
where the product is taken in $\expp_{\Poly_{< m}^n\times M_{m,\gamma}^*}$.
The following result shows that  there is no dependence on $\bar{g}_1,\dots,\bar{g}_n$ in the second factor, under our assumption on $m$.

\begin{lemma}\label{lem:order_small}
Assume that $m\leq e+1-\gamma$. 
Then 
$$
\ord(\theta(f(\bar{g}_1 + h_2q_1,\dots,\bar{g}_n + h_2q_n) -f(h_2q_1,\dots,h_2q_n))) < -1.
$$
\end{lemma}

\begin{proof}
We note that the difference $f(\bar{g}_1 + h_2q_1,\dots,\bar{g}_n + h_2q_n) -f(h_2q_1,\dots,h_2q_n)$ is a sum of monomials of degree $d$ in $\bar{g}_1,\dots,\bar{g}_n, h_2q_1,\dots,h_2q_n$. Moreover, the degree in $\bar{g}_1,\dots,\bar{g}_n$ is at least 1. We know that $\ord h_2 q_i \leq e$ and $\ord \bar{g}_i \leq m-1\leq e$.  Thus, using the bound $m-1$ for at least one of the $\bar{g}_i$-factors of each monomial, and the bound~$e$ for all the other factors, we get
$$
\ord (f(\bar{g}_1 + h_2q_1,\dots,\bar{g}_n + h_2q_n) -f(h_2q_1,\dots,h_2q_n)) \leq (d-1)e + m-1.
$$
Moreover, we know that $\ord (\theta)\leq -de-2 + \gamma$. Combining the two estimates, we get
\begin{align*}
\ord \left(\theta (f(\bar{g}_1 + h_2q_1,\dots,\bar{g}_n + h_2q_n) -f(h_2q_1,\dots,h_2q_n) )\right)
&\leq 
(d-1)e + m-1 -de-2 +\gamma\\
&\leq -e-3+m + \gamma.
\end{align*}
This is  bounded by $-2$ if $m + \gamma \leq e + 1$. 
 \end{proof}
 
We may now  write the contribution of $M_{m,\gamma}^*$ as
$$
N_{\maj}(m) =
 \LL^{-de-1} S_{m}(f) \cdot[\Poly_{\leq e-m}^n\times M_{m,\gamma}^*, \res(\theta f(h_2q_1,\dots,h_2q_n)],$$
now viewed in $\expp_{M_{m,\gamma}^*}$,
where 
$S_{m}(f)$ is given by  \eqref{eq:def-Smf}.
We apply an averaging argument to the rightmost factor, in order to remove the dependence in $m$ completely.  
We first of all re-introduce a factor $\Poly_{<m}^n \simeq \A^{nm}$ to obtain
\begin{align*}
[\Poly_{\leq e-m}^n\times &M_{m,\gamma}^*, \res(\theta f(h_2q_1,\dots,h_2q_n)]_{M_{m,\gamma}^*}
\\& =\LL^{-nm}[\Poly_{\leq e-m}^n\times \Poly_{<m}^n\times  M_{m,\gamma}^*, \res(\theta f(h_2q_1,\dots,h_2q_n))]_{M_{m,\gamma}^*} .
\end{align*}
Now we use  Lemma \ref{lem:order_small},  together with  the isomorphism coming from Euclidean division, in order to get that the right hand side is
\begin{align*}
&=\LL^{-nm}[\Poly_{\leq e-m}^n\times \Poly_{<m}^n\times  M_{m,\gamma}^*, \res(\theta f(\bar{g}_1 + h_2q_1,\dots,\bar{g}_2 + h_2q_n))]_{M_{m,\gamma}^*}\\
 &= \LL^{-nm}[\Poly_{\leq e}^n\times M_{m,\gamma}^*, \res(\theta f(g_1,\dots,g_n))]_{M_{m,\gamma}^*} \\
 &= \LL^{-nm}[\Poly_{\leq e}^n\times \A^{(-de-2+ \gamma, -de-2)},  \res(\theta f(g_1,\dots,g_n))]_{M_{m,\gamma}^*},
\end{align*} 
where the latter is viewed as an element of $\expp_{M_{m,\gamma}^*}$ via the morphism~(\ref{eq:morphism_to_majarcsA}).
Pushing everything to $\expp_{\C}$, we end up with
\begin{align*}
N_{\maj}(m) &= \LL^{-de-1} S_{m}(f) \LL^{-nm} [\Poly_{\leq e}^n \times\A^{(-de-2+ \gamma, -de-2)}, \res(\theta f(g_1,\dots,g_n))] \\
&= S_{m}(f) \LL^{-nm} \int_{\A^{(-de-2+ \gamma, -de-2)}} [\Poly_{\leq e}^n \times \A^{(-de-2+ \gamma, -de-2)}, \res(\theta f(g_1,\dots,g_n))] d\theta ,
\end{align*} 
where the integral is the one from Section \ref{sect:motivic_integrals}. 
Inserting this into \eqref{eq:calculus}, we are therefore led to the following result. 

\begin{prop}\label{prop:major_arcs_first_expression}
We have 
$$
N_{\maj} =  \left( \sum_{0\leq m\leq e+1-\gamma} S_{m}(f) \LL^{-nm} \right)  \int U(\theta)d\theta,
$$
in $\expp_{\C}$, where
\begin{equation}\label{eq:U(theta)}
U(\theta) := [\Poly_{\leq e}^n \times \A^{(-de-2+\gamma, -de-2)}, \res(\theta f(g_1,\dots,g_n))]\in \expp_{\A^{(-de-2+\gamma, -de-2)}}.
\end{equation}
\end{prop}

The aim is now to analyse the two remaining factors to obtain motivic analogues of the singular series and the singular integral. 

\subsection{Singular series}

It is now time to treat the truncated singular series in Proposition~\ref{prop:major_arcs_first_expression}.
Recalling the definition \eqref{eq:gamma} of $\gamma$, we have 
$$
\sum_{0\leq m\leq e+1-\gamma} S_{m}(f) \LL^{-nm}=
\sum_{0\leq m\leq \lfloor \frac{e+1}{2}\rfloor} S_{m}(f) \LL^{-nm}.
$$
We begin with a useful upper bound, which  also proves valuable in the treatment of the minor arcs.

\begin{lemma}\label{lem:jens}
Let  $\Delta=\lfloor \frac{e+1}{2}\rfloor$.  Then 
$$
\max_{m\geq \Delta}\left( 2^{d} m - \left\lfloor \frac{m}{d-1}\right\rfloor n \right)\leq 
\left\lfloor \frac{e+1}{2d-2}\right\rfloor (2^{d}(d-1) - n) +2^d(d-1).
$$
\end{lemma}

\begin{proof}
To begin with, we note that 
\begin{align*}
2^{d} m - \left\lfloor \frac{m}{d-1}\right\rfloor n&\leq 
2^{d}\left(\left\lfloor \frac{m}{d-1}\right\rfloor + 1\right) (d-1) - \left\lfloor \frac{m}{d-1}\right\rfloor n\\
&\leq 
\left\lfloor \frac{m}{d-1}\right\rfloor (2^{d}(d-1) - n) +2^d(d-1).
\end{align*}
Since $n > 2^{d}(d-1)$, the maximum is reached when $m= \Delta= \lfloor \frac{e+1}{2}\rfloor$
is minimal.

If  $2\mid e+1$, it now follows that 
$$
\max_{m\geq \Delta}
\left(2^{d} m - \left\lfloor \frac{m}{d-1}\right\rfloor n\right)
\leq 
\left\lfloor \frac{e+1}{2d-2}\right\rfloor (2^{d}(d-1) - n) +2^d(d-1),
$$
as claimed in the lemma.  
Suppose next that $2\nmid e+1$, so that $\Delta=\frac{e}{2}$.
Then 
\begin{align*}
2^{d} m - \left\lfloor \frac{m}{d-1}\right\rfloor n
&\leq 
\left\lfloor \frac{e+2}{2d-2}\right\rfloor (2^{d}(d-1) - n) +2^d(d-1)\\
&\leq 
\left\lfloor \frac{e+1}{2d-2}\right\rfloor (2^{d}(d-1) - n) +2^d(d-1),
\end{align*}
if $m\geq \Delta+1$, whereas 
$$
2^{d} m - \left\lfloor \frac{m}{d-1}\right\rfloor n
=2^{d-1} e - \left\lfloor \frac{e}{2d-2}\right\rfloor n
$$
if $m=\Delta$.
In order to complete the proof, it remains to prove that 
\begin{equation}\label{eq:rtp}
2^{d-1} e - \left\lfloor \frac{e}{2d-2}\right\rfloor n\leq \left\lfloor \frac{e+1}{2d-2}\right\rfloor (2^{d}(d-1) - n) +2^d(d-1).
\end{equation}
Writing $e=(2d-2)q+r$ for $r<2d-2$, we must have  $r\leq 2d-4$, since $e$ is even. 
But then 
$$
 \left\lfloor \frac{e}{2d-2}\right\rfloor=
 \left\lfloor \frac{e+1}{2d-2}\right\rfloor=q,
 $$
 whence \eqref{eq:rtp} is equivalent to 
 $2^{d-1}r\leq 2^d(d-1)$, which is self-evident. 
 \end{proof}

\subsubsection{Convergence of the series}

Recall from  \eqref{eq:def-Smf} that 
$$
S_{m}(f) = \left[\Poly_{<m}^n\times B_{m}, \res \left(\left(\frac{h_1}{h_2}\right)f(\bar{g}_1,\dots,\bar{g}_n)\right)\right],
$$
where $B_{m}$ is given by \eqref{eq:define_Bm}.
We
can use Proposition \ref{pro:stone} to bound this sum, obtaining $$
w_{B_{m}} (S_{m}(f)) \leq \frac{\max_{(h_1,h_2)\in B_m}\dim N(h_1/h_2) + nm(2^{d-1}-(d-1))}{2^{d-2}}.
$$
We apply Lemma~\ref{lemma:general_bound_exp_sums} with $\theta = 0$ (so that $\psi = -\infty$) and $\rho =E= m$. It follows from Remark \ref{rem:theta=0} that we can take
$$
s=m-1-\left\lfloor \frac{m-1}{d-1}\right\rfloor.
$$
Using this, it therefore follows  that 
 \begin{align*}
 \dim N(h_1/h_2)
 &\leq (d-2)mn + \left(m-1-\left\lfloor \frac{m-1}{d-1}\right\rfloor
 \right)n\\
&\leq (d-1)mn - 
\left(1+\left\lfloor \frac{m-1}{d-1}\right\rfloor
 \right)n,
\end{align*}
whence
 $$w_{B_{m}} (S_{m}(f)) \leq \frac{2^{d-1}n m - 
 (1+\lfloor \frac{m-1}{d-1}\rfloor)n
 }{2^{d-2}} =  2nm- \frac{(1+\lfloor \frac{m-1}{d-1}\rfloor)n}{2^{d-2}}.
 $$
 From this, noticing that $\dim B_{m} = 2m$,  property (\ref{item:weight-property-pushforward}) of Proposition \ref{prop:weight-properties} yields
 \begin{equation}\label{eq:weight-of-complete-sum}
 w (S_{m}(f)) - 2nm 
 \leq 
 4m - \frac{(1+\lfloor \frac{m-1}{d-1}\rfloor)n}{2^{d-2}}.
 \end{equation}

We now compute the radius of convergence of the series $\sum_{m\geq 0}S_m(f)T^m$, in the sense of Section \ref{sect:convergence_power_series}. 
We have   $
1+\left\lfloor \frac{m-1}{d-1}\right\rfloor \geq \frac{m}{d-1},
 $
 for any $m\geq 0$. Hence
$$
\frac{w(S_{m}(f))}{2m}\leq n +  \left(\frac{2^{d}(d-1)-n}{2^{d-1}(d-1)}\right) ,
$$
and so as soon as the condition $n>2^{d}(d-1)$ is satisfied, we get
$$\limsup_{m\to \infty} \frac{w(S_{m}(f))}{2m}\leq n + \text{negative constant}.$$
 We may deduce that the radius of convergence of the series 
$\sum_{m\geq 0}S_m(f)T^m$  is bounded by $n-\eps$ for some $\eps >0$, and  the series converges for $|T| < \LL^{-n+ \eps}$. In particular, the \textit{motivic singular series}
\begin{equation}\label{eq:SS}
 \mathfrak{S}(f) := \sum_{m\geq 0} S_{m}(f) \LL^{-nm}
 \end{equation}
 is well-defined as an element of $\widehat{\expp_{\C}}$

 \subsubsection{Replacing by full series}
These bounds also allow us to evaluate the error made by replacing the finite sum
$$
\sum_{0\leq m\leq \lfloor \frac{e+1}{2}\rfloor} S_{m}(f) \LL^{-nm}
$$
 by the full sum $\mathfrak{S}(f)$ defined in \eqref{eq:SS}.
 For this, it suffices to  bound the weight of a term
$
S_{m}(f) \LL^{-mn}$
for $m\geq \lfloor \frac{e+1}{2}\rfloor+1$.  It follows from 
\eqref{eq:weight-of-complete-sum} that 
$$
 w (S_{m}(f) \LL^{-mn})
\leq
\frac{1}{2^{d-2}}
\left(2^d (m-1)-
\left(1+ \left\lfloor \frac{m-1}{d-1}\right\rfloor\right) n\right)+4-\frac{n}{2^{d-2}}.
$$
Hence  Lemma \ref{lem:jens} implies that 
$$
\max_{m\geq \lfloor \frac{e+1}{2}\rfloor+1}   w (S_{m}(f) \LL^{-mn})\leq 
\frac{\left\lfloor \frac{e+1}{2d-2}\right\rfloor (2^{d}(d-1) - n) +2^d(d-1)}{2^{d-2}}
+4-\frac{n}{2^{d-2}}.
$$
Property (\ref{item:weight-property-sum}) of Proposition \ref{prop:weight-properties} allows us to  conclude as follows. 

\begin{lemma}\label{lem:train}
Recall the definition \eqref{eq:tilde-nu} of $\tilde \nu$.
Then 
  $$
  w \left(\mathfrak{S}(f) - \sum_{0 \leq m\leq \lfloor \frac{e+1}{2}\rfloor} S_{m}(f) \LL^{-nm}\right) \leq 4
  -\tilde\nu\left(1+ \left\lfloor \frac{e+1}{2d-2}\right\rfloor \right).
  $$
  \end{lemma}
  
  \begin{remark}\label{rem:weight-SS}
  We can also conclude from our argument that 
   $\mathfrak{S}(f)$ has weight zero, equal to $1$ plus an element of $\widehat{\expp_{\C}}$ of negative weight. 
 To see this we use  
 $1+\left\lfloor \frac{m-1}{d-1}\right\rfloor \geq \frac{m}{d-1},$ for any $m\geq 1$, so that 
 \eqref{eq:weight-of-complete-sum} yields
 \begin{equation}\label{eq:train}
  w (S_{m}(f) \LL^{-nm}) \leq 
  4m - \frac{nm}{2^{d-2}(d-1)} = \left(\frac{2^{d}(d-1)-n}{2^{d-2}(d-1)}\right)m .
 \end{equation}
The condition $n> 2^{d}(d-1)$ therefore implies that  
$$w(\mathfrak{S}(f)-1) = w\left(\sum_{m \geq 1} S_{m}(f)\LL^{-mn}\right) \leq -\nu, 
$$
for a constant $\nu>0$.
\end{remark}
 
  \subsubsection{Motivic Euler product decomposition}\label{s:SS1}
  
  \begin{prop}[Factorisation property]
  \label{prop:factor}
   Let $m_1,m_2$ be positive integers, and let $V_{m_1,m_2}$ be the space of pairs of coprime monic polynomials $l_1,l_2$ such that $\deg(l_i) = m_i$ for $i = 1,2$. Let $\pi_i:V_{m_1,m_2}\to \MPoly_{m_i}$ be given by $(l_1,l_2)\mapsto l_i$, and let $\pi_{12}:V_{m_1,m_2}\to \MPoly_{m_1 + m_2}$ be given by $(l_1,l_2)\mapsto l_1l_2$. Then, there is an isomorphism 

$$\pi_{12}^*S_{m_1 + m_2}(f) \simeq \pi_1^* S_{m_1}(f) \times_{V_{m_1,m_2}} \pi_2^{*} S_{m_2}(f)$$
of varieties with exponentials over $V_{m_1,m_2}$ . Moreover, when $m_1 = m_2$, this isomorphism  commutes with  switching $l_1$ and $l_2$. 
  \end{prop}
  
  \begin{proof}
We follow the argument in  \cite[Lemma 3.5]{BS}.
Recall from \eqref{eq:def-Smf} that 
$$
S_{m}(f) = \left[\Poly_{<m}^n\times B_{m}, \res \left(\left(\frac{h_1}{h_2}\right)f({g}_1,\dots,{g}_n)\right)\right],
$$
for any positive integer $m$, where
$B_{m} = ( \Poly_{<m}\times \MPoly_{m})_{*}$.

To begin with, 
there is an isomorphism 
\begin{equation}\label{eq:isom1}
(\Poly_{<m_1+m_2})^n\to (\Poly_{<m_1})^n\times (\Poly_{<m_2})^n,
\end{equation}
which is given by sending a tuple of polynomials $(g_1,\dots,g_n)$ of degree $<m_1+m_2$ to the pair of tuples
$$
((g_{1,1},\dots,g_{n,1}),(g_{1,2},\dots,g_{n,2})),
$$ 
where each $g_{i,1}$ has degree $<m_1$ (respectively,
each $g_{i,2}$ has degree $<m_2$), and such that $g_i\equiv g_{i,1} \bmod l_1$ and 
$g_i\equiv g_{i,2} \bmod l_2$, for $1\leq i\leq n$. The inverse map is given by 
$$
((g_{1,1},\dots,g_{n,1}),(g_{1,2},\dots,g_{n,2})) \mapsto 
\left(l_2 (g_{i,1}l_2^{-1}\bmod  l_1) +l_1 (g_{i,2}l_1^{-1}\bmod  l_2)\right)_{1\leq i\leq n},
$$
where the inverses are understood to be modulo $l_1$ and $l_2$, respectively. Since $l_1$ and $l_2$ are coprime,  their inverses modulo each other are polynomial. Moreover, since they are monic it follows from  Euclid's algorithm  that the modulo operation is also polynomial.

We also obtain an isomorphism of $V_{m_1,m_2}$-varieties 
\begin{equation}\label{eq:B-iso}
B_{m_1+m_2} \times_{\pi_{12}} V_{m_1,m_2} \to 
B_{m_1} \times_{\pi_{1}} V_{m_1,m_2} \times_{\pi_{2}} 
B_{m_2},
\end{equation}
given by 
$$
\left((h_1,l_1l_2), (l_1,l_2)\right)\to \left((h_1l_2^{-1}\bmod{l_1}, l_1), (l_1,l_2), 
(h_1l_1^{-1}\bmod{l_2}, l_2)\right), 
$$
where inverses are again understood to be modulo $l_1$ and $l_2$, respectively. The  inverse mapping is given by
$$
\left((h_{1,1},l_1),(l_1,l_2), 
(h_{1,2},l_2)\right) \to 
\left( (h_{1,1}l_2+h_{1,2}l_1, l_1l_2), (l_1,l_2)\right).
$$
To check that the isomorphisms (\ref{eq:isom1}) and (\ref{eq:B-iso}) induce isomorphisms of varieties with exponentials, we compute, using the same notation:
\begin{align*}\res &
\left(
 \frac{h_{1,1}}{l_1 }  f\left(  g_{1,1},\dots,g_{n,1} \right)+\frac{h_{1,2}}{l_2 }  f\left(  g_{1,2},\dots,g_{n,2}\right)\right)\\
  & = 
  \res
\left(
 \frac{h_{1,1}}{l_1 }  f\left(  g_{1},\dots,g_{n} \right)+\frac{h_{1,2}}{l_2 }  f\left(  g_{1},\dots,g_{n}\right)\right)\\
 & = 
  \res
\left(
\frac{ h_{1,1} l_2 +h_{1,2} l_1}{l_1l_2} f\left(  g_{1},\dots,g_{n}\right)\right)\\
&=
\res\left( \frac{h}{l_1 l_2 }f\left(  g_{1},\dots,g_{n} \right)\right).
\end{align*}
The lemma follows on putting these together. 
  \end{proof}

  We can now write the series $\sum_{m\geq 0}S_m(f)T^m$ as a motivic Euler product. For this, 
  for every $m\geq 0$,
  we need to write  $S_{m}(f)$ as the $m$-th coefficient of a motivic Euler product; that is, as the sum of some configuration spaces. 
For  $i\geq 1$, define  
\begin{equation}\label{eq:def-Ui}
U_i(f)  = \left[\Poly_{<i}^n\times C_i , \res \left(\frac{h}{(t-x)^i}f(g_1,\dots,g_n)\right)\right],
\end{equation}
where $C_i = \{(h,x)\in \Poly_{<i}\times \A^1: h(x) \neq 0\}.$ We view this as an element of $\expp_{\A^1}$ via the second projection $C_i\to \A^1$. 
    Note that if $S_i(f)$ is viewed as an element of $\expp_{\MPoly_i}$ via the projection morphism $B_i \to \MPoly_i$, given by $(h_1,h_2)\mapsto h_2$, then  $U_i(f)$ is the restriction of $S_i(f)$ to the locus $\Delta_i \simeq \A^1$ of polynomials of the form $(t-x)^i$.

Recalling the definition of motivic Euler products from  Section \ref{sect:motivic_euler_prods}, we have the following result. 
  
  \begin{prop}
  \label{pro:EPD}
     We have the motivic Euler product decomposition
  \begin{equation}\label{eq:sing_series_euler_prod}\mathfrak{S}(f) = \prod_{x\in \A^1}\left(1 + \sum_{i\geq 1} U_i(f)_x T^i\right)\Big|_{T=\LL^{-n}}.
  \end{equation}
  \end{prop}
  
  \begin{proof} 
Recalling the definition
\eqref{eq:SS} of 
$\mathfrak{S}(f)$, it suffices to prove that 
$$
1 + \sum_{m\geq 1} S_m(f) T^m=
\prod_{x\in \A^1}\left(1 + \sum_{i\geq 1} U_i(f)_x T^i\right).
$$
The element $S_{m}(f)$ comes from a variety with exponential over $\MPoly_m$. On the other hand, there is an identification $\MPoly_m\simeq \Sym^m\A^1$ sending a monic polynomial $P = \prod_{i=1}^m(t-x_i)$ to the effective zero-cycle $\sum_{i=1}^mx_i$, and giving rise to a stratification of $\MPoly_m$ into pieces $\Conf^{\omega}(\A^1)$ indexed by partitions $\omega = (m_i)_i$ of $m$. We are going to exhibit an isomorphism 
  $$S_{m}(f)_{|\Conf^{\omega}(\A^1)} \to \Conf^{\omega}(U_i(f))_{i\geq 1}.$$
  
  We define the variety
  $$V_{\omega} = \{(l_{i,1},\dots,l_{i,m_i})_{i\geq 1}: l_{i,j}\ \text{relatively prime, monic},\ \deg l_{i,j} = i\} = \left(\prod_{i\geq 1}(\MPoly_i)^{m_i}\right)_{*}$$
  where we denote by $()_*$ the fact that we take only tuples of relatively prime polynomials. We also define $W_{\omega} \subset V_{\omega}$ by 
  $$W_{\omega} = \{(l_{i,j})_{i,j}\in V_{\omega}: 
  l_{i,j}\ \text{is of the form}\ (T-x_{i,j})^i\} = \left( \prod_{i\geq 1} \Delta_i^{m_i}\right)_{*}.
  $$
 Finally, we consider the morphisms
  $$\pi_{i,j}: V_{\omega}\to \MPoly_i,\ \ \ (l_{p,q})_{p,q} \mapsto l_{i,j}, $$
  which restrict to $\pi_{i,j}: W_{\omega} \to \Delta_i$,
  as well as 
  $$\pi_{\omega}: V_{\omega}\to \MPoly_m,\ \ \ (l_{p,q})_{p,q} \mapsto \prod_{i,j} l_{i,j}.$$
  
  By iterating the factorisation property in Proposition \ref{prop:factor}, we obtain an isomorphism
    $$
    \pi_{\omega}^* S_{m}(f) \simeq \prod_{i\geq 1}\prod_{1 \leq j \leq m_i} \pi_{i,j}^*S_{i}(f),
    $$
  which is compatible with any permutation of the $l_{i,j}$ with fixed $i$. 
  Restriction to $W_{\omega}$ now gives
  $$\pi_{\omega}^* S_{m}(f)_{| W_{\omega}} \simeq \prod_{i\geq 1}\prod_{1 \leq j \leq m_i} \pi_{i,j}^*U_{i}(f),$$
    There is a natural permutation action of $\mathfrak{S}_{\omega} = \prod_{i} \mathfrak{S}_{m_i}$ on $V_{\omega}$, which restricts to $W_{\omega}$, and such that via the isomorphism $\Delta_i \simeq \A^1$, the quotient $W_{\omega}$ is naturally identified with $\Conf^{\omega}(\A^1)$. Taking the quotient by this permutation action, we get the result. 
    \end{proof}

  \subsubsection{Interpretation as local densities}\label{s:SS2}
 
   \begin{lemma} \label{lem:factor}
   Let $x\in \A^1.$ Then, for every $N\geq 1$, we have 
  $$1 + \sum_{i=1}^N U_i(f)_{x}\LL^{-in} =\LL^{-N(n-1)} [ \Lambda_N(f,x)],$$
  where $ \Lambda_N(f,x)$ is given by \eqref{eq:lambda}.
  \end{lemma}

  \begin{proof}
Recall the definition of $U_i$ from \eqref{eq:def-Ui}. We are going to reason completely analogously to the passage from the singular series to local densities in the classical circle method, which relies on some simple properties of Ramanujan sums. (See, for example, the proof of Lemma 2.14 in \cite{Browning_book}.) We work over $\Poly_{<i}^n$, which we may partition into three pieces: the first piece is $\Lambda_i(f,x)$, the second one is 
 $$\Lambda'_i(f,x) = \{\g \in \Poly_{<i}^n: f(\g)\equiv 0\bmod{(t-x)^{i-1}}\ \text{but}\ f(\g) \not\equiv 0 \bmod{(t-x)^{i}}\},$$
 and the third one is 
 $$\Lambda''_i(f,x) = \{\g \in \Poly_{<i}^n: f(\g)\not\equiv 0 \bmod{(t-x)^{i-1}}\}.$$

 For $\g\in \Lambda_i(f,x)$, the residue is zero and so the fibre of $U_i(f)_x$ above $\g$ is simply given by the class of $C_i$, which is easily seen to be equal to $\LL^{i-1}(\LL-1)$, since by sending $h\in C_i$ to the tuple given by the coefficients of its Taylor expansion at $x$, we have an isomorphism \begin{equation}\label{eq:isoC_i}C_i\simeq (\A^{1}-\{0\})\times \A^{i-1}.\end{equation}
 
 We now take $\g\not\in \Lambda_i(f,x)$. Thus, we may write 
$f(\g) = (t-x)^{i- \ell}p(t)$ for some integer $\ell \geq 1$ and some polynomial $p$ in $t$ such that $p(x)\neq 0$. The fibre of $U_i(f)_x$ at $\g$ then takes the form
$$\left[C_i,\res\left(\frac{h(t)p(t)}{(t-x)^{\ell}}\right)\right]$$
in $\expp_{\kappa(\g)}$ where $\kappa(\g)$ is the residue field of $\g$. By using the  isomorphism (\ref{eq:isoC_i}) and denoting by $\a$ the vector of the coefficients of orders $0,\ldots,\ell - 1$ of the Taylor expansion of $h$ at $x$, this can be rewrittten as
 $$\LL^{i-\ell}\left[(\A^{1}-\{0\})\times \A^{\ell-1}, \res \left(\frac{\lambda_0(\a) + \lambda_1(\a)(t-x) + \cdots + \lambda_{\ell-1}(\a)(t-x)^{\ell-1}}{(t-x)^{\ell}}\right)\right],$$
 where $\lambda_0,\ldots,\lambda_{\ell-1}$ are linear forms in the coordinates of $\a$. Note that one of the coefficients of $\lambda_{\ell-1}$ is $p(x)$, and therefore $\lambda_{\ell-1}$ is non-zero. Observe also that for any constant $c$, and for any integer $j\geq 1$, 
 $$\res\left(\frac{c}{(t-x)^j}\right) = \left\{\begin{array}{ll} 0& \text{if}\ j\geq 2,\\
 c& \text{if}\ j = 1.\end{array}\right.$$
 
 We now distinguish between the case when $\g\in \Lambda'_i(f,x)$ and the case when $\g\in \Lambda''_i(f,x)$. In the former case, we have $\ell = 1$, and our class becomes equal to 
 $$\LL^{i-1}[\A^1-\{0\}, \lambda_{0}(\a)] = \LL^{i-1}[\A^1,\lambda_{0}(\a)] -\LL^{i-1} = -\LL^{i-1}$$
 by Lemma \ref{lemma:linear_form}. In the latter case, we have $\ell\geq 2$, and by a similar argument we find that our class vanishes. 
  Thus, we have shown  the  relation
  $$
  U_i(f)_x = \LL^{i-1}(\LL-1) [\Lambda_i(f,x)] - \LL^{i-1}[\Lambda'_{i}(f,x)],
  $$
  in $\expp_{\Poly_{<i}^n}.$

Finally, on observing that $[\Lambda'_i(f,x)] = \LL^{n}[\Lambda_{i-1}(f,x)] - [\Lambda_i(f,x)]$, the left-hand side in the lemma  leads to a telescopic sum, and the only term remaining is the one on the right-hand side.
  \end{proof}

 \begin{remark}\label{rem:faisant}
 Recall the 
 motivic Euler product introduced in Proposition
\ref{pro:EPD}.
Consider the companion motivic Euler product
\begin{equation}\label{eq:auxiliary_product}\prod_{x\in \A^1}\left( 1 + \sum_{i\geq 1} U_i(f)_x\LL^{-in}S T^{i}\right),
\end{equation}
with an extra variable $S$.  Then, by definition, the motivic Euler product $\mathfrak{S}(f)$ 
is equal to the product (\ref{eq:auxiliary_product}) evaluated at $S = 1$ and then at $T = 1$.
On the other hand,  
the exponential sum estimates from Section \ref{sec:exp_sums_general_bound} apply in the same way as in the proof of 
\eqref{eq:train}, and allow us to deduce the bound
$$
w(U_i(f)_x\LL^{-in}) \leq \frac{(2^{d-1}(d-1) -n)i}{2^{d-2}(d-1)}.
$$
On assuming that $n>2^{d-1}(d-1)$, 
this  shows that we may evaluate the product \eqref{eq:auxiliary_product} at $T = 1$ and then Lemma \ref{lem:factor} shows that we obtain the motivic Euler product
$$\prod_{x\in \A^1}\left(1 + \left(\lim_{N\to \infty}\LL^{-N(n-1)}[ \Lambda_N(f,x)]-1\right) S\right).$$
Hence, on evaluating at $S = 1$,  we get that 
$$\mathfrak{S}(f) = \prod_{x\in \A^1} \left( 1 + \left(\lim_{N\to \infty}\LL^{-N(n-1)}[ \Lambda_N(f,x)]-1\right) S\right)_{|S = 1}.$$

\end{remark}

  Bearing in mind Notation \ref{notation:euler_product}, the following result is a consequence of 
  Remark \ref{rem:faisant}.
  
\begin{cor}\label{cor:local-factor}
We have
$$
\mathfrak{S}(f)=
\prod_{x\in \A^1}\lim_{N\to \infty}\LL^{-N(n-1)}[ \Lambda_N(f,x)],
$$
where $\Lambda_N(f,x)$ is given by \eqref{eq:lambda}.
\end{cor}

\subsection{Singular integral}

Returning to Proposition \ref{prop:major_arcs_first_expression}, we proceed by analysing the term
$$
\int U(\theta) d\theta,
$$
where $U(\theta)$ is given by \eqref{eq:U(theta)}. We want to rewrite it in a way that draws an obvious comparison with  the singular integral in the classical circle method.

\subsubsection{Rewriting the sum (over polynomials) as an integral (over power series in $t^{-1}$)} For this, we show that the function $(g_1,\dots,g_n) \mapsto \res (\theta f(g_1,\dots,g_n))$ is invariant modulo $\T^n$; i.e., for every $\x \in t^{e}k[[t^{-1}]]$ and every $\sigma\in \T^n$, we claim that
$$
\ord(\theta (f(\x + \sigma) - f(\x))) < -1.
$$
The polynomial $f(\x + \sigma) - f(\x)$ is a sum of monomials of global degree $d$, with the degree in $\x$ being at most $d-1$. Thus, bounding the order of the components of $\x$ by $e$ and the order of the components of $\sigma$ by $-1$, we get
$$\ord((f(\x + \sigma) - f(\x))) \leq (d-1) e -1,$$
whence
$$\ord(\theta (f(\x + \sigma) - f(\x)))\leq -de-2 + \gamma + (d-1)e -1 = -e-3 + \gamma <-1,$$
since  $\gamma \leq e$
in \eqref{eq:gamma}.
Thus, it makes sense to consider the element $$[\A^{n(e,-1)}\times \A^{(-de-2+\gamma, -de-2)}, (\x,\theta)\mapsto\res(\theta f(\x))]\in \expp_{\A^{n(e,-1)}\times \A^{(-de-2+\gamma, -de-2)}},$$ and moreover, since there is an obvious isomorphism $\A^{n(e,-1)}\simeq \Poly_{\leq e}^n$, we get the equality
 $$U(\theta) = \int_{\A^{n(e,-1)}} [\A^{n(e,-1)}\times \A^{(-de-2+\gamma, -de-2)}, \res(\theta f(\x))] d\x$$
in $\expp_{\A^{(-de-2+\gamma, -de-2)}}.$

\subsubsection{Change of variables to get an integral over $\T^n$}
Using Remark \ref{rem:change_of_var}, the change of variables $\x = t^{e+1}\y$, together with the homogeneity of the polynomial $f$, gives the equality 
 $$U(\theta) = \LL^{n(e+1)} \int_{\A^{n(-1,-e-2)}}  [\A^{n(-1,-e-2)}\times \A^{(-de-2+\gamma, -de-2)}, \res(\theta t^{(e+1)d}f(\y))]d\y,$$
in $\expp_{\A^{(-de-2+\gamma, -de-2)}}.$

\subsubsection{Integrating over the space of $\theta$, and a change of variable}
Now, using the change of variable $\beta = t^{(e+1)d}\theta$ and Remark \ref{rem:change_of_var}, we get
\begin{equation}\label{eq:theta-to-beta}
\int U(\theta) d\theta= \LL^{n(e+1)}\LL^{-d(e+1)}\int
\mathcal{J}(\y)d\y,
\end{equation}
where
$$
\mathcal{J}(\y)=
 \int_{\A^{(d-2+\gamma, d-2)}} [\A^{n(-1,-e-2)}\times \A^{(d-2+\gamma, d-2)}, \res(\beta f(\y))]d\beta.
 $$
We may now establish the following result.

\begin{prop}\label{pro:orp}
We have 
$$
\int U(\theta) d\theta = \LL^{-de-1+\gamma}
[V_{d+\gamma -1}],
$$
where
$$
V_{d+ \gamma -1} = \{ \y\in \A^{n(-1,-e-2)}: \ord f(\y) < -d-\gamma + 1\}.
$$
\end{prop}

\begin{proof}
By definition, we have 
$$
\mathcal{J}(\y)=
 \LL^{d-1} [\A^{n(-1,-e-2)}\times \A^{(d-2 + \gamma,d-2)}, \res(\beta f(\y))]
 $$
 in $\expp_{\A^{n(-1,-e-2)}}.$
We write
$$
\beta = b_{d-2  +\gamma} t^{d-2 + \gamma} + \dots + b_{d-1}t^{d-1} + t^{d-2} k[[t^{-1}]]
$$ 
and
$$f(\y) = c_{d}(\y)t^{-d} + c_{d+1}(\y)t^{-d-1} +\cdots, $$
where we have used  the fact that $f$ is homogeneous of degree $d$ to deduce that  $f(\y)$ only has terms of degrees $\leq -d$. 
Then $$\res (\beta f(\y)) = b_{d-1} c_d(\y) + b_{d} c_{d+1}(\y)+ \dots + b_{d + \gamma -2} c_{d + \gamma -1}(\y).$$ 
Thus, by Lemma \ref{lemma:linear_form}, for every $\y\in \A^{n(-1,-e-2)}$, the value of  
$\mathcal{J}(\y)$
at $\y$ is $\LL^{d-1 + \gamma}$ if and only if $$c_d(\y)=c_{d+1}(\y)= \dots= c_{d + \gamma -1}(\y) = 0,$$
that is, if and only if $\ord f(\y) < -d-\gamma +1$, and 0 otherwise. Using Lemma \ref{lemma:zero_criterion}, 
it now follows from \eqref{eq:theta-to-beta} that
\begin{align*} 
\int U(\theta) d\theta
&= \LL^{n(e+1)}\LL^{-d(e+1)}
\int
[V_{d+\gamma -1}]\LL^{d-1 + \gamma}d\y\\
&= \LL^{n(e+1)}\LL^{-d(e+1)}
\LL^{-n(e+1)}[V_{d+\gamma -1}]\LL^{d-1 + \gamma},
\end{align*}
with 
$
V_{d+ \gamma -1}$ as in the statement.  The proposition easily follows.
\end{proof}

\subsection{Connection to jet spaces}
For any integer $N$, 
recall that the $N$th {\em jet space}  of $X$ is given by
$$
\mathcal{L}_N(X)= \{\x = \x_{0} + \x_1 t + \dots + \x_{N} t^N:  f(\x) \equiv 0 \bmod t^{N+1}\},
$$
where $\x_0,\dots,\x_N$ run over $\A^n$.
We take $\mathcal{L}_N(X)=\varnothing$ when $N<0$, and note that $\mathcal{L}_0(X)=X$. 

Jet spaces of smooth varieties are very well understood.  In our setting, the variety $X$ has one singular point at 0, so we decided it would be useful to give a few details on the structure of jet spaces for hypersurfaces defined by non-singular  homogeneous polynomials $f\in \C[x_1,\dots,x_n]$ of degree $d$.

\begin{lemma}\label{lem:jet_space} 
For any $N\geq 1$, we have 
$$
\frac{[\mathcal{L}_{N}(X)]}{\LL^{(N+1)(n-1)}}=
\frac{[\mathcal{L}_{N-1}(X)]}{\LL^{N(n-1)}}+
\LL^{N-n-n\lfloor N/d\rfloor}\cdot 
\begin{cases}
\LL-1
& \text{ if $d\nmid N$,}\\
\LL \left([X]-\LL^{n-1}\right) & 
\text{ if $d\mid N$.}
\end{cases}
$$
in $\expp_\C$. 
\end{lemma}

\begin{proof}
We  write 
$Z_N$ for the locus of points in  $\mathcal{L}_N(X)$ with $\x_0=\mathbf{0}$ and we let 
$U_N=\mathcal{L}_N(X)\setminus Z_N$. 
We claim that 
$$
Z_N \simeq 
\begin{cases}
\A^{Nn} & \text{ if $0\leq N\leq d-1$,}\\
\mathcal{L}_{N-d}(X)\times \A^{(d-1)n} &  \text{ if $N\geq d$.}
\end{cases}
$$
To see  this, we note that $Z_N$ is the set of $\x_1 t+\dots +\x_{N}t^N$ such that 
$f(\x_1 t+\dots +\x_{N}t^N)\equiv 0 \bmod{t^{N+1}}$. 
When $N\leq d-1$ it is clear that $Z_N\simeq \A^{Nn}$. Suppose that $N\geq d$. 
Then $Z_N$  is equal to  space 
of $\x=\x_1 t+\dots +\x_{N}t^{N}$ such that 
$f(\x)\equiv 0 \bmod{t^{N+1}}$. Since 
$f$ is homogeneous of degree $d$, we have 
$f(\x_1 t+\dots +\x_{N}t^{N})=t^df(\x_1 +\x_2t\dots +\x_{N}t^{N-1})$, whence $Z_N$  is equal to  space 
of $\x=\x_1+\x_2 t+\dots +\x_{N}t^{N-1}$ such that 
$f(\x)\equiv 0 \bmod{t^{N+1-d}}$.
The claim is now obvious. 

Now let $N\geq 1$ and 
consider the truncation  
 morphism $
 \pi :U_N\to U_{N-1}.
 $
For each  $\x=\x_0+\x_1 t+\dots +\x_{N-1}t^{N-1}\in U_{N-1}$, the fibre $\pi^{-1}(\x)$ is the set of $\x_N\in \A^n$ such that $f(\x+\x_N t^N)\equiv 0 \bmod t^{N+1}$. Note  that 
$f(\x)\equiv 0 \bmod t^N$ and 
$$
f(\x+\x_N t^N)=f(\x)+t^N\x_N \cdot \nabla f(\x_0) +O(t^{N+1}),
$$ by Taylor's theorem. 
Since $f$ is non-singular, 
 $\nabla f(\x_0)$ does not vanish for $\x_0\neq \mathbf{0}$. Hence
 $\pi^{-1}(\x)\simeq \A^{n-1}$. 
 Hence it follows that 
 \begin{equation}\label{eq:pavel}
[U_N]= [U_{N-1}]\LL^{n-1},
 \end{equation}
 for any $N\geq 1$, since  $\pi$
is a  Zariski locally trivial fibration with fibre $\A^{n-1}$.
 
To prove the lemma we write
 \begin{equation}\label{eq:pavel'}
\Psi_N=
\frac{[\mathcal{L}_{N}(X)]}{\LL^{(N+1)(n-1)}}-
\frac{[\mathcal{L}_{N-1}(X)]}{\LL^{N(n-1)}},
\end{equation}
for any $N\geq 1$. 
To begin with, we assume that $N>d$. Then 
 \begin{align*}
[ \mathcal{L}_N(X)]&= [U_N] +[Z_N]=[U_N]+[\mathcal{L}_{N-d}(X)]\LL^{(d-1)n},
\end{align*}
and the same expression holds with $N-1$ in place of $N$. Thus
\begin{align*}
\Psi_N
&= \LL^{-(N+1)(n-1)}\left([U_N]- [U_{N-1}] \LL^{n-1}\right) +\LL^{-(n-d)} \Psi_{N-d}
= \LL^{-(n-d)} \Psi_{N-d},
\end{align*}
by \eqref{eq:pavel}.  
We next calculate $\Psi_N$ when $N\leq d$, for which we observe that 
$$
[Z_N]=\begin{cases}
\LL^{Nn} &\text{ if $0\leq N<d$,}\\
[X]\LL^{(d-1)n} &\text{ if $N=d$},
\end{cases}
$$
since $\mathcal{L}_0(X)=X$.
It easily follows from \eqref{eq:pavel'} that 
\begin{align*}
\Psi_d
&=  \LL^{-(d+1)(n-1)}\left([Z_d]- [Z_{d-1}] \LL^{n-1}\right)\\
&=  \LL^{-(d+1)(n-1)}\left([X]\LL^{(d-1)n}-  \LL^{dn-1}\right)\\
&= \LL^{-(2n-d-1)}  \left([X]-\LL^{n-1}\right).
\end{align*}
Similarly,
\begin{align*}
\Psi_N
&= \LL^{-(N+1)(n-1)}\left([Z_N]- [Z_{N-1}] \LL^{n-1}\right) \\
&= \LL^{-(N+1)(n-1)}\left(\LL^{Nn}-  \LL^{Nn-1} \right) \\
&= \LL^{-(n-N)}\left(\LL-1\right),
\end{align*}
if $1\leq N\leq d-1$.

We are now ready to complete the proof of the lemma. Any $N\geq 1$ can be written 
$N=qd+r$, where $q=\lfloor N/d\rfloor$ and
$0\leq r<d$. 
If  $r>0$ then we may put  the above calculations together to deduce that 
$$
\Psi_N=\LL^{-(n-d)q} \Psi_r= \LL^{N-n(q+1)}(\LL-1)=
\LL^{N-n-n\lfloor N/d\rfloor}(\LL-1)
$$
On the other hand, if $r=0$ then $q\geq 1$ and 
$$
\Psi_N=\LL^{-(n-d)(q-1)} \Psi_d=
\LL^{N-n+1-nN/d}
 \left([X]-\LL^{n-1}\right).
$$
This completes the proof of the lemma.
\end{proof}

\begin{remark}\label{rem:pad}
Let $N\geq 1$. If 
 $d\nmid N$ then 
$$
\dim\left( 
\LL^{N-n-n\lfloor N/d\rfloor}\cdot (\LL-1)\right)
\leq (N+1)(1-n/d),
$$
since
$\lfloor N/d\rfloor\geq N/d-1+1/d$. Moreover, if $d\mid N$ then 
  $$
\dim\left( 
\LL^{N-n-n\lfloor N/d\rfloor}
\cdot\LL \left([X]-\LL^{n-1}\right)
 \right)\leq N(1-n/d).
  $$
\end{remark}

\begin{cor}\label{cor:dim}
Assume that $n>d$. Then the $N$-th jet space $\mathcal{L}_N(X)$ is irreducible and has dimension $(n-1)(N+1)$. 
\end{cor}

\begin{proof}
Let us begin by calculating the  dimension, noting that
the lower bound
$$\dim \mathcal{L}_N(X)\geq (n-1)(N+1)$$ 
is trivial. To show that 
$\dim \mathcal{L}_N(X)\leq (n-1)(N+1)$, we argue by induction on $N$.
When $N=0$ we have 
  $\mathcal{L}_0(X)=X$ and the claim is obvious. 
If $N\geq 1$ it 
 follows from 
Lemma~\ref{lem:jet_space}, the induction hypothesis and Remark \ref{rem:pad}
that   
$$
\dim\left( \frac{[\mathcal{L}_{N}(X)]}{\LL^{(N+1)(n-1)}}-
\frac{[\mathcal{L}_{N-1}(X)]}{\LL^{N(n-1)}} \right)\leq (N+1)(1-n/d)<0,
  $$
since $n>d$. 
This confirms that $\dim \mathcal{L}_N(X)\leq (n-1)(N+1)$. 

Turning to the question of irreducibility, we appeal to work of 
Mustaţa \cite{must}. This confirms that all the jet spaces
$\mathcal{L}_N(X)$ are irreducible if and only if $X$  has 
 canonical singularities. 
 But $X$ is an affine hypersurface with an isolated singularity at the point $(0,\dots,0)$.
 Since $n>d$, it follows from Koll\'ar \cite[Corollary~3.3]{kollar} that this singularity is canonical, which completes the proof. 
\end{proof}

We now proceed by relating the jet spaces to the spaces that we met in Corollary \ref{cor:local-factor} 
and 
Proposition \ref{pro:orp}.

\begin{lemma}\label{lem:relation'}
Let $x\in \A^1$ and let $N\geq 1$. Then we have the relation
$$
[\Lambda_N(f,x)] =[\Lambda_N(f,\infty)]=  [\mathcal{L}_{N-1}(X)]
$$
in the Grothendieck ring of varieties, 
where 
$\Lambda_N(f,x)$ is given by 
\eqref{eq:lambda} and 
$ \Lambda_N(f,\infty)$ is given by 
\eqref{eq:jet}.
\end{lemma}

\begin{proof}
The first equality is clear and so it suffices to prove that
$\Lambda_N(f,x)\simeq \mathcal{L}_{N-1}(X)$, for any $x\in \A^1$.
According to \eqref{eq:lambda},  $\Lambda_N(f,x)$ is the space of $\g=\g_0+\g_1 t+\dots +\g_{N-1}t^{N-1}$ such that 
$f(\g)\equiv 0\bmod{(t-x)^N}$, for $\g_0,\dots,\g_{N-1}\in \A^n$.
This is isomorphic to  the space of 
$\g'=\g_0'+\g_1' (t-x)+\dots +\g_{N-1}'(t-x)^{N-1}$ such that 
$f(\g')\equiv 0\bmod{(t-x)^N}$, via the map 
$\g\mapsto\g'$ given by 
\begin{align*}
\g_0'
&=\g_0+x\g_1+\dots +x^{N-1} \g_{N-1},\\
\g_1'
&=\g_1+2x\g_2+3x^2\g_3+\dots +(N-1)x^{N-1} \g_{N-1},\\
\g_2'
&=\g_2+\tbinom{3}{2}x\g_3+\tbinom{4}{2}x^2\g_4+\dots +\tbinom{N-1}{2}x^{N-1} \g_{N-1},\\
&~~\vdots\\
\g_{N-1}'&=\g_{N-1}.
\end{align*}
But this is just the jet space 
$\mathcal{L}_{N-1}(X)$, which thereby  completes the proof.
\end{proof}

Turning to the  space $V_{d+ \gamma -1}$ that appears in Proposition \ref{pro:orp}, we  prove the following result, relating it to the variety 
$ \Lambda_N(f,\infty)$ in 
\eqref{eq:jet} for a suitable choice of $N$.

\begin{lemma}\label{lem:relation}
We have the relation
$$[V_{d+ \gamma -1}] = [\Lambda_{\gamma}(f,\infty)]\LL^{n(e+1-\gamma)}$$
in the Grothendieck ring of varieties.
\end{lemma}
\begin{proof}
For any $\y =  \y_1 t^{-1} + \dots + \y_{e+1}\t^{-e-1}\in V_{d+ \gamma -1}$, the condition  $\ord f(\y)< -d-\gamma +1$ says that the coefficients of the terms in $f(\y)$ of degrees $-d, \dots, -d-\gamma +1$ should all be zero. This  does not affect the coordinates $\y_{\gamma +1},\dots,\y_{e+1}$, because $f$ is homogeneous of degree $d$. Thus, there is an isomorphism
  $$
  V_{d + \gamma -1} \simeq \{ \y \in \A^{n(-1, -\gamma-1)}: \ord f(\y) < -d-\gamma +1\} \times \A^{n(e+1-\gamma)},
  $$
 sending $\y$ to $(\y_1t^{-1} + \dots + \y_{ \gamma}t^{-\gamma}, (\y_{\gamma+1},\dots,\y_{e+1})).$
 Writing $\y_1t^{-1} + \dots + \y_{\gamma}t^{-\gamma}=t^{-1}\g$, where 
 $\g=\y_1 +\y_2 t^{-1}+ \dots + \y_{\gamma}t^{-(\gamma-1)}$, 
 we see that 
$\ord f(\y) < -d-\gamma +1$ if and only if 
$f(\g) \in t^{-\gamma}\C[t^{-1}]$. 
Hence 
 $V_{d + \gamma -1} \simeq \Lambda_\gamma(f,\infty)\times \A^{n(e+1-\gamma)}$, in the notation of \eqref{eq:jet}.
 \end{proof}

\subsection{Major arcs: conclusion}\label{sect:major_arcs_conclusion}

We are now ready to conclude the proof of
Proposition~\ref{prop:major}.
We  return to Proposition \ref{prop:major_arcs_first_expression}, where
$\gamma\geq 1$ is given by \eqref{eq:gamma}. 
It follows from Proposition \ref{pro:orp} 
and Lemma \ref{lem:relation}
that 
\begin{align*} N_{\maj} 
& =  \LL^{\mu(e)} \left( \sum_{0\leq m\leq \lfloor \frac{e+1}{2}\rfloor} S_{m}(f) \LL^{-nm} \right) \LL^{-(n-1)\gamma}[\Lambda_{\gamma}(f,\infty)],
\end{align*}
where   $\mu(e)$ is given by \eqref{eq:mu}.

Next, we observe that 
$w(\LL^{-(n-1)\gamma}[\Lambda_{\gamma}(f,\infty)])\leq 0$,
as follows from Lemma \ref{lem:relation'} and Corollary \ref{cor:dim}.
Appealing to Lemma~\ref{lem:train}, we conclude that 
$$
N_{\maj}
 =  \LL^{\mu(e)}\left(\mathfrak{S}(f) \cdot  \LL^{-(n-1)\gamma}[\Lambda_{\gamma}(f,\infty)]
+ R_e' \right),
$$
where 
$\mathfrak{S}(f)$ is the motivic Euler product described in 
Corollary \ref{cor:local-factor} and 
  $$
  w \left(R_e'\right) \leq 4
  -\tilde\nu \left(1+\left\lfloor \frac{e+1}{2d-2}\right\rfloor\right),
  $$
with   $\tilde\nu$ given by \eqref{eq:tilde-nu}. Inserting 
Lemma \ref{lem:relation'} into  Lemma \ref{lem:jet_space}, it therefore follows
from Remark \ref{rem:pad} that 
$$
\LL^{-(n-1)\gamma}[\Lambda_{\gamma}(f,\infty)] 
=
\lim_{N\to \infty} \LL^{-(n-1)N}[\Lambda_{N}(f,\infty)]
+E,
$$
in $\widehat{\expp_\C}$, where
$$
w(E)\leq 
\begin{cases}
-2\gamma (n/d-1)
&\text{ if $d\mid \gamma$,}\\
-2(\gamma+1) (n/d-1) &\text{ if $d\nmid \gamma$.}
\end{cases}
$$

We now return to our expression for the major arcs and we recall from  
Remark \ref{rem:weight-SS} that $w(\mathfrak{S}(f))=0$.
This allows us to write
$$
N_{\maj}
 =  \LL^{\mu(e)}\left(\mathfrak{S}(f) \cdot \lim_{N\to \infty} \LL^{-(n-1)N}[\Lambda_{N}(f,\infty)]
+ R_e'' \right),
$$
where
\begin{equation}\label{eq:max}
  w \left(R_e''\right) 
  \leq \max\left( -\frac{(e+1+\kappa_{d,e})(n-d)}{d}
  , ~4
  -\tilde\nu \left(1+\left\lfloor \frac{e+1}{2d-2}\right\rfloor\right)\right)
\end{equation}
and
$$
\kappa_{d,e} =
\begin{cases}
0 &\text{ if $d\mid \gamma$,}\\
2 &\text{ if $d\nmid \gamma$.}
\end{cases}
$$

Suppose first that $d\geq 3$. 
Then we claim that the second term always exceeds the first term in the maximum \eqref{eq:max}, 
as claimed in  Proposition~\ref{prop:major}.
To check the claim, we note that 
it is equivalent to 
$$
\tilde\nu \left(1+ \left\lfloor \frac{e+1}{2d-2}\right\rfloor\right)\leq \frac{(e+1+\kappa_{d,e})(n-d)}{d}+4
.
$$
On inserting the definition 
\eqref{eq:tilde-nu} of 
$\tilde \nu$, and collecting together the coefficients of $n$, we see that this holds if 
$
A_{d,e}\leq nB_{d,e},
$
where $A_{d,e}=\kappa_{d,e}-(e+1)-4$ and 
$$
B_{d,e}=\frac{e+1+\kappa_{d,e}}{d} -\frac{1}{2^{d-2}}\left(1+
\left\lfloor \frac{e+1}{2d-2}\right\rfloor\right).
$$
Since $A_{d,e}<0$, the claim  follows if we can prove that $B_{d,e}\geq 0$ for $e\geq 1$ and $d\geq 3$. 
If $d\nmid \gamma$ then  $\kappa_{d,e}=2$ and
\begin{align*}
B_{d,e}
&\geq (e+1)\left(\frac{1}{d}-\frac{1}{2^{d-1}(d-1)}\right) +\frac{2}{d}-\frac{1}{2^{d-2}}
\geq \frac{4}{d}-\frac{d}{2^{d-2}(d-1)},
\end{align*}
on taking $e\geq 1$. Thus we deduce that $B_{d,e}\geq 0$ for all $d\geq 2$.
If $d\mid \gamma$ then $\kappa_{d,e}=0$ and 
we must have  $e+1\geq 2d-1\geq 5$, if $d\geq 3$. 
But then it follows that 
\begin{align*}
B_{d,e}
&\geq (e+1)\left(\frac{1}{d}-\frac{1}{2^{d-1}(d-1)}\right) -\frac{1}{2^{d-2}}
\geq \frac{5}{d}-\frac{2d+3}{2^{d-1}(d-1)}.
\end{align*}
This is non-negative for all $d\geq 3$.

Finally, we suppose that $d=2$. We may also suppose that $2\mid \gamma$, since 
the previous paragraph suffices to handle the case $2\nmid \gamma$, even when $d=2$.
Thus \eqref{eq:max} becomes
\begin{align*}
  w \left(R_e''\right) 
 & \leq \max\left( -\frac{(e+1)(n-2)}{2}
  , ~4
  -\tilde\nu \left(1+\left\lfloor \frac{e+1}{2}\right\rfloor\right)\right)\\
  & \leq \max\left( -\frac{(e+1)(n-2)}{2}, ~-\frac{n(e+2)}{2}+2e+8\right)
\end{align*}
since $\kappa_{2,e}=0$, $\tilde \nu=n-4$
and $1+\lfloor \frac{e+1}{2}\rfloor \geq e/2+1$.
This completes the proof of Proposition~\ref{prop:major}.

\section{The motivic minor arcs}\label{sect:minor_arcs}

Recall the expression for $M_{m,\gamma}$ in Lemma \ref{lem:6.1} and put 
$\Delta = \left\lfloor \frac{e+1}{2}\right \rfloor.$ 
It follows from \eqref{eq:Major} that  the full set of  major arcs is $\mathfrak{M}: = M_{\Delta, \gamma}$, where $\gamma=\left\lceil \frac{e+1}{2}\right \rceil.$ 
 Our aim in this section  is therefore  to get a bound for
 $$N_{\minor} = \LL^{-de-1}[ \Poly_{\leq e}^n \times S, \res(\alpha f(g_1,\dots,g_n))].
 $$
 where $S=
\A^{de+1}-\mathfrak{M}$.
Now it follows from Proposition \ref{pro:stone} that 
\begin{equation}\label{eq:joust}
w_{S}(
[\Poly_{\leq e}^n\times S, \res(\alpha f(g_1,\dots,g_n)]
)\leq 
\frac{\max_{\alpha\in S} \dim N(\alpha) + N(2^{d-1} - (d-1))}{2^{d-2}},
\end{equation}
where $N=(e+1)n$ and 
$$
N(\alpha) = \left\{ (\u^{(1)},\dots,\u^{(d-1)})\in \left(\Poly_{\leq e}^n\right)^{d-1}:
\begin{array}{l}
\ord \{\alpha \Psi_j(\u^{(1)},\dots,\u^{(d-1)})\} < -e-1\\
\forall j\in \{1,\dots,n\}
\end{array}
\right\}.
$$

\subsection{The minor arc bound}

We recall the definition of 
$A_m^{de+1}$ from Remark \ref{rem:Am}. In particular, Remark \ref{rem:Am_0} implies that 
$A_{m_0}^{de+1}=\A^{de+1}$, where
$m_0=\lceil\frac{de+1}{2} \rceil$. Bearing this in mind, we begin by  proving the following result. 

\begin{lemma}\label{lem:m}
Let $\alpha\in A^{de+1}_{m+1} -A^{de+1}_m$, with  $m\leq m_0-1$.  Then 
$$
\dim N(\alpha) \leq  \left(de + d -e-2 -  \left \lfloor \frac{m}{d-1}\right \rfloor \right) n.
$$
\end{lemma}

\begin{proof}
We 
seek to apply Lemma \ref{lemma:general_bound_exp_sums}.
  Since $\alpha\in A^{de+1}_{m+1}$, there exist coprime $h_1,h_2$ with $\deg h_1 < \deg h_2 \leq m+1$ and $\theta$ such that $\ord \theta \leq -de-2 + m+1 - \deg h_2$. 
  Moreover, since $\alpha \not \in A^{de+1}_{m}$, we have the following dichotomy: either $\deg h_2 = m+1$, or if $\deg h_2 \leq m$, then we cannot have
 $\ord \theta \leq -de-2 + m -\deg h_2$.  Thus  $$
\ord \theta = -de-2 + m+1 - \deg h_2$$
if $\deg h_2\leq m$.
 Using this, we now prove a bound on $N(\alpha)$ using Lemma \ref{lemma:general_bound_exp_sums} with $E=e+1$. 
 
 We place ourselves in the first case, where, using the notation of the lemma, $\rho = m+1$ and $\psi \leq -de-2.$ 
  We start by checking condition $(1)$ in the lemma. The first inequality 
 is satisfied
 provided that $(d-1)s\geq m-e.$
 Using $\psi \leq -de-2$, we see that the second inequality 
  is also implied by $(d-1)s\geq m-e$. 
    We now check condition $(2)$ in the lemma. In fact, as we have no longer have a 
    useful lower bound on $\psi$ in this setting, we  find a condition for the first inequality to hold, which is equivalent to  $$(d-1)e -m \leq (d-1) s.$$
  
  We now treat the second case, where $\rho \leq m$ and $\psi = -de-2 + m+1 - \rho = -de-1 + m - \rho$. 
    We start by  checking condition $(1)$ in the lemma. The first inequality 
is implied by 
$(d-1)s\geq m-e$.
The second  inequality   is equivalent to  $ (d-1) s\geq m-e$.
 We now check condition $(2)$ in the lemma, noting that  we have no useful lower bound on $\rho$ in this setting. But we see that the  second inequality holds
if and only if  $$(d-1)e-m \leq (d-1) s.$$
 
 In  conclusion, we are able to apply 
 Lemma \ref{lemma:general_bound_exp_sums} provided that $s$ is an integer chosen to satisfy
 $$
 (d-1)s\geq \max\left\{ 0, m-e, (d-1)e-m\right\}.
 $$
 Since $d\geq 2$, one  checks that the maximum is 
$ (d-1)e-m$ if
 $m\leq m_0-1$. 
 Hence, taking  
 $s= e - \left \lfloor \frac{m}{d-1}\right \rfloor$, we arrive at the upper bound
 $$\dim N(\alpha) \leq (d-2) (e+1) n + \left(e-  \left \lfloor \frac{m}{d-1}\right \rfloor \right)n = \left(de + d -e-2 -  \left \lfloor \frac{m}{d-1}\right \rfloor \right) n,
 $$
 as claimed.
\end{proof}

The bound in Lemma 
\ref{lem:m}  becomes stronger as $m$ increases. Thus, to optimise our minor arc estimate, we  stratify the minor arcs.
We note from Remarks \ref{rem:Am_0} and \ref{rem:Am}, together with  Lemma \ref{lem:6.1}, 
that 
$A^{de+1}_{\Delta}\subset M_{\Delta,\gamma}\subset \mathfrak{M}$ and $A^{de+1}_{m_0}=\A^{de+1}$, where
 $m_0=\lceil\frac{de+1}{2} \rceil$.
Since the spaces $A^{de+1}_{m}$ form an increasing sequence, we 
may therefore  cut up the  complement of the major arcs as
$$\A^{de+1}-\mathfrak{M} = \left(\bigsqcup_{\Delta \leq m\leq m_0-1} \left( A^{de+1,\star}_{m+1} -A^{de+1,\star}_m\right) \right),
$$
where $ A^{de+1,\star}_{m} : = A^{de+1}_{m} - \mathfrak{M}.$

We are now ready to produce our final bound for the minor arc contribution
$$N_{\minor} = \LL^{-de-1}\left[ \Poly_{\leq e}^n \times \left(\A^{de+1}-\mathfrak{M}\right), \res(\alpha f(g_1,\dots,g_n)\right].
$$ 
Thus, using  
property (\ref{item:weight-property-sum}) in Proposition \ref{prop:weight-properties}, we get 
\begin{align*} w(N_{\minor}) & \leq -2de-2 + w\left( \sum_{m=\Delta}^{m_0-1}[\Poly_{\leq e}^n\times (A^{de+1,\star}_{m+1} - A^{de+1,\star}_m), \res(\alpha f(g_1,\dots,g_n))]\right)\\
& \leq -2de-2  + \max_{\Delta\leq m \leq m_0-1}w( [\Poly_{\leq e}^n\times (A^{de+1,\star}_{m+1} - A^{de+1,\star}_m), \res(\alpha f(g_1,\dots,g_n))]).
\end{align*}
According to Remark \ref{rem:stratum-dimension} we have 
$$
\dim(A^{de+1,\star}_{m+1} - A^{de+1,\star}_m)
\leq \dim A^{de+1}_{m+1}
\leq 2(m+1).
$$
We now use successively property (\ref{item:weight-property-pushforward}) of Proposition \ref{prop:weight-properties} and the inequality \eqref{eq:joust}, to write
\begin{align*}
w( [&\Poly_{\leq e}^n\times (A^{de+1,\star}_{m+1} - A^{de+1,\star}_m), \res(\alpha f(g_1,\dots,g_n))])\\
&
\leq 
w_{A^{de+1,\star}_{m+1} - A^{de+1,\star}_m}\left( [\Poly_{\leq e}^n\times (A^{de+1,\star}_{m+1} - A^{de+1,\star}_m), \res(\alpha f(g_1,\dots,g_n))]\right) +4(m+1)\\
&\leq 
\frac{\dim_{A^{de+1,\star}_{m+1} - A^{de+1,\star}_m} N(\alpha) + (e+1)n(2^{d-1} - (d-1))}{2^{d-2}} +4(m+1),
\end{align*}
for any $m$. It now follows from Lemma \ref{lem:m} that 
$w(N_{\minor}) $ is 
\begin{align*} 
& \leq 2-2de  + \max_{\Delta\leq m \leq m_0-1}\left( \frac{\left(de + d - e - 2 - \left\lfloor \frac{m}{d-1}\right\rfloor\right) n + (e+1)n(2^{d-1} - (d-1)))}{2^{d-2}} + 4m \right)\\ 
&\leq 2-2de -\frac{n}{2^{d-2}} + 2n(e+1) + \frac{1}{2^{d-2}}\max_{\Delta\leq m \leq m_0-1}\left( 2^{d} m - \left\lfloor \frac{m}{d-1}\right\rfloor n \right).
\end{align*}
We now appeal to Lemma \ref{lem:jens}, to conclude that 
\begin{align*} w(N_{\minor}) 
 &\leq 2-2de -\frac{n}{2^{d-2}} + 2n(e+1) + 4(d-1) +\frac{1}{2^{d-2} }\left\lfloor \frac{e+1}{2d-2}\right\rfloor (2^{d}(d-1) -n).\\
&= \underbrace{2n(e+1) -2(de+1)}_{2\times \text{expected dimension}} + 4d - \frac{n}{2^{d-2}}  +\frac{1}{2^{d-2} }\left\lfloor \frac{e+1}{2d-2}\right\rfloor (2^{d}(d-1) -n).\\
& =2\mu(e)
+ \frac{1}{2^{d-2}} \left( (2^{d} d -n) +\left\lfloor \frac{e+1}{2d-2}\right\rfloor (2^{d}(d-1) -n)\right) \end{align*}
The statement of Proposition \ref{prop:minor} now follows.

\section{The space of morphisms}

In this section we prove Theorem \ref{theorem_morspace}, which deals with the space
$$
\Mor_{e}(\Proj^1,Z) = \left\{(g_1,\dots,g_n)\in (\Poly^n_{\leq e}-\{0\})/\C^{\times}: 
\begin{array}{l}
\max \deg g_i = e\\
\gcd(g_1,\dots,g_n) = 1 \\
f(g_1,\dots,g_n) = 0
\end{array}
\right\},
$$
where  $(\Poly^n_{\leq e}-\{0\})/\C^{\times}$ is the space of non-zero $n$-tuples of polynomials of degree $\leq e$, viewed modulo the multiplication by a non-zero scalar. 
We begin with the  following result, which  allows us to pass from the classes of the naive moduli spaces $M_e$ to 
the class of the moduli space $\Mor_{e}(\Proj^1,Z)$.

\begin{lemma}\label{lem:morspace_naivespace}
Let $e\geq 1$. Then 
$$
[\Mor_{e}(\Proj^1,Z)] = \frac{[M_e] - (\LL+1)[M_{e-1}] + \LL[M_{e-2}]}{\LL -1},
$$
with the convention that $M_{-1}$ is a point.
\end{lemma}

\begin{proof}
For every $e\geq 1$, note that $M_{e-1}$ is a subset of $M_e$, and that a point $g = (g_1,\dots,g_n)\in M_e- M_{e-1}$ may be written in the form
$$g = (hg'_1,\dots,hg'_n)$$
where $h = \gcd (g_1,\dots,g_n)$ and $(g'_1,\dots,g'_n)$, defines, up to multiplication by a non-zero scalar multiple, an element of $\Mor_{e-\deg h}(\Proj^1,Z)$. Thus, we have the  decomposition
$$
\left(M_e -M_{e-1}\right) / \C^{\times} = \bigsqcup_{i=0}^e\MPoly_{i} \times  \Mor_{e-i}(\Proj^1,Z),
$$
where we recall that $\MPoly_{i}$ is the space of degree $i$ monic polynomials.
Thus, in the Grothendieck ring, we have the relation
$$[M_{e}] - [M_{e-1}] = (\LL-1)\sum_{i = 0}^e\left([\MPoly_{i}] \times  [\Mor_{e-i}(\Proj^1,Z)]\right),
$$
for every $e \geq 1$.

In terms of generating series, we get the relation
$$
\sum_{e\geq 1 } ([M_{e}] - [M_{e-1}]) T^{e} = (\LL-1) \sum_{e\geq 1} \left( \sum_{i = 0}^e[\MPoly_{i}] \times  [\Mor_{e-i}(\Proj^1,Z)]\right)  T^e.
$$
The left-hand side may be rewritten as
$$\sum_{e\geq 1} [M_e] T^e - \sum_{e\geq 0} [M_e] T^{e+1} = (1-T) \sum_{e\geq 0}[M_e] T^e - [X],$$
using the fact that $M_0 = X$. 
As for the right-hand side, we rewrite it as
$$(\LL-1) \left( \left( \sum_{i\geq 0} [\MPoly_{i}] T^ i \right) \left( \sum_{i\geq 0} [\Mor_{i}(\Proj^1,Z)]T^i\right) - [Z]\right),$$
since $\Mor_{0}(\Proj^1,Z)=Z$.
Noting that $[\MPoly_{i}] = \LL^{i}$, we get that 
$$\sum_{i\geq 0} [\MPoly_{i}] T^ i = \sum_{i\geq 0} \LL^{i} T^i = \frac{1}{1-\LL T}.$$
We have $[Z]=(\LL-1)^{-1}([X]-1)$, from which it follows that 
$$(1-T)(1-\LL T)\left(\sum_{e\geq 0}[M_e] T^e\right) = (\LL -1)\sum_{e\geq 0} [\Mor_{e}(\Proj^1,Z)]T^e
+1-\LL T
.$$
We expand and deduce that  the left-hand side is
\begin{align*}
& =  (1-(1+\LL)T + \LL T^2)\left(\sum_{e\geq 0}[M_e] T^e\right) \\
& =  \sum_{e\geq 0} [M_e] T^e - (\LL + 1) \sum_{e\geq 0} [M_e] T^{e+1} + \LL \sum_{e\geq 0} [M_e]T^{e+2}\\
& = \sum_{e\geq 2}[M_e]T^{e} -  (\LL + 1) \sum_{e\geq 2} [M_{e-1}] T^{e} + \LL \sum_{e\geq 2} [M_{e-2}]T^{e} + [M_0] + [M_1] T - (\LL + 1)[M_0]T\\
& = [M_0] + ([M_1]-(\LL + 1)[M_0])T + \sum_{e\geq 2}([M_e] - (\LL+1)[M_{e-1}] + \LL[M_{e-2}]) T^e.
\end{align*}
The statement of the lemma easily follows.
\end{proof}

Our next result expresses the product of local densities in a more convenient form. 

\begin{lemma}\label{lem:*}
Assume that $n>d$ and let $x\in \A^1$. Then 
\begin{align*}
\lim_{N\to \infty }\LL^{-N(n-1)} [\Lambda_N(f,x)]
&= \lim_{N\to \infty }\LL^{-N(n-1)} [\Lambda_N(f,\infty)]\\
&= \LL^{-(n-2)}[Z]  (1-\LL^{-1}) (1-\LL^{-(n-d)})^{-1}.
\end{align*}
\end{lemma}

\begin{proof}
We prove the part
involving $\Lambda_N(f,x)$,  the proof for $\Lambda_N(f,\infty)$ being identical. 
 Given $x\in \A^1$ and integer $N\geq 0$, define 
$$
    \Lambda_N^*(f,x) = \left\{
    \g\in \C[t]^n: 
    \begin{array}{l}
        \deg (g_1),\dots,\deg(g_n)<N, ~ f(\g)\equiv 0\bmod{(t-x)^N}\\
        (g_1,\dots,g_n) \not\equiv (0,\dots,0) \bmod (t-x)
        \end{array}
        \right\},
$$
as an analogue of 
\eqref{eq:lambda}.
Next, let us put 
$$
    \Lambda_{N,i}(f,x) = \left\{
    \g\in \C[t]^n: 
    \begin{array}{l}
        \deg (g_1),\dots,\deg(g_n)<N, ~ f(\g)\equiv 0\bmod{(t-x)^N}\\
        \g \equiv \0 \bmod (t-x)^{i},~        \g \not\equiv \0 \bmod (t-x)^{i+1}
        \end{array}
        \right\},
$$
for any $0\leq i<N.$  
Let $ \g\in  \Lambda_{N,i}(f,x)$. Then $(t-x)^i\mid \g$, whence 
$(t-x)^{di}\mid f(\g)$. 
If
$N\leq id$,  then 
the constraint $f(\g)\equiv 0\bmod{(t-x)^N}$ is vacuous and it follows that 
$\Lambda_{N,i}(f,x)$
is isomorphic  to the set of polynomials 
 $\g\in \C[t]^n$ of degree $<N-i$ such that 
$\g \not\equiv \0 \bmod (t-x)$. Hence 
$[\Lambda_{N,i}(f,x)]=\LL^{n(N-i)}-\LL^{n(N-i-1)}$ if $N\leq id$.
On the other hand, if
 $N>id$ then 
$ \Lambda_{N,i}(f,x)$ is isomorphic  to the set of polynomials 
 $\g\in \C[t]^n$ of degree $<N-i$ such that 
$f(\g)\equiv 0\bmod{(t-x)^{N-di}}$ and 
$\g \not\equiv \0 \bmod (t-x)$.
In particular the coefficients of $\g$ are unconstrained in degrees 
$N-di,\dots,N-i-1$, and it follows that 
$\Lambda_{N,i}(f,x)\simeq\A^{ni(d-1)}\times  \Lambda_{N-di}^*(f,x)$ 
if 
$N>id$. 
We conclude that 
\begin{align*}
[\Lambda_N(f,x)] 
&= 
(1-\LL^{-n})
\sum_{N/d\leq i<N}
\LL^{n(N-i)}
+
\sum_{0\leq i< N/d} \LL^{n(d-1)i}
 [\Lambda_{N-di}^*(f,x)]\\
 &= 
\LL^{n(N-\lceil \frac{N}{d}\rceil)}+
\sum_{0\leq i< N/d} \LL^{n(d-1)i}
 [\Lambda_{N-di}^*(f,x)],
\end{align*}
from which it follows that 
\begin{equation}\label{eq:sofa}
\LL^{-N(n-1)}[\Lambda_N(f,x)] =  \LL^{N-n
\lceil \frac{N}{d}\rceil}
+
\hspace{-0.2cm}
\sum_{0\leq i< N/d} \LL^{-(n-d)i} 
 \LL^{-(N-di)(n-1)}
 [\Lambda_{N-di}^*(f,x)].
\end{equation}

The proof of Lemma 
\ref{lem:relation'} goes  through and yields
$[\Lambda_{N}^*(f,x)]=[U_{N-1}]$, where we recall that 
$$
U_{N-1}= \{\x = \x_{0} + \x_1 t + \dots + \x_{N-1} t^{N-1}:  f(\x) \equiv 0 \bmod t^{N}, ~\x_0\neq \mathbf{0}\}.
$$
It now follows from \eqref{eq:pavel} that 
$$
 \LL^{-(N+1)(n-1)} [\Lambda_{N+1}^*(f,x)] =
  \LL^{-N(n-1)} [\Lambda_{N}^*(f,x)], 
$$
for any $N\geq 1$, whence
$$
 \LL^{-(N-di)(n-1)}
 [\Lambda_{N-di}^*(f,x)] = \LL^{-(n-1)}
 [\Lambda_{1}^*(f,x)] = \LL^{-(n-1)} (\LL-1) [Z]
$$
in \eqref{eq:sofa}. 
Observing that 
$$
N-n
\left\lceil \frac{N}{d}\right\rceil \leq N-\frac{nN}{d}=
-N(n/d-1),
$$
and taking the limit $N\to \infty$, it follows that 
\begin{align*}
\lim_{N\to \infty }\LL^{-N(n-1)} [\Lambda_N(f,x)]
&= 
 \LL^{-(n-1)} (\LL-1) [Z]
 \sum_{i\geq 0} \LL^{-(n-d)i} \\
&= 
 \LL^{-(n-2)}[Z]  (1-\LL^{-1}) (1-\LL^{-(n-d)})^{-1}.
 \end{align*}
The statement of the lemma follows.
\end{proof}

\begin{proof}[Proof of Theorem \ref{theorem_morspace}]
Assume that $d\geq 3$, $e\geq 1$ and $n>2^d(d-1)$.
It  eases notation if we put
$$
\sigma_\infty(f)= \lim_{N\to \infty} \LL^{-(n-1)N}[\Lambda_{N}(f,\infty)],
$$
where $\Lambda_N(f,\infty)$ is given by \eqref{eq:jet}. 
Then
it follows from  Theorem \ref{theorem_naivespace} that 
\begin{align*}
[M_e] - &(\LL+1)[M_{e-1}] + \LL[M_{e-2}]\\
 &= \LL^{\mu(e)}(\mathfrak{S}(f) \sigma_\infty(f) + R_e) - ( \LL+1)\LL^{\mu(e-1)}(
 \mathfrak{S}(f) \sigma_\infty(f) + R_{e-1})\\
&\quad + \LL^{\mu(e-2) + 1}(\mathfrak{S}(f)\sigma_\infty(f) + R_{e-2}) \\
 &=\LL^{n-1 + e(n-d)}\mathfrak{S}(f)
\sigma_\infty(f)\
  (1 - (\LL + 1)\LL^{-(n-d)} + \LL^{-2(n-d) + 1})  \\
 & \quad + \LL^{n-1 + e(n-d)}(R_e - (\LL + 1)\LL^{-(n-d)} R_{e-1} + \LL^{-2(n-d)+1}R_{e-2}).
\end{align*}
Dividing by $\LL-1$, 
it follows from Lemma \ref{lem:morspace_naivespace}  
that 
\begin{equation}\label{eq:staging-post}
[\Mor_{e}(\Proj^1,Z)] 
= \LL^{\mu(e) - 1}\left( \mathfrak{S}(f)\sigma_\infty(f) \frac{(1-\LL^{-(n-d)})(1-\LL^{-(n-d) + 1})}{1-\LL^{-1}} + S_e\right),
\end{equation}
where
$$
S_e=\frac{R_e - (\LL + 1)\LL^{-(n-d)} R_{e-1} + \LL^{-2(n-d)+1}R_{e-2}}{1-\LL^{-1}}.
$$

We analyse the main term using 
Lemma \ref{lem:*}. 
Denoting by $F_v(T)$ the local factors of the motivic Euler product \eqref{eq:sing_series_euler_prod} of $\mathfrak{S}(f)$,
by multiplicativity of motivic Euler products and the fact that Kapranov's zeta function for $\A^1$ is 
$$
Z_{\A^1}(T)=\prod_{v\in \A^1} (1-T)^{-1} = \frac{1}{(1-\LL T)},
$$
we therefore obtain
\begin{align*}
\mathfrak{S}(f) (1-\LL^{-(n-d) + 1})
&=  \prod_{v\in \A^1}F_v(\LL^{-n}) \prod_{v\in \A^{1}} (1-\LL^{-(n-d)})
\\ 
&= \left(\prod_{x\in \A^1}F_v(\LL^{-n}T)(1-\LL^{-(n-d)}T)\right)_{|T = 1}.
\end{align*}
Also used here is the  compatibility of Euler products with transformations of the form $T\mapsto \LL^{r}T.$
By Remark \ref{rem:faisant} and Lemma \ref{lem:*}, each local factor of the latter motivic Euler product is equal  to  $\LL^{-(n-2)}[Z]  (1-\LL^{-1})$. Combining this with an application of Lemma~\ref{lem:*} to the contribution of $\sigma_{\infty}$ and compatibility of motivic Euler products with finite products, we may write
\begin{align*}
[\Mor_{e}(\Proj^1,Z)] 
= \frac{\LL^{\mu(e) - 1}}{1-\LL^{-1}}\left( 
\prod_{v\in \Proj^1} c_v
+ S_e\right),
\end{align*}
where
$
c_v=(1-\LL^{-1})\LL^{-(n-2)} [Z].
$ 

Turning to  the error term, 
we deduce from Theorem \ref{theorem_naivespace} that 
\begin{align*}
w(R_e) &\leq  4
  - \frac{n-2^{d}(d-1)}{2^{d-2}} \left(1+\left\lfloor \frac{e+1}{2d-2}\right\rfloor\right)\\
  &\leq  4
  - \frac{n-2^{d}(d-1)}{2^{d-1}(d-1)}(e+1),
\end{align*}
for any $e\geq 1$.
By convention $R_{e}=\varnothing$ if $e\leq 0$. 
Hence 
$$
w(S_e)\leq 
\max\left(4 - \nu(e+1), 6 - 2(n-d)-\nu e, 6 - 4(n-d) -\nu(e-1)\right)=4 - \nu(e+1),
$$
where
$$\nu=\frac{n-2^{d}(d-1)}{2^{d-1}(d-1)}>0.$$
The statement of 
Theorem \ref{theorem_morspace}
is now clear.
\end{proof}

\section{The variety of lines}
\label{sec:lines}

Our task in this section is to prove Theorem \ref{thm:lines}.
It follows from Lemma 
\ref{lem:morspace_naivespace} that 
$$
[\Mor_1(\Proj^1, Z)] = \frac{[M_1] - (\LL +1)[M_0]+\LL}{\LL-1},
$$
where $[M_0]=[X]=(\LL-1)[Z]+1$.
The variety of lines $F_1(Z)$ is obtained as a quotient of $\Mor_1(\Proj^1,Z)$ by the action of the automorphism group of $\Proj^1$. The quotient map
$$\Mor_1(\Proj^1,Z)\to F_1(Z)$$
is a Zariski-locally trivial fibration with fibre $\mathrm{PGL}_2$. 
Since $[\mathrm{PGL}_2]=\LL^3-\LL$, it follows that 
\begin{equation}\label{eq:from_theorem}
\begin{split}
[F_1(Z)] 
= \frac{[\Mor_1(\Proj^1, Z)]}{\LL^3-\LL}
&=
\frac{[M_1] -(\LL^2-1)[Z]-1}{(\LL^3-\LL)(\LL-1)}\\
&=
 \frac{[M_1]-1}{(\LL^3-\LL)(\LL-1)}-
\frac{[Z]}{\LL(\LL-1)}.
 \end{split}
 \end{equation}
 
We proceed by summarising what our work says about the class of $M_1$. Let  $d\geq 3$ and  assume that $n>2^d(d-1)$. Rather than 
Theorem \ref{theorem_naivespace}, we  invoke a version in which the truncated singular series appears. This amounts to 
combining 
the remark after 
Proposition~\ref{prop:minor} with the contents of Section \ref{sect:major_arcs_conclusion} in  \eqref{eq:plan}. Thus 
$$[M_1] = \LL^{\mu(1)}\left((1+S_{1}(f) \LL^{-n})  \LL^{-(n-1)}[X] + R_{1}\right),
$$
where $R_1$ is an error term satisfying
$$
w(R_1) \leq  4
  - \frac{n-2^{d}(d-1)}{2^{d-2}}.
$$
Recalling that $\mu(1)=2n-d-1$, by \eqref{eq:mu}, 
we deduce that 
\begin{equation}\label{eq:M1}
[M_1]=
 \LL^{n-d}(\LL-1) [Z] \left(1+S_{1}(f) \LL^{-n} \right)+ \widetilde{R_{1}},
\end{equation}
where
\begin{equation}\label{eq:weight-R1}
\begin{split}
w(\widetilde{R_{1}})
&\leq 
\max\left(2n-2d, 4n-2d+2
  - \frac{n-2^{d}(d-1)}{2^{d-2}}\right)\\&=4n-2d+2
  - \frac{n-2^{d}(d-1)}{2^{d-2}}.
  \end{split}
 \end{equation}

We proceed by computing $S_1(f)$. Note that in this case $B_1\simeq \G_m\times \A^1$ is just the space of pairs $(h,T-x)$ where $h$ is a nonzero constant, and the polynomials $g_1,\dots,g_n\in \Poly_{< 1}$ are just constants. Thus
$$\res\left(\frac{h}{T-x}f(g_1,\dots,g_n)\right)= hf(g_1,\dots,g_n).$$
We therefore get

\begin{align*}S_1(f) & = \left[\Poly_{<1}^n\times B_1: \res\left(\frac{h}{T-x}f(g_1,\dots,g_n)\right)\right]\\
& = [\Poly_{<1}^n\times \G_m\times \A^1, hf(g_1,\dots,g_n)].
\end{align*}
As an element of $\expp_{B_1}$, this is the pullback of 
$$[\Poly_{<1}^n\times \G_m, hf(g_1,\dots,g_n)]\in \expp_{\G_m}$$
via $B_1\to \G_m$ given by the projection $(h, T-x)\mapsto h$.

Applying the orthogonality relation, we see that
\begin{align*} [\Poly_{<1}^n\times \G_m, hf(g_1,\dots,g_n)] &= [\Poly_{<1}^n\times \A^1, hf(g_1,\dots,g_n)] - [\Poly_{<1}^n\times \{0\}, 0]\\
& =   [X]\LL-\LL^n,
\end{align*}
whence
\begin{equation}\label{eq:S1}
S_1(f)=[X]\LL^2-\LL^{n+1}.
\end{equation}

\begin{remark}
Applying \eqref{eq:weight-of-complete-sum} with $m=1$, it follows that our work provides the bound 
$$w([X]\LL^2 - \LL^{n+1}) \leq 2n +4-\frac{n}{2^{d-2}}.$$
That is, for any  hypersurface $X\subset \A^n$ defined by a non-singular form of degree $d$, we have the weight bound 
$$
w([X] - \LL^{n-1}) \leq 2(n-1) - \frac{n-2^{d-1}}{2^{d-2}}.
$$
When working over a finite field $k=\F_q$, this should be compared with the bound 
$$
\#X(\F_q)-q^{n-1}=O_{d,n}(q^{n/2}),
$$
that follows from Deligne's resolution of the Weil conjectures, as explained by Hooley \cite[Theorem 1]{hooley}.
\end{remark}

\begin{proof}[Proof of Theorem \ref{thm:lines}]
It follows from  \eqref{eq:S1}  that 
$
S_1(f)=\LL^2(\LL-1)[Z]+\LL^2-\LL^{n+1}.
$
Combining this with  \eqref{eq:M1}, we deduce that 
\begin{align*}
[M_1]
&=
 \LL^{n-d}(\LL-1) [Z] \left(1+
 \LL^{-n+2}(\LL-1)[Z]+\LL^{-n+2}-\LL \right)+ \widetilde{R_{1}}\\
&= (\LL-1)^2 \left(  \LL^{-d+2} [Z]^2 -\LL^{n-d} [Z] \right)
+ \widetilde{\widetilde{R_{1}}},
\end{align*}
where $\widetilde{\widetilde{R_1}}$ has the same weight as
$\widetilde R_1$
 in  \eqref{eq:weight-R1}.
Inserting this into 
\eqref{eq:from_theorem}, it follows that 
\begin{align*}
\LL^2[F_1(Z)] 
&= 
 \frac{\LL([M_1]-1)}{(\LL^2-1)(\LL-1)}-
\frac{\LL[Z]}{\LL-1}\\
&= 
\frac{
 \LL^{-d+2} [Z]^2-\LL^{n-d} [Z] }
 {1+\LL^{-1}}-
\frac{[Z]}{1-\LL^{-1}} 
-\frac{1}{(\LL^2-1)(1-\LL^{-1})}
+\frac{\LL\widetilde{\widetilde{R_1}}}
{(\LL^2-1)(\LL-1)}.
\end{align*}
The term $(1-\LL^{-1})^{-1}[Z]$ has weight  $2n-4$ 
and the third term has weight $-4$. Both of these are  dominated by the upper bound we have for the weight of the last  term on the right hand side. Dividing both sides by $\LL^2$, we thereby arrive at the statement of Theorem \ref{thm:lines}.
\end{proof}

\begin{proof}[Proof of Corollary \ref{cor:coefficients}]
As described in \cite[Corollary 17.2.2]{Arapura}, for example,  it is well-known that 
 the Hodge--Deligne polynomial of the hypersurface $Z$ takes the shape
$$
\mathrm{HD}(Z) = (uv)^{n-2} ( 1 + (uv)^{-1} + (uv)^{-2} + \cdots ) +g_1(u,v),
$$
where $g_1\in \Z[u,v][[(uv)^{-1}]]$ has only terms of total degree at most $n-2$ in $u$ and $v$. More precisely, it has monomials in $u,v$ of total degree exactly $n-2$ corresponding to the middle Hodge numbers $h^{p,q}(Z)$ with $p+q = n-2$, as well as a diagonal term $-(uv)^i$ for each negative $i$.   Noting that 
$$
(uv)^{n-2} ( 1 + (uv)^{-1} + (uv)^{-2} + \cdots ) =\frac{(uv)^{n-2}}{1-(uv)^{-1}},
$$
we 
thereby deduce that 
$$
\mathrm{HD}([Z]^2) = \frac{(uv)^{2n-4}}{(1-(uv)^{-1})^2} +g_2(u,v),
$$
where $g_2\in \Z[u,v][[(uv)^{-1}]]$ only has terms of total degree at most $2(n-2) + n-2 = 3n-6$. 

Thus, there exists $g_3\in \Z[u,v][[(uv)^{-1}]]$ with terms of total degree at most $3n-2d-6$ such that 
\begin{align*} 
\mathrm{HD}\left(\frac{\LL^{-d}[Z]^2 - \LL^{n-d-2}[Z]}{1 + \LL^{-1}}\right) 
& = \frac{(uv)^{2n-d-5} }{(1-(uv)^{-1})^{2}(1+(uv)^{-1})}+g_3(u,v)\\
& = \frac{(uv)^{2n-d-5} }{(1-(uv)^{-1})(1-(uv)^{-2})}+g_3(u,v)\\
& = (uv)^{2n-d-5} \sum_{k,l\geq 0}(uv)^{-k-2l}+g_3(u,v)\\
& = \sum_{m\geq 0} \left(\left\lfloor\frac{m}{2}\right\rfloor +1 \right)(uv)^{2n-d-5-m}+g_3(u,v).
\end{align*}
Using Theorem \ref{thm:lines}, we see that this should agree with the Hodge--Deligne polynomial of $F_1(Z)$ up to terms $u^pv^q$, with $p+q \leq 4n-2d-6 - \frac{n-2^{d}(d-1)}{2^{d-2}}$, which thereby concludes the proof. 
\end{proof}

\begin{remark}\label{rem:burillo}
We now assume $d=3$, so that we are in the case of a smooth cubic hypersurface $Z\subset \Proj^{n-1}$. Recall Galkin and Shinder's relation 
\eqref{eq:GS}.
We can use work of Burillo to calculate 
$\mathrm{HD}(\Sym^2(Z))$ and so show that  Corollary 
\ref{cor:coefficients} is consistent with it.  
Indeed, Burillo's formula \cite[(2.4)]{Burillo} states that $\mathrm{HD}(\Sym^2(Z))$ is given by the coefficient of $t^2$ in 
$$\prod_{p,q\geq 0} (1-(-1)^{p+q}u^pv^qt)^{(-1)^{p+q+1}h^{p,q}(Z)}.$$
Since  $h^{p,q}(Z) = \delta_{p,q}$, 
for  $p+q> n-2$ and $p+q\leq 2(n-2)$, 
we see that modulo terms of total degree at most $n-2$ in $u,v$, the polynomial $
\mathrm{HD}(\Sym^2(Z))$ is the coefficient of $t^2$ in the expansion
\begin{align*}
\prod_{0\leq p\leq n-2} (1-(uv)^pt)^{-1} =~& 
(1 + (uv)^{n-2} t + (uv)^{2(n-2)}t^2 + \cdots ) 
\\ & \times
(1+ (uv)^{n-3} t + (uv)^{2(n-3)}t^2 + \cdots )\cdots. 
\end{align*}
Thus there exists a polynomial $g_1\in \Z[u,v]$ of degree $\leq n-2$ such that 
$$
\mathrm{HD}(\Sym^2(Z)) = \sum_{n-2 \geq p\geq p'\geq 0} (uv)^{p+p'} +g_1(u,v).
$$
This may be rewritten
\begin{align*}
\mathrm{HD}(\Sym^2(Z)) &= \sum_{ 0\leq q \leq q'} (uv)^{2(n-2) -q-q'} + g_2(u,v)
\\
&=  \sum_{m\geq 0} \left(\left\lfloor\frac{m}{2}\right\rfloor +1 \right)(uv)^{2n-4-m}+g_2(u,v),
\end{align*}
where $g_2(u,v)\in \Z[u,v][[(uv)^{-1}]]$ has terms of total degree at most $n-2$ in $u$ and $v$.

Taking $d=3$ and $n\geq 17$ in Corollary \ref{cor:coefficients}, we compute the Hodge--Deligne polynomial of
$\LL^2[F_1(Z)] + (1 + \LL^{n-2})[Z],
$
modulo degree $\leq  \frac{7}{2}n$ in $u,v$.  This shows that there exists 
$h\in \Z[u,v][[(uv)^{-1}]]$ with terms of degrees $\leq \frac{7}{2}n$ such that 
\begin{align*}
\mathrm{HD}(\LL^2[F_1(Z)] + (1 + \LL^{n-2})[Z])=~&
\sum_{m\geq 0} \left(\left\lfloor\frac{m}{2}\right\rfloor +1 \right)(uv)^{2n-6-m} 
\\
&+ (uv)^{2(n-2)}( 1+ (uv)^{-1} + (uv)^{-2} + \cdots) +h(u,v).
\end{align*}
This is easily seen to agree with our  expression for $\mathrm{HD}(\Sym^2(Z))$ above. 
\end{remark}

\appendix
\section{Geometry of numbers over function fields}\label{sec:gon}

\subsection{Basic facts}

Let $k$ be a field. In this note we discuss basic facts from the geometry of numbers in the setting of the function field $k(t)$.  A non-archimedean absolute value $|\cdot|: k(t)\to \R_{\geq 0}$ 
is given by  taking $|0|=0$ and 
$
|x|=2^{\deg(p)-\deg(q)},
$
if $x=p/q$ for polynomials $p,q\in k[t]$, with $q\neq 0$.
The completion of $k(t)$ with respect to this absolute value is the  
field of Laurent series $K_\infty=k((t^{-1}))$, whose elements take the shape
$$
x=\sum_{-\infty<i\leq M} x_i t^i,
$$
for $x_i\in k$ and $M\in \Z$. 
The absolute value is extended to 
$K_\infty$ by taking 
$$
|x|=
\begin{cases}
0  &\text{ if $x=0$,}\\
2^M &\text{ 
if $x=\sum_{-\infty<i\leq M} x_i t^i\in K_\infty$ with 
$x_M\neq 0$. }
\end{cases}
$$
We  extend this to vectors by setting
$$
|\x|=\max(|x_1|,\dots,|x_n|),
$$
for any $\x\in K_\infty^n$. This provides a distance function on $K_\infty^n$.

Mahler \cite{mahler} initiated an extensive investigation of lattices $\sfl\subset K_\infty^n$, proving analogues of many 
results from  the classical setting of lattices in $\R^n$. 
In this section we  summarise some of the basic  facts that are needed in our application. (In doing so, we  recover work of Lee \cite{lee}, which is concerned with  the special case $k=\F_q$.)

\medskip

A (full rank) {\em lattice} $\mathsf{\Lambda}\subset K_\infty^n$ is defined to be a set of the form
$$
\sfl=\{u_1\x_1+\dots + u_n \x_n: u_1,\dots,u_n\in k[t]\},
$$
where 
 $\x_1,\dots, \x_n\in K_\infty^n$
are 
linearly independent vectors over $K_\infty$.
 This set of vectors is called the {\em basis} of the lattice. 
 Let $\mathbf{M}=(\x_1,\dots,\x_n)$ be the associated 
 $n\times n$ matrix of basis vectors.  The {\em determinant} of the lattice is defined to be
 $$
 \det(\sfl)=|\det \mathbf{M}|.
 $$ 
 This does not depend on the choice of basis for $\sfl$, as proved in \cite[Section~8]{mahler}
 or \cite[Lemma 2]{mahler'}.
 Next,  as in 
 \cite[Section~9]{mahler}, 
  the {\em successive minima} $\sigma_1,\dots,\sigma_n$
 associated to $\sfl$ are defined as follows. 
 We take 
$2^{\sigma_1}$ to be  minimum of
 $|\x_1|$, for non-zero
 $\x_1\in \sfl$.  
 Next, $2^{\sigma_2}$ is the minimum of
 $|\x_2|$, for 
 $\x_2\in \sfl\setminus \Span_{K_\infty}(\x_1)$. 
 One continues in this way, ultimately defining 
$2^{\sigma_n}$ to be the minimum of
 $|\x_n|$, for 
 $\x_n\in \sfl\setminus \Span_{K_\infty}(\x_1,\dots,\x_{n-1})$. 
It is clear from the construction that 
$\sigma_1,\dots,\sigma_n$ are integers satisfying
$
 -\infty< \sigma_1\leq \dots\leq \sigma_n.
$
The analogue of Minkowski's theorem is 
 \begin{equation}\label{eq:det}
 \det(\sfl)=2^{\sigma_1+\dots+\sigma_n},
 \end{equation}
which is established in \cite[Section~9]{mahler}.

For lattices in $\R^n$ there is a wealth of literature around the problem of counting the number of lattice points that are constrained to lie in Euclidean balls of growing radius in $\R^n$. In the function field setting, we are interested in the ``size'' of the set
$$
\left\{\x\in \sfl: |\x|< 2^R\right\}.
$$
This set forms a finitely generated $k$-vector space and we may consider its {\em dimension} as  a $k$-vector space.
The key object of interest  is then the quantity
\begin{equation}\label{eq:def-nu}
\nu(\sfl,R)=\dim \left\{\x\in \sfl: |\x|< 2^R\right\},
\end{equation}
as $R\to \infty$. 
For example, when $\sfl=k[t]^n$ we see that the monomials
$1,t,t^2,\dots, t^{R-1}$ are linearly independent over $k$,  
whence $\nu(k[t]^n,R)=nR.$

Our starting point for the analysis of $\nu(\sfl,R)$ is the identification of a suitable basis for the lattice $\sfl$. The following result is adapted from an argument of Davenport \cite[Lemma 12.3]{dav} (as already 
adapted to the  setting $k=\F_q$ by Lee \cite{lee}).

\begin{lemma}\label{lem:choose}
Let $\sfl\subset K_\infty^n$ be a lattice and let $\sigma_1,\dots,\sigma_n$ be the successive minima of $\sfl$. 
Then, possibly after a linear automorphism of $K_\infty^n$, 
there exists a basis $\x_1,\dots,\x_n$ of $\sfl$ such that, 
for $1\leq i\leq n$, we have 
$$
\x_i=(x_{i,1},\dots,x_{i,i},0,\dots,0)
\quad \text{ and }
\quad |\x_i| =|x_{i,i}|=2^{\sigma_i}.
$$
\end{lemma}

\begin{proof}
Choose vectors $\y_1,\dots,\y_n\in \sfl$  such that, 
for each $i\in \{1,\dots,n\}$, $\y_i\in \sfl$ is chosen to be  
linearly independent from $\y_1,\dots,\y_{i-1}$ and  to have norm
$|\y_i| =2^{\sigma_i}$.
As observed by Mahler in the proof of  \cite[Lemma 1]{mahler'}, it follows 
that $\y_1,\dots,\y_n$ form a basis for $\sfl$.
The basis matrix formed from these column vectors is an $n\times n$ matrix of full rank. 
After possibly composing with a linear automorphism 
of $K_\infty^n$,
we can henceforth assume that these vectors take the shape
$$
\y_i=(y_{i,1},\dots,y_{i,i},0,\dots,0),
$$
for $1\leq i\leq n$, with $y_{1,1}\cdots y_{n,n}\neq 0$. We take $\x_1=\y_1$. Next, we choose $\x_2\in \sfl\cap \Span_{k(t)}( \y_1,\y_2)$
such that 
$$
\Span_{k[t]}(\x_1,\x_2)=
\sfl\cap \Span_{k(t)}( \y_1,\y_2).
$$
It is clear that there exists $q,u_1,u_2\in k[t]$, with $q\neq 0$, such that 
\begin{equation}\label{B.4}
q\x_2=u_1\y_1+u_2\y_2.
\end{equation}
This implies that $qx_{2,2}=u_2y_{2,2}$, from which it follows that 
$u_2\neq 0$, since $q,x_{2,2},y_{2,2}$ are all non-zero.

We claim that the choice of $\x_2$ can be made in such a way that 
$|u_1|,|u_2|\leq |q|$, which   implies that 
$$
|\x_2|\leq |q|^{-1}\max\left\{ |u_1| |\y_1|,  |u_2| |\y_2|\right\}\leq |\y_2|=2^{\sigma_2},
$$
by the ultrametric inequality.  
To prove the claim, we note that 
$$
\y_2\in  \Span_{k[t]}(\x_1,\x_2)=\Span_{k[t]}(\y_1,\x_2).
$$ 
Hence, there exist 
$v_1,v_2\in k[t]$ such that 
$
\y_2=v_1\y_1+v_2\x_2.
$
But then \eqref{B.4} implies that 
$$
(q-u_2v_2)\x_2=(u_1+u_2v_1)\y_1,
$$
whence 
$(q-u_2v_2)x_{2,1}=(u_1+u_2v_1)y_{1,1}$ and 
$(q-u_2v_2)x_{2,2}=0$, since $y_{1,2}=0$. Thus it follows that 
$q=u_2v_2$ and $u_1=-u_2v_1$. Once inserted into 
\eqref{B.4} and recalling that $u_2\neq 0$, we obtain
$$
v_2 \x_2=-v_1\y_1+\y_2.
$$
Moreover, $v_2\neq 0$. Writing $-v_1=w_1v_2+r_1$, for $w_1,r_1\in k[t]$ such that 
$|r_1|<|v_2|$, we therefore  obtain
$$
v_2(\x_2-w_1\y_1)=r_1\y_1+\y_2.
$$ 
The claim follows on  redefining 
$\x_2-w_1\y_1=\x_2-w_1\x_1$ to be 
$\x_2$.

In a similar way, for $3\leq i\leq n$, 
$\x_i\in 
 \sfl\cap \Span_{k(t)}( \y_1,\dots,\y_i)$ can be chosen so that
$$
\Span_{k[t]}(\x_1,\dots,\x_i)=
\sfl\cap \Span_{k(t)}( \y_1,\dots ,\y_i),
$$
with 
$
|\x_i|\leq 2^{\sigma_i}.
$
Proceeding in this way, we  have 
$$
|x_{i,i}|\leq |\x_i|\leq 2^{\sigma_i},
$$
for $1\leq i\leq n$.

It remains to prove the lower bound 
$$
|x_{i,i}|\geq 2^{\sigma_i},
$$
for $1\leq i\leq n$.
The vectors $\x_1,\dots,\x_n$ form a basis for 
$\sfl$. Thus 
$$
\det(\sfl)=|\det(\x_1,\dots,\x_n)|=|x_{1,1}\dots x_{n,n}|.
$$
It follows that 
$$
\det(\sfl)\leq 2^{\sigma_1+\dots+\sigma_{i-1}} |x_{i,i}| 
2^{\sigma_{i+1}+\dots+\sigma_{n}},
$$
for any $i\in \{1,\dots,n\}$. Appealing to \eqref{eq:det}, we deduce that 
 $|x_{i,i}| \geq 2^{\sigma_i}$, as required.
\end{proof}

We are now ready to prove our key lattice point counting result. 

\begin{lemma}\label{lem:counting}
Let $\sfl\subset K_\infty^n$ be a lattice and let $\sigma_1,\dots,\sigma_n$ be the successive minima of $\sfl$.
Then for any  $R\in \Z_{> 0}$, we have 
$$
\nu(\sfl,R)=\sum_{i=1}^n \max\{0,R-\sigma_i\}.
$$
\end{lemma}

\begin{proof}
Let $\sfl\subset K_\infty^n$ be a lattice and let $\sigma_1,\dots,\sigma_n$ be the successive minima of $\sfl$. After composing with  
a 
linear automorphism of $K_\infty^n$, 
we may 
choose a basis 
$\x_1,\dots,\x_n$ for  $\sfl$ as in Lemma~\ref{lem:choose}.
Thus 
$$
\x_i=(x_{i,1},\dots,x_{i,i},0,\dots,0), \quad \text{for $1\leq i\leq n$},
$$
with 
$
|\x_i| =|x_{i,i}|=2^{\sigma_i},
$
for $1\leq i\leq n$. Any $\x\in \sfl$ has the form
$$
\x=u_1 \x_1+\dots +u_n \x_n,
$$
for $u_1,\dots,u_n\in k[t]$. 
The condition $|\x|< 2^R$ is then equivalent to the system of inequalities
\begin{align*}
|u_1+
u_2x_{2,1}/x_{1,1}+\dots+
u_{n}x_{n,1}/x_{1,1}|&<2^{R-\sigma_1}, \\
&\hspace{0.2cm}\vdots  \\
|u_{n-1}+ u_nx_{n,n-1}/x_{n,n}|
&<2^{R-\sigma_{n-1}}, \\
| u_n| &<2^{R-\sigma_n}.
\end{align*}
The set of polynomials 
$u_1,\dots,u_n\in k[t]$ satisfying these constraints clearly defines a $k$-vector space of dimension 
$\sum_{i=1}^n \max\{0,R-\sigma_i\}$,
as claimed in the statement of the lemma.
\end{proof}

Let $\sfl\subset K_\infty^n$ be a lattice with basis matrix $\mathbf{M}$. 
The   {\em dual lattice} $\sfl^*$ is the lattice with basis matrix 
$\mathbf{M}^{\text{adj}}$, obtained by taking the adjoint matrix of $\mathbf{M}$.
We clearly have 
$\det(\sfl^*)=1/\det(\sfl)$.  In fact the successive minima of $\sfl$ share a close correspondence with the successive minima $\sigma_1^*,\dots,\sigma_n^*$ of the dual lattice. 
The relation
\begin{equation}\label{eq:*-min}
\sigma_i=-\sigma_{n-i+1}^*, \quad  \text{for $1\leq i\leq n$},
\end{equation}
is established in \cite[Section~10]{mahler}.

\subsection{Davenport's shrinking lemma}\label{sec:shrink}

Given any $x=\sum_{-\infty<i\leq M} x_it^i\in K_\infty$ we define the {\em distance to the nearest integer} function to be
$$
\|x\|=\left|\sum_{-\infty<i\leq -1} x_it^i\right|.
$$
Associated to any $a,b\in \Z$ such that $b>0$,  and any symmetric $n \times n$ matrix $\mathbf{U}$ with entries in $K_\infty$,
we set $\sfl_{a,b}(\mathbf{U})\subset K_\infty^{2n}$ for the  lattice with underlying basis matrix
$$
\mathbf{M}_{a,b}(\mathbf{U}) = \begin{pmatrix} t^{-a} \mathbf{I}_n & \mathbf{0} \\ t^b \mathbf{U} & t^b \mathbf{I}_n \end{pmatrix} ,
$$
where $\mathbf{I}_n$ is the $n\times n$ identity matrix. 

\begin{remark}\label{rem:5ways}
Recalling \eqref{eq:def-nu}, we see that 
$\nu(\sfl_{a,b}(\mathbf{U}),0)$ is by definition the dimension of the $k$-vector space of pairs $(\x,\y)\in (k[t]^n)^2$ such that
$$|t^{-a}\x|< 1,\ \ \ \text{and}\ \ \ |t^{b}\mathbf{U}\x + t^b\y|< 1,$$
i.e. such that
$$|\x|< 2^{a},\ \ \ \text{and}\ \ \ |\mathbf{U}\x + \y|< 2^{-b}.$$
Now, given that $b>0$, for each $\x\in k[t]^n$ there will be at most one $\y\in k[t]^n$ satisfying $|\mathbf{U}\x + \y|< 2^{-b}$, namely the vector whose coordinates are given by the negatives of the integer parts of the respective coordinates of the vector $\mathbf{U}\x$, and the latter will work exactly when $\| \mathbf{U} {\bf x} \| < 2^{-b}$. Thus we may conclude that $\nu(\sfl_{a,b}(\mathbf{U}),0)$ is the dimension of the $k$-vector space of vectors 
$\x\in k[t]^n$ for which $| {\bf x} | < 2^a$ and $\| \mathbf{U} {\bf x} \| < 2^{-b}$.  
\end{remark}

The following version of Davenport's {\em shrinking lemma} generalises \cite[Lemma 5.3]{BSfree} to arbitrary function fields.

\begin{lemma}\label{new-geometry} 
Let $\mathbf{U}$ be a symmetric $n \times n$ matrix with entries in $K_\infty$. 
Let $a, b,s\in \mathbb{Z}$ such that $b >0$ and $s \geq 0$.
Then 
$$
\nu(\sfl_{a,b}(\mathbf{U}),0)\leq 
\nu(\sfl_{a-s,b+s}(\mathbf{U}),-s)
+
ns + n  \max\left\{ \left\lfloor \frac{a-b}{2} \right\rfloor, 0 \right\}.
$$
\end{lemma}

\begin{proof}
The inequality  is trivial when $a\leq 0$ since then the left hand side is $0$. Hence we may assume that 
$a>0$. 
Let $-\infty<\sigma_1\leq \dots \leq \sigma_{2n}$ be the successive minima of 
$\sfl_{a,b}(\mathbf{U})$ and let $-\infty<\sigma_1^*\leq \dots\leq \sigma_{2n}^*$ be the successive minima of the 
 dual lattice
$\sfl_{a,b}(\mathbf{U})^*\subset K_\infty^{2n}$, with  underlying matrix
\[ 
\mathbf{M}_{a,b}(\mathbf{U})^{\text{adj}} = \begin{pmatrix} t^{a} \mathbf{I}_n & -t^a \mathbf{U} \\ 
\mathbf{0} &
t^{-b} \mathbf{I}_n \end{pmatrix} .
\] 
We note that  
\begin{align*}
t^{b-a} \mathbf{M}_{a,b}(\mathbf{U})^{\text{adj}}
&=
\begin{pmatrix} t^{b} \mathbf{I}_n & -t^b \mathbf{U} \\ 
\mathbf{0} &
t^{-a} \mathbf{I}_n \end{pmatrix} 
=
\begin{pmatrix} \mathbf{0} & \mathbf{I}_n \\ - \mathbf{I}_n & \mathbf{0} \end{pmatrix} 
\mathbf{M}_{a,b}(\mathbf{U})
\begin{pmatrix} \mathbf{0} & \mathbf{I}_n \\ - \mathbf{I}_n & \mathbf{0} \end{pmatrix} ^{-1} .
\end{align*}
It follows that the lattice
with underlying basis matrix 
$t^{b-a} \mathbf{M}_{a,b}(\mathbf{U})^{\text{adj}}$
 is equal to the 
one with basis matrix 
$\mathbf{M}_{a,b}(\mathbf{U})$, up to 
left and right multiplication by a matrix in $\mathrm{SL}_{2n}(k)$.
Hence the associated lattices  share the same successive minima, whence
$
2^{\sigma_i}=2^{b-a+\sigma_i^*},
$
for $1\leq i\leq 2n$. Appealing to 
\eqref{eq:*-min}, it follows that 
$
\sigma_i+\sigma_{2n-i+1}=b-a,
$
for $1\leq i\leq 2n$.
Taking $i=n+1$, we deduce that 
\begin{equation}\label{eq:mint}
 \sigma_{n+1}\geq \left\lceil \frac{b-a}{2} \right\rceil.
\end{equation}

We now apply Lemma \ref{lem:counting} to deduce that 
$$
\nu(\sfl_{a,b}(\mathbf{U}),0)=
\sum_{i=1}^{2n} \max\{0,-\sigma_i\}
$$
and 
$$
\nu(\sfl_{a-s,b+s}(\mathbf{U}),-s)
=\sum_{i=1}^{2n} \max\{0,-s-\sigma_i\}.
$$
For $1\leq i\leq n$, it is clear that 
$$
\max\{0,-\sigma_i\}- \max\{0,-s-\sigma_i\}\leq s.
$$
Moreover, for $n+1\leq i\leq 2n$ we have 
$$
\max\{0,-\sigma_i\}- \max\{0,-s-\sigma_i\}\leq 
\max\left\{\left \lfloor \frac{a-b}{2} \right\rfloor, 0\right\},
$$
by \eqref{eq:mint}. The statement of the lemma is now clear.
\end{proof}

\end{document}